\documentclass[letterpaper,11pt]{amsart}

\usepackage{amsmath}
\usepackage{amsthm}
\usepackage{mathrsfs}
\usepackage{amssymb,amscd}
\usepackage[all]{xy}
\usepackage{mathtools}
\usepackage{fullpage}
\usepackage{comment}
\usepackage{tikz-cd}

\usepackage{xcolor}
\usepackage[colorlinks=true,linkcolor=blue,breaklinks,pagebackref]{hyperref}

\usepackage[lite,abbrev,msc-links,alphabetic]{amsrefs}

\newcommand{\fa}{\mathfrak{a}}

\newcommand{\fg}{\mathfrak{g}}

\newcommand{\fX}{\mathfrak{X}}

\newcommand{\bA}{\mathbb{A}}

\newcommand{\bH}{\mathbb{H}}

\newcommand{\G}{\mathbb{G}}
\newcommand{\bG}{\mathbb{G}}

\newcommand{\C}{\mathbb{C}}
\newcommand{\Q}{\mathbb{Q}}
\newcommand{\R}{\mathbb{R}}
\newcommand{\bW}{\mathbb{W}}
\newcommand{\bV}{\mathbb{V}}

\newcommand{\Z}{\mathbb{Z}}

\newcommand{\cA}{\mathcal{A}}
\newcommand{\cB}{\mathcal{B}}
\newcommand{\cC}{\mathcal{C}}

\newcommand{\cF}{\mathcal{F}}
\newcommand{\cG}{\mathcal{G}}
\newcommand{\cH}{\mathcal{H}}

\newcommand{\cL}{\mathcal{L}}
\newcommand{\cM}{\mathcal{M}}

\newcommand{\cO}{\mathcal{O}}
\newcommand{\cP}{\mathcal{P}}
\newcommand{\cS}{\mathcal{S}}
\newcommand{\cT}{\mathcal{T}}
\newcommand{\cU}{\mathcal{U}}

\newcommand{\cW}{\mathcal{W}}

\newcommand{\cZ}{\mathcal{Z}}

\newcommand{\rd}{\mathrm{d}}

\newcommand{\sfc}{\mathsf{c}}

\newcommand{\tS}{\mathtt{S}}
\newcommand{\tT}{\mathtt{T}}

\newcommand{\tp}[1]{\prescript{t}{}{#1}}

 \DeclareMathOperator{\vol}{vol}

\DeclareMathOperator{\Gal}{Gal}

\providecommand{\aabs}[1]{\lVert#1\rVert}

\newcommand{\valP}[1]{\left|#1\right|}

\DeclareMathOperator{\Tr}{Tr}

\DeclareMathOperator{\rR}{R}

\DeclareMathOperator{\GL}{GL}

\DeclareMathOperator{\U}{U}

\DeclareMathOperator{\Res}{Res}

\DeclareMathOperator{\Hom}{Hom}

\DeclareMathOperator{\disc}{disc}

\DeclareMathOperator{\BC}{BC}

\DeclareMathOperator{\As}{As}

\DeclareMathOperator{\Temp}{Temp}

\DeclareMathOperator{\Pet}{Pet}
\DeclareMathOperator{\cusp}{cusp}

\DeclareMathOperator{\otimeshat}{\widehat{\otimes}}

\DeclareFontFamily{U}{mathx}{\hyphenchar\font45}
\DeclareFontShape{U}{mathx}{m}{n}{
      <5> <6> <7> <8> <9> <10>
      <10.95> <12> <14.4> <17.28> <20.74> <24.88>
      mathx10
      }{}
\DeclareSymbolFont{mathx}{U}{mathx}{m}{n}
\DeclareMathAccent{\widecheck}{\mathalpha}{mathx}{"71}

\makeatletter
\def\Ddots{\mathinner{\mkern1mu\raise\p@
\vbox{\kern7\p@\hbox{.}}\mkern2mu
\raise4\p@\hbox{.}\mkern2mu\raise7\p@\hbox{.}\mkern1mu}}
\makeatother

\theoremstyle{definition}
\newtheorem{definition}{Definition}[section]

\theoremstyle{plain}
\newtheorem{theorem}[definition]{Theorem}
\newtheorem{prop}[definition]{Proposition}
\newtheorem{lemma}[definition]{Lemma}
\newtheorem{coro}[definition]{Corollary}

\theoremstyle{remark}
\newtheorem{remark}[definition]{Remark}

\numberwithin{equation}{section}

\begin{document}

\title{The global Gan--Gross--Prasad conjecture for Fourier--Jacobi periods on unitary groups \\ III: Proof of the main theorems}

\author{Paul Boisseau}

\author{Weixiao Lu}

\author{Hang Xue}

\address{Paul Boisseau, Max Planck Institute for Mathematics, Vivatsgasse 7, 53111 Bonn, Germany}
\email{boisseau@mpim-bonn.mpg.de}

\address{Weixiao Lu, Aix Marseille Univ, CNRS, I2M, Marseille, 13009, France}
\email{weixiao.lu@univ-amu.fr}

\address{Hang Xue, Department of Mathematics, The University of Arizona, Tucson, AZ, 85721, USA}

\email{xuehang@arizona.edu}

\date{\today}

\begin{abstract} 
This is the third and the last of a series of three papers where we prove the Gan--Gross--Prasad conjecture for Fourier--Jacobi periods on unitary groups and an Ichino--Ikeda type refinement. Our strategy is based on the comparison of relative trace formulae formulated by Liu. In this paper, we compute the spectral expansions of these formulae and end the proof of the conjectures via a reduction to the corank zero case.
\end{abstract}

\setcounter{tocdepth}{1}

\maketitle

\tableofcontents

\section{Introduction}

In the 1990s, Gross and Prasad~\cites{GP1,GP2} formulated some conjectures on the
restriction problems for orthogonal groups. These conjectures were later
extended to all classical groups by Gan, Gross and Prasad in their
book~\cite{GGP}. Ichino and Ikeda~\cite{II} gave a refinement of the
conjecture of Gross and Prasad, and this refinement was further extended to
other cases~\cites{NHarris,Liu2,Xue2,Xue7}. They are usually
referred to as the Gan--Gross--Prasad (GGP) conjectures and have attracted
a significant amount of research. They describe the relation
between certain period integrals of automorphic forms on classical groups and
the central values of some $L$-functions. 

There are two kinds of period integrals in the GGP conjecture: Bessel periods
and Fourier--Jacobi periods. Fourier--Jacobi periods usually involve a theta
function while Bessel periods do not. Specifying to the case of unitary
groups, Bessel periods are on the unitary groups $\U(n) \times \U(m)$ where
$n-m$ is odd, while the Fourier--Jacobi case is when $n-m$ is even. The
cases of Bessel periods on unitary groups, together with their refinements,
are now completely settled through the work of many people,
cf.~\cites{JR,Yun,Zhang1,Zhang2,Xue3,BP,BP1,BPLZZ,BPCZ,BPC22} for an
incomplete list. The proof follows the approach of Jacquet and Rallis
via the comparison of relative trace formulae.

Inspired by the work of Jacquet and Rallis, Liu~\cite{Liu} proposed a
relative trace formula approach towards the Fourier--Jacobi case of the conjecture. Following
this approach, the third author proved in~\cites{Xue1,Xue2} some cases of the
GGP conjecture for $\U(n) \times \U(n)$ under various local conditions. 

The goal of this series of three papers (\cite{BLX1}, \cite{BLX2} and the current text) is to work out the comparison of these relative trace formulae in general, and obtain the GGP and Ichino--Ikeda conjectures for arbitrary Fourier--Jacobi periods
on unitary groups. This is the last paper in this series. Its objective is to compare the two spectral expansions of Liu's relative trace formulae and to prove the conjectures. We now precisely state our main results.

\subsection{The GGP conjecture for Fourier--Jacobi periods}
\label{subsec:GGP_intro}

We first recall the construction of the Fourier--Jacobi period from \cite{GGP}. They are integrals of automorphic forms along a Fourier--Jacobi subgroup against a theta series built using a Heisenberg--Weil representation.

\subsubsection{Fourier--Jacobi groups, Heisenberg--Weil representations}
\label{subsubsec:HW_reps}

Let $E/F$ be a quadratic extension of number fields, and write $\bA$ (resp. $\bA_E$) for the ring of adeles of $F$ (resp. of $E$). If $\bG$ is an algebraic group over $F$, we write $[\bG]$ for the automorphic quotient $\bG(F) \backslash \bG(\bA)$.

We fix $n$ and $m$ two integers such that $n=m+2r$ with $r$ a non-negative integer. Let $W \subset V$ be two non-degenerate skew-Hermitian spaces of respective dimensions $m$ and $n$. We assume that $W^\perp$ is split, i.e. that $W^\perp=X \oplus X^*$ where $X$ is an isotropic subspace of dimension $r$. Let $\U(V)$ and $\U(W)$ be the $F$-algebraic groups of $E$-linear unitary transformations of $V$ and $W$. Let $\{0\} \subset X_1 \subset \hdots \subset X_r=X$ be a complete flag of $X$, and let $U_r$ be the unipotent radical of the parabolic subgroup of $\U(V)$ stabilizing the flag $\{0\} \subset X_1 \subset \hdots \subset X \subset X \oplus W \subset V$. Note that its Levi subgroup stabilizing $W$ is isomorphic to $(\Res_{E/F}\GL_1)^r \times \U(W)$. We define 
\begin{equation*}
    \cU_W:=\U(V) \times \U(W), \quad \cH_W:=\U(W) \ltimes U_r,
\end{equation*}
where the latter is embedded in the former by taking the product of the natural inclusion and the projection on $\U(W)$. The group $\cH_W$ is the \emph{Fourier--Jacobi group}.

Let $q(\cdot,\cdot)_W$ be the skew-Hermitian form on $W$, and let $\bW$ be the $F$-symplectic space with underlying $F$-vector space $W$ and symplectic form $\Tr_{E/F} \circ q(\cdot,\cdot)_W$. We take a polarization $\bW=Y \oplus Y^\vee$, so that $Y^\vee$ is a $F$-vector space of dimension $m$. We have the Heisenberg group $S(W)$ associated to $\bW$, whose points are $\bW \times F$ and whose group law is described in \eqref{eq:Heisenberg_group} below. If we write $U_{r-1}$ for the unipotent radical of the parabolic subgroup of $\U(V)$ stabilizing the flag $\{0\} \subset X_1 \subset \hdots \subset X_{r-1} \subset X \oplus W \subset V$, we have the alternative description $\cH_W=\U(W) \ltimes S(W) \ltimes U_{r-1}$. 

Let $\psi$ be a non-zero automorphic character of $F \backslash \bA_F$, and let $\mu$ be an automorphic character of $E^\times \backslash \bA_E^\times$ whose restriction to $\bA_F^\times$ is $\eta_{E/F}$ the quadratic character associated to $E/F$ by global class field theory. They decompose as $\psi=\otimes_v \psi_v$ and $\mu=\otimes_v \mu_v$. For each place $v$ of $F$, by the Stone--Von-Neumann theorem, there exists a unique irreducible smooth infinite-dimensional representation of $S(W)(F_v)$ with central character $\psi_v^{-1}$, which we denote by $\rho_{v}^\vee$. It is realized on the local Schwartz space $\cS(Y^\vee(F_v))$. Moreover, $\mu^{-1}_v$ induces a splitting of the $S^1$-metaplectic cover $\mathrm{Mp}(\bW)_v \to \mathrm{Sp}(\bW)(F_v)$ over $\U(W)(F_v)$ which yields a Weil representation $\omega^\vee_{v}$ of $\U(W)(F_v)$ associated to $(\psi_v^{-1},\mu_v^{-1})$. We refer to \cite{MVW} for proofs of these statements. Finally, we take the $F$-morphism $\lambda : U_{r-1} \to \Res_{E/F} \bG_a$ defined in \eqref{eq:lambda_N}. It is stable under the action of $\U(W) \ltimes S(W)$. Consider $\psi_U:=\psi \circ \Tr_{E/F} \circ \lambda$ which is an automorphic character of $U_{r-1}$. We can now define $\nu_v^\vee$ the \emph{local Heisenberg--Weil} representation of $\cH_W(F_v)$ by the rule 
\begin{equation*}
    \nu_v^\vee(g h u)=\omega_v^\vee(g) \rho_v^\vee(h) \overline{\psi}_{U,v}(u), \quad g \in \U(W)(F_v), \quad h \in S(W)(F_v), \quad u \in U_{r-1}(F_v).
\end{equation*}
It is realized on $\cS(Y^\vee(F_v))$. As recalled in \cite{GGP}*{Section~22}, the \emph{global Heisenberg--Weil representation} $\nu^\vee:=\otimes'_v \nu_v^\vee$ realized on $\cS(Y^\vee(\bA_F))$ is an automorphic representation of $\cH_W(\bA_F)$ via the Theta series
\begin{equation*}
    \theta_{\cH_W}^\vee(h,\phi)=\sum_{y \in Y^\vee(F)} (\nu^\vee(h)\phi)(y), \quad h \in \cH_W(\bA_F), \quad \phi \in \nu^\vee.
\end{equation*}

\subsubsection{Global Fourier--Jacobi periods}

We give $\cH_W(\bA)$ the Tamagawa measure from \S\ref{subsec:tamagawa_measure_new}. Let $\sigma$ be a cuspidal automorphic representation of $\cU_W(\bA)$. The \emph{Fourier--Jacobi period} is the absolutely convergent integral
    \begin{equation}
        \label{eq:global_FJ_intro}
         \mathcal{P}_{\cH_W}(\varphi, \phi) :=
    \int_{[\cH_W]} \varphi(h) \theta^\vee_{\cH_W}(h,\phi) \rd h, \quad \varphi \in \sigma, \quad \phi \in \nu^\vee.
    \end{equation}

\subsubsection{Base-change}
For any $k \geq 1$, we write $G_k$ for the restriction of scalars $\Res_{E/F} \GL_k$. We say that $\sigma$ admits a \emph{weak base-change} if there exists a discrete automorphic representation $\BC(\sigma)$ of $G_n \times G_m$ such that, at almost all split places $v$, $\BC(\sigma)_v$ is the split base change of $\sigma_v$. By \cite{Ram}, such a representation is unique.

There is also a notion of \emph{strong base-change}. By local Langlands functoriality, for all place $v$ the representation $\sigma_v$ admits a base-change $\mathrm{BC}(\sigma_v)$ which is a smooth irreducible representation of $G_n(F_v) \times G_m(F_v)$. If we assume that all the $\mathrm{BC}(\sigma_v)$ are generic, then by \cite{Mok} and \cite{KMSW}, $\mathrm{BC}(\sigma)$ always exists and is a strong base-change, i.e. it satisfies $\mathrm{BC}(\sigma)_v=\mathrm{BC}(\sigma_v)$ for all $v$.

Write $\BC(\sigma)=\BC(\sigma)_n \boxtimes \BC(\sigma)_m$. In the generic case, for $k \in \{n,m\}$ \cite{Mok} and \cite{KMSW} also tell us that $\BC(\sigma)_k$  can be written as a parabolic induction $\pi_{k,1} \times \hdots \times \pi_{k,r_k}$ where the $\pi_{k,i}$ are cuspidal self-dual automorphic representations of some $G_l$'s, mutually non-isomorphic, and such that the Asai $L$-functions $L(s,\pi_{k,i},\As^{(-1)^{k+1}})$ have a pole at $s=1$. The $\pi_{k,i}$ are unique up to permutation, and we set $ S_\sigma=(\Z / 2 \Z)^{r_n} \times (\Z / 2 \Z)^{r_m}$.

Conversely, any automorphic representation $\Pi$ of $G_n \times G_m$ that can be written as an induction with the above desiderata is called a \emph{discrete Hermitian Arthur parameter}. 

\subsubsection{The GGP conjecture}

The first main result of this paper is the Gan--Gross--Prasad conjecture for cuspidal automorphic representation with generic base-change as stated in \cite{GGP}. The statement we give is independent of \cite{Mok} and \cite{KMSW} as it only relies on the notion of weak base-change.

\begin{theorem}
    \label{thm:GGP}
Let $\Pi$ be a discrete Hermitian Arthur parameter of $G_n \times G_m$. Then the following assertions are equivalent:
    \begin{enumerate}
        \item the complete Rankin--Selberg $L$-function satisfies $L(\frac{1}{2}, \Pi \otimes \overline{\mu}) \neq 0$;
        \item there exist two non-degenerate skew-Hermitian spaces $W\subset V$ of respective dimensions $m$ and $n$ and $\sigma$ a cuspidal automorphic representation of $\cG_{W}$ such that $W^\perp$ is split, the weak base-change of $\sigma$ is $\Pi$, and $\cP_{\cH_{W}}$ does not vanish identically on $\sigma \otimes \nu^\vee$.
    \end{enumerate}
\end{theorem}

\begin{remark}
By the local Gan--Gross--Prasad conjecture from \cite{GI2} and \cite{Xue6}, the spaces $W \subset V$ and the representation $\sigma$ given by (2) in the theorem are unique if they exist.
\end{remark}

\subsection{Local Fourier--Jacobi periods}

The second main result of this paper is of local nature. We keep the notation of \S\ref{subsec:GGP_intro}, and henceforth take $v$ a place of $F$ (which is allowed to be split in $E$). Let $\sigma_v$ be a smooth irreducible representation of $\cU_W(F_v)$. We can consider the space of local Fourier--Jacobi functionals $\Hom_{\cH_W(F_v)}(\sigma_v \otimes \nu^\vee_v,\C)$. By \cite{GGP} and \cite{LiuSun}, it is of dimension at most $1$.

If we now assume that $\sigma_v$ is tempered, we can produce an element in $\Hom_{\cH_W(F_v)}(\sigma_v \otimes \nu^\vee_v,\C)$ by means of integration. Let $\langle \cdot,\cdot\rangle_v$ be an invariant inner-product on $\sigma_v$, and write $\langle \cdot, \cdot \rangle_{L^2,v}$ for the invariant inner-product on $\nu_v^\vee$ obtained by integration along $Y^\vee(F_v)$. Then we define the \emph{local Fourier--Jacobi period} 
\begin{equation*}
    \cP_{\cH_W,v}(\varphi_v,\phi_v):=\int_{\cH_W(F_v)}^{\mathrm{reg}} \langle \sigma_v(h_v) \varphi_v,\varphi_v \rangle_v \langle \nu_v^\vee(h_v) \phi_v,\phi_v \rangle_{L^2,v} \rd h_v, \quad \varphi_v \in \sigma_v, \quad \phi_v \in \nu_v^\vee.
\end{equation*}
This integral is not absolutely convergent if $r \geq 1$ and needs to be suitably regularized (see \S\ref{subsec:local_FJ_periods_defi}). We prove that \emph{tempered intertwining} holds for Fourier--Jacobi functionals. This was previously known in the $r=0$ case by \cite{Xue2}.

\begin{theorem}
    \label{thm:temp_int}
    Assume that $\sigma_v$ is tempered. Then the space $\Hom_{\cH_W(F_v)}(\sigma_v \otimes \nu^\vee_v,\C)$ is non-zero if and only if the sesquilinear form $\cP_{\cH_W,v}$ is non-zero on $\sigma_v \otimes \nu^\vee_v$.
\end{theorem}

\subsection{The global Ichino--Ikeda conjecture}
\label{subsec:global_II}

We go back to the global setting and let $\sigma$ be a cuspidal automorphic representation of $\cU_W$. We henceforth assume that $\sigma_v$ is tempered at all places $v$ so that $\BC(\sigma)$ exists and is generic. Note that under the Ramanujan conjecture, the genericity of $\BC(\sigma)$ should be equivalent to the temperedness of the $\sigma_v$'s. Our last main result is a factorization of $\valP{\cP_{\cH_W}}^2$ in terms of $L$-functions and of the local periods $\cP_{\cH_W,v}$.

We equip the automorphic quotients $[\cU_W]$ and $[\cH_W]$ with the invariant Tamagawa measure. We have the Petersson inner-product $\langle \cdot,\cdot \rangle$ on $\sigma$, and the inner-product $\langle \cdot,\cdot,\rangle_{L^2}$ on $\nu^\vee$ given by integration along $Y^\vee(\bA_F)$ against the Tamagawa measure. We take a factorization $\langle \cdot,\cdot \rangle=\prod_v \langle \cdot,\cdot \rangle_v$ and we have $\langle \cdot,\cdot \rangle_{L^2}=\prod_v \langle \cdot,\cdot \rangle_{L^2,v}$. We also take a factorization $\rd h =\prod_v \rd h_v$ of the Tamagawa measure on $\cH_W(\bA)$.

We consider the product of completed $L$-functions
\begin{equation*}
    \cL(s,\sigma)=\prod_{i=1}^n L(i+s-\frac{1}{2},\eta_{E/F}^i)\frac{L(s,\mathrm{BC}(\sigma) \otimes \overline{\mu})}{L(s+\frac{1}{2},\mathrm{BC}(\sigma),\mathrm{As}^{n,m})},
\end{equation*}
with $\mathrm{As}^{n,m}=\mathrm{As}^{(-1)^n} \boxtimes \mathrm{As}^{(-1)^m}$. The function $\cL(s,\sigma)$ is regular at $1/2$ (see Remark~\ref{rk:L_pole} for a more general statement). For each place $v$ we denote by $\cL(s,\sigma_v)$ the corresponding local factor. We define the \emph{normalized local Fourier--Jacobi period} 
\begin{equation*}
    \cP^\sharp_{\cH_W,v}(\varphi_v,\phi_v):=\cL(\frac{1}{2},\sigma_v)^{-1}  \cP_{\cH_W,v}(\varphi_v,\phi_v).
\end{equation*}
It follows from the unramified Ichino--Ikeda conjecture of \cite{Boi} that if
$\varphi=\otimes_v \varphi_v$ and $\phi=\otimes_v \phi_v$ are
factorizable vectors, then $\cP^\sharp_{\cH_W,v}(\varphi_v,\phi_v)=1$ for almost all $v$. Our second main result is the following factorization expression. It is the Fourier--Jacobi version of a conjecture of Ichino and Ikeda in \cite{II} for Bessel periods on orthogonal groups, and of Harris in \cite{NHarris} for Bessel periods on unitary groups.

\begin{theorem}
    \label{thm:II}
    Let $\sigma$ be as above. For every factorizable vectors $\varphi =
\otimes_v \varphi_v \in \sigma$ and $\phi=\otimes_v \phi_v \in
\nu^\vee$ we have
    \begin{equation*}
        |\cP_{\cH_W}(\varphi,\phi)|^2=|S_\sigma|^{-1} \cL(\frac{1}{2},\sigma) \prod_v \cP^\sharp_{\cH_W,v}(\varphi_v,\phi_v).
    \end{equation*}
\end{theorem}

\subsection{The Gan--Gross--Prasad conjecture for generic Eisenstein series}

In our approach, we first prove generalizations of Theorems \ref{thm:GGP} and \ref{thm:II} in the corank-zero situation (i.e. $V=W$) for a certain class of generic Eisenstein series. This implies the result in the general case by unfolding. We now explain the main steps of this strategy.

\subsubsection{Reduction to corank-zero}
\label{subsubsec:reduc_zero}

The first part of the argument is a reduction to the corank-zero situation where $r=0$ and therefore $V=W$, $\cU_W=\U(V) \times \U(V)$, $\cH_W=\U(V)$ embedded diagonally and $\nu^\vee$ is simply the Weil representation of $\U(V)$. In this setting, we will write $\U_V$ for $\U(V) \times \U(V)$ and $\U_V'$ for $\U(V)$. 

If $\sigma$ is a representation of $\cU_W$ (either in the global or local context), we can consider the parabolic subgroup $P(X)$ of $\U(V)$ stabilizing the flag $X \subset X \oplus W \subset V$. Its Levi subgroup stabilizing $W$ is $G_r \times \U(W)$. For any representation $\tau$ of $G_r$, we can consider the parabolic induction $\Sigma:=I_{\U(V) \times P(X)}^{\U_V} \sigma \boxtimes \tau$. We show in \S\ref{sec:local_FJ} and \S\ref{sec:global_FJ} that the global and local periods $\cP_{\cH_W}$ and $\cP_{\cH_W,v}$ on $\sigma$ can be unfolded in terms of $\cP_{\U_V'}$ and $\cP_{\U_V',v}$ on $\Sigma$ respectively. For the first case, this relies on the unramified computations carried out in \cite{Boi}. In the local situation, if we assume that $\sigma$ and $\tau$ are tempered, so will $\Sigma$, and this reduces Theorem~\ref{thm:temp_int} to the $r=0$ case which was proved in \cite{Xue2}. In the global situation, $\Sigma$ will not be cuspidal anymore. This leads us to consider the Gan--Gross--Prasad and Ichino--Ikeda conjectures for certain Eisenstein series. 

\subsubsection{Regular parameters}

We now describe the class of Eisenstein series we are interested in. Let $Q$ be a standard (with respect to the upper triangular matrices) parabolic subgroup of $\U_V$ with standard Levi $M_Q:=(G_{n_1} \times \hdots \times G_{n_k} \times \U(V_0)) \times (G_{n'_1} \times \hdots \times G_{n'_{k'}} \times \U(V_0'))$, where $V_0$ and $V_0'$ are skew-Hermitian spaces included in $V$. Let $\sigma$ be a cuspidal representation of $M_Q$ which we decompose accordingly as $\sigma=(\tau_1 \boxtimes \hdots \boxtimes \tau_k \boxtimes \sigma_0) \boxtimes (\tau_1' \boxtimes \hdots \boxtimes \tau'_{k'} \boxtimes \sigma_0')$. Set $\Sigma=I_Q^{\U_V} \sigma$. Then we say that $\Sigma$ is \emph{$(\U_V,\U_V',\overline{\mu})$-regular with generic base-change} if
\begin{itemize}
    \item the base-changes $\mathrm{BC}(\sigma_0)$ and $\mathrm{BC}(\sigma_0')$ are generic;
    \item the representations $\tau_1,\hdots,\tau_k, \tau_1^*, \hdots, \tau_k^*$ are all distinct, where we write $\tau^*$ for the conjugate dual of $\tau$;
    \item the representations $\tau'_1,\hdots,\tau'_{k'}, \tau_1^{',*}, \hdots, \tau_{k'}^{',*}$ are all distinct;
    \item for all $1 \leq i \leq k$ and $1 \leq j \leq k'$, $\overline{\mu} \tau_i$ is not isomorphic to the contragredient of $\tau_j'$.
\end{itemize}
The induced representation $\Sigma$ considered in \S\ref{subsubsec:reduc_zero} is $(\U_V,\U_V',\overline{\mu})$-regular with generic base-change for suitable choice of $\tau$ as soon as $\BC(\sigma)$ is generic.

\subsubsection{Fourier--Jacobi periods for Eisenstein series}

Let $\Sigma=I_Q^{\U_V} \sigma$ be $(\U_V,\U_V',\overline{\mu})$-regular. For $\varphi \in \Sigma$ and $\lambda \in i \fa_Q^*$ we have the Eisenstein series $E(\varphi,\lambda)$. For $\phi \in \nu^\vee$, the Fourier--Jacobi period $\cP_{\U_V'}$ defined in \eqref{eq:global_FJ_intro} is not absolutely convergent. In \cite{BLX1}, we defined a truncation operator $\Lambda^T_u$. It is an adaptation of the work of \cite{IY2}, our situation being slightly more complicated due to the theta series $\theta_{\U_V'}$. We show in Proposition~\ref{prop:IY_stability} that, because $\Sigma$ is $(\U_V,\U_V',\overline{\mu})$-regular, the truncated period 
\begin{equation*}
    T \mapsto \int_{[\U_V']} \Lambda^T_u \left(E(\varphi,\lambda) \theta_{\U_V'}^\vee(\phi) \right)(h) \rd h
\end{equation*}
is constant equal to a number $\cP_{\U_V'}(E(\varphi,\lambda),\phi)$. This is the \emph{regularized Fourier--Jacobi period}.

\subsubsection{The conjectures for Eisenstein series}

Our main result on Fourier--Jacobi periods for Eisenstein series is the following.

\begin{theorem}
    \label{thm:GGP-IY}
    Assume that $\Sigma$ is $(\U_V,\U_V',\overline{\mu})$-regular with generic base-change, and let $\lambda \in i \fa_Q^*$. Then the Gan--Gross--Prasad and Ichino--Ikeda conjectures hold for the periods $\cP_{\U_V'}(E(\varphi,\lambda),\phi)$.
\end{theorem}

We defer to Theorems~\ref{thm:ggp_intro} and \ref{thm:II_regular_intro} for precise statements. In fact, the versions we prove do not rely on the existence of $\BC(\sigma)$ and are independent of \cite{Mok} and \cite{KMSW}. As claimed above, this result implies Theorems \ref{thm:GGP} and \ref{thm:II} by unfolding.

\subsection{The spectral expansion of Liu's trace formula}

To prove Theorem~\ref{thm:GGP-IY}, we perform a comparison of the spectral sides of Liu's relative trace formulae. This analysis takes the first half of this paper and relies on the results of \cite{BLX1} and \cite{BLX2}.

\subsubsection{Two RTFs}
The first RTF we consider is the \emph{unitary} one, and is associated to the double quotient $\U_V' / \U_V \backslash \U_V'$. More precisely, for a Schwartz function $f_V$ on $\U_V$ with kernel $K_{f_V}$,  and for $\phi_1,  \phi_2 \in \nu^\vee$, we tentatively define
\begin{equation*}
    J_V(f_V \otimes \phi_1 \otimes \phi_2)=\int_{[\U_V']}^{\mathrm{reg}} \int_{[\U_V']} K_{f_V}(h_1,h_2) \theta^\vee(h_1,\phi_1) \theta^\vee(h_2,\phi_2) \rd h_1 \rd h_2.
\end{equation*}
This integral is not absolutely convergent in general and needs to be regularized.

The second RTF is defined on general linear groups. Set $G=G_n \times G_n$, $H=G_n$ (embedded diagonally into $G$) and $G'=\GL_n \times \GL_n$ (seen as a subgroup of $G$). For $\Phi$ a Schwartz function in $n$ variables over $\bA_E$, we have the Epstein Eisenstein series $\Theta(\Phi)$ from \cite{JS}. For $f$ a Schwartz function on $G$, we define
\begin{equation*}
    I(f \otimes \Phi)=\int_{[H]}^{\mathrm{reg}} \int_{[G']} K_f(h,g') \Theta(h,g') \overline{\mu}(\det h) \valP{\det h}^{1/2} \eta(g') \rd g' \rd h,
\end{equation*}
where $\eta$ is some character (see \S\ref{sec:coarse_spectral_u} for the precise definition). Once again, this is not convergent in general.

\subsubsection{Main inputs from \cite{BLX1}}

In \cite{BLX1}, we build regularizations of $J_V$ and $I$ using truncation methods inspired by \cite{Zydor3}. Moreover, we compute the geometric expansions of these RTFs in terms of orbital integral. They also need to be regularized. Finally, we show that $J_V$ and $I$ admit coarse spectral expansions in terms of cuspidal data.

\subsubsection{Main inputs from \cite{BLX2}}

In \cite{BLX2}, we define a notion of matching for tuples $(f_V,\phi_1,\phi_2)$ and $(f,\Phi)$ as a correspondence between orbital integrals. We also provide a version of the fundamental lemma, and prove that the spaces of transferable test functions are dense. This gives all the tools needed for the comparison of the geometric expansions of $J_V$ and $I$.

We also study a local analogue of Liu's trace formula which mirrors the work of \cite{BP} in the Bessel case. We show that local matching is equivalent to an equality of local relative characters. On the unitary side, they are attached to the local Fourier--Jacobi period $\cP_{\U_V',v}$. This spectral characterization of matching will be used in the proof of the Ichino--Ikeda conjecture.

\subsubsection{This paper: comparison of the spectral sides}

In this paper, we compute the spectral expansions of $J_V$ and $I$. More precisely, on the unitary side we derive the spectral contributions of $(\U_V,\U_V',\overline{\mu})$-regular automorphic representations with generic base-change. These contributions take the form of a relative characters which involve the regularized period $\cP_{\U_V'}(E(\varphi,\lambda),\phi)$. 

On the side of general linear groups, we compute the contributions coming from base-changes of these representations. In this paper, they are called $(G,H,\overline{\mu})$-regular Hermitian Arthur parameters. We defer to \S\ref{subsec:reg_herm_art} for a precise definition. The expansion in that case is written in terms of Rankin--Selberg and Flicker--Rallis periods, respectively denoted by $\cP_H$ and $\cP_{G'}$. The first computes the central value of the Rankin--Selberg $L$-function by \cite{JPSS83}, and the second detects the image of $\mathrm{BC}$ by \cite{Fli}. However, in both cases these periods also need to be regularized to allow for integrations of Eisenstein series. For the $\cP_{G'}$ this was already done in \cite{BPC22}, and for $\cP_H$ we use a process of analytic continuation in the spirit of Tate's thesis.

Once the spectral expansions of $J_V$ and $I$ are computed, we relate them via the equality of the geometric expansions for suitable test functions. To isolate representations, we crucially use the results of \cite{BPLZZ} on multipliers. This step requires the spectral characterization of matching from \cite{BLX2}, as it allows us to ensure that spectral projections of matching test functions still match.

The outcome of the comparison is stated as Theorem~\ref{thm:global_identity} in the text and says that the relative character associated to $\BC(\Sigma)$ can be compared to the sum of those associated to $\Sigma'$, as $\Sigma'$ varies in the global Arthur packet of $\Sigma$. This readily implies Theorem~\ref{thm:GGP-IY}.

\subsection{Organization of the paper}

In Section~\ref{sec:preliminaries}, we fix some notation pertaining to automorphic forms with a special emphasis on spaces of functions. We also prove a technical result on the convergence of relative characters (Proposition~\ref{prop:strong_K_basis}). The rest of the paper is divided in two parts.

In Part~\ref{part:corank_zero}, we only work in the corank-zero situation and prove Theorem~\ref{thm:GGP-IY}. We compute in Section~\ref{sec:spec_GLN} and \ref{sec:spec_U} the spectral expansions of Liu's RTFs, first on the general linear side and then on the unitary side. We then proceed in Section~\ref{sec:comparison} to compare them using the comparisons of the geometric sides from \cite{BLX1} and \cite{BLX2}. The final result is Theorem~\ref{thm:global_identity}. We finally prove Theorem~\ref{thm:GGP-IY} in Section~\ref{sec:proof_Eisenstein}

In Part~\ref{part:positive}, we prove our results in positive corank. In Section~\ref{sec:local_FJ}, we study local Fourier--Jacobi periods and prove that they satisfy certain unfolding relations. We also obtain the non-vanishing result of Theorem~\ref{thm:temp_int}. In Section~\ref{sec:global_FJ}, we switch to the global setting and write an unfolding equality for $\cP_{\cH_W}$. Finally, in Section~\ref{sec:proofs_positive} we reduce Theorems~\ref{thm:GGP} and \ref{thm:II} to the corank-zero situation which we proved in Theorem~\ref{thm:GGP-IY}.

\subsection{Acknowledgments}
We thank Rapha\"el Beuzart-Plessis and Wei Zhang for many helpful discussions.
PB was partly funded by the European Union ERC Consolidator Grant, RELANTRA, project number 101044930. Views and opinions expressed are however those of the author only and do not necessarily reflect those of the European Union or the European Research Council. Neither the European Union nor the granting authority can be held responsible for them. WL was partially supported by the National Science Foundation under Grant No. 1440140, while he was in residence at the Mathematical Sciences Research Institute in Berkeley, California, during the semester of Spring 2023. HX is partially supported by the NSF grant DMS~\#2154352.

\section{Preliminaries on spaces of functions and automorphic representations}

\label{sec:preliminaries}

In this section, we collect some general notation on reductive groups and their automorphic forms. We give a special emphasis on spaces of functions on automorphic quotients.

\subsection{Reductive groups}

Let $F$ be a field of characteristic zero. All algebraic groups are defined over $F$.

\subsubsection{General notation}

Let $G$ be an algebraic group. Let $Z_G$ be the center of $G$. Let $N_G$ be the unipotent radical of $G$ and let $X^*(G)$ be the group of $F$-algebraic characters of $G$. Set $\fa^*_G=X^*(G) \otimes_{\Z} \R$ and $\fa_G=\Hom_\Z(X^*(G),\R)$. 

Assume that $G$ is connected reductive. Let $P_0$ be a minimal parabolic subgroup of $G$. Let $M_0$ be a Levi factor of $P_0$. We say that a parabolic subgroup of $G$ is standard (resp. semi-standard) if it contains $P_0$ (resp. if it contains $M_0$). If $P$ is a semi-standard parabolic subgroup of $G$, we will denote by $M_P$ its unique Levi factor containing $M_0$. We have a decomposition $P=M_P N_P$. 

Let $A_G$ be the maximal central $F$-split torus of $G$. If $P$ is a standard parabolic subgroup of $G$, set $A_P=A_{M_P}$. Write $A_0$ for $A_{P_0}$. The restriction maps $X^*(P) \to X^*(M_P) \to X^*(A_P)$ induce isomorphisms $\fa_P^* \simeq \fa^*_{M_P} \simeq \fa_{A_P}^*$. If $P=P_0$, we write $\fa_0$ and $\fa_0^*$ for $\fa_{P_0}$ and $\fa_{P_0}^*$. For $P \subset Q$, we have the subspaces $\fa_P^Q$ and $\fa_P^{Q,*}$ induced by the restriction maps $X^*(Q)\to X^*(P)$ and $X^*(A_P) \to X^*(A_Q)$ respectively. We will also write $\fa_P^{M_Q,*}$.

Set $\fa_{P,\C}^Q=\fa_{P}^Q \otimes_\R \C$ and $\fa_{P,\C}^{Q,*}=\fa_{P}^{Q,*} \otimes_\R \C$. We have decompositions $\fa_{P}^{Q}=\fa_{P}^{Q} \oplus i \fa_{P}^{Q}, \quad \fa_{P,\C}^{Q,*}=\fa_{P}^{Q,*} \oplus i \fa_{P}^{Q,*}$ where $i^2=-1$. We denote by $\Re$ and $\Im$ the real and imaginary parts associated to these decompositions, and by $\overline{\lambda}$ the complex conjugate of any $\lambda \in \fa_{P,\C}^{Q,*}$.

Let $W$ be the Weyl group of $(G,A_0)$, which is by definition the quotient of the normalizer $N_{G(F)}(A_0(F))$ by the centralizer $Z_{G(F)}(A_0(F))$. It acts on $\fa_0$ and $\fa_0^*$. If $w \in W$, we will write again $w$ for a representative in $G(F)$.

Let $P$ be a standard parabolic subgroup of $G$, and write $\Delta_0^P$ for the set of roots of $A_0$ in $M_P \cap P_0$. If $Q$ is another parabolic subgroup of $G$, we write $W(P,Q)$ for the set of $w \in W$ such that $w \in \Delta_0^P=\Delta_0^Q$. In particular, if $w \in W(P,Q)$ we have $w M_P w^{-1}=M_Q$. 

\subsection{Automorphic quotients and Haar measures}

We now assume that $F$ is a number field. Let $G$ be a connected reductive group over $F$.

\subsubsection{Automorphic quotients} Let $\bA$ be the adele ring of $F$, let $\bA_f$ be its ring of finite adeles. Set $F_\infty=F \otimes_\Q \R$. Let $V_F$ be the set of places of $F$ and let $V_{F,\infty} \subset V_F$ be the subset of Archimedean places. For $v \in V_F$, let $F_v$ be the completion of $F$ at $v$. If $v$ is non-Archimedean, let $\cO_{v}$ be its ring of integers. Let $\valP{\cdot}$ be the absolute value $\bA^\times \to \R_+^\times$ given by taking the product of the normalized absolute values $\valP{\cdot}_v$ on each $F_v$. If $\tS$ is a finite set of places of $F$, set $F_\tS=\prod_{v \in \tS} F_v$. If $\tS$ only contains non-Archimedean places, write $\cO_F^{\tS}=\prod_{v \notin \tS} \cO_{v}$.

Let $P=M_P N_P$ be a semi-standard parabolic subgroup of $G$. Set 
\begin{equation*}
    [G]_{P}=M_P(F) N_P(\bA) \backslash G(\bA).
\end{equation*}
Let $A_{P,\Q}$ be the maximal $\Q$-split subtorus of the Weil restriction $\Res_{F/\Q}A_P$, and let $A_P^\infty$ be the neutral component of $A_{P,\Q}(\R)$. Set 
\begin{equation*}
    [G]_{P,0}=A_P^\infty M_P(F) N_P(\bA) \backslash G(\bA).
\end{equation*}
If $P=G$, we simply write $[G]$ and $[G]_0$ for $[G]_G$ and $[G]_{G,0}$ respectively.
\subsubsection{Maximal compact subgroups} \label{subsubsec:max_compact_prelim}Let $K=\prod_{v \in V_F} K_v \subset G(\bA)$ be a "good" maximal compact subgroup in good position relative to $M_0$. We write $K=K_\infty K^\infty$ where $K_\infty=\prod_{v \in V_{F,\infty}} K_v$ and $K^\infty=\prod_{v \in V_F \setminus V_{F,\infty}} K_v$.

\subsubsection{Harish-Chandra map} Let $P$ be a semi-standard parabolic subgroup of $G$. There is a canonical morphism $ H_P : P(\bA) \to \fa_P$ such that $\langle \chi, H_P(g) \rangle=\log \valP{\chi(g)}$ for any $g \in P(\bA)$ and $\chi \in X^*(P)$. We extend it to the Harish--Chandra map $H_P : G(\bA) \to \fa_P$ which satisfies: for any $g \in G(\bA)$ we have $H_P(g)=H_P(p)$ whenever $g \in pK$ with $p \in P(\bA)$.

\subsubsection{Tamagawa measure}
\label{subsec:tamagawa_measure_new}

Let $\psi$ be a non-trivial automorphic character of $\bA$. For every place $v$ of $F$, let $d_{\psi_v} x_v$ be the unique Haar measure on $F_v$ which is autodual with respect to $\psi_v$ the local component of the additive character $\psi$ of $\bA$.

The choice of a left-invariant rational volume form $\omega$ on $G$ together with the measure $d_{\psi_v} x_v$ determine a left-invariant Haar measure $d_{\psi_v} g_v$ on each $G(F_v)$ (\cite{Wei}). By~\cite{Gro97}, there is an Artin--Tate $L$-function $L_{G}(s)=\prod_v L_{G,v}(s)$. More generally for $\tS$ a finite set of places we can consider $L_{G}^{\tS}(s)=\prod_{v \notin \tS} L_{G,v}(s)$. Define $\Delta_{G}^*$ and $\Delta_{G}^{\tS,*}$ to be the leading coefficient of the Laurent expansion at $s=0$ of $L_{G}(s)$ and $L_{G}^{\tS}(s)$ respectively. For each place $v$, set $\Delta_{G,v}:=L_{G,v}(0)$. We equip $G(\bA)$ with the Tamagawa measure $d g$ defined as $\rd g=\rd_\psi g_{\tS} \times \rd_\psi g^{\tS}$ where $\rd_\psi g_{\tS}=\prod_{v \in \tS} \rd_{\psi_v} g_v$ and $\rd_\psi g^{\tS}=(\Delta_{\G}^{\tS,*})^{-1}\prod_{v \notin \tS} \Delta_{G,v} \rd_{\psi_v} g_v$. Note that for any model of $G$ over $\cO_F^{\tS}$ we have for almost all $v$
\begin{equation}
\label{eq:Delta}
    \vol(G(\cO_v),\rd_{\psi_v} g_v)=\Delta_{G,v}^{-1}.
\end{equation}
Although the $\rd_{\psi_v} g_v$ depend on various choices, the Tamagawa measure $\rd g$ does not. This construction is valid as soon as $G$ is a linear algebraic group over $F$ (not necessarily reductive).

\subsubsection{Measures on subgroups} \label{subsubsec:first_measure} 
Let $P$ be a standard parabolic subgroup of $G$. We equip $\fa_P$ with the Haar measure that gives covolume $1$ to the lattice $\Hom(X^*(P),\Z)$. We give $i \fa_P^*$ the dual Haar measure. If $P \subset Q$, we equip $\fa_P^Q=\fa_P/\fa_Q$ with the quotient measure. We equip $A_P^\infty$ with the Haar measure compatible with the isomorphism $A_P^\infty \simeq \fa_P$ induced by the Harish-Chandra morphism (see \cite{BPCZ}*{Section~2.2.5}).

We give $[G]_P$ the quotient of our measure on $G(\bA)$ by the product of the counting measure on $M_P(F)$ with our measure on $N_P(\bA)$, and $[G]_{P,0}$ the "semi-invariant" quotient measure by the one of $A_P^\infty$ (defined for functions on $[G]_P$ which transform by $\delta_P$ under left-translation by $A_P^\infty$).

\subsection{Spaces of functions}

We keep the assumption that $F$ is a number field and that $G$ is connected reductive over $F$. 

\subsubsection{Comparison of functions}
\label{subsubsec:functions_comp}

Let $X$ be a set, let $f$ and $g$ be two positive functions on $X$. We write $f(x) \ll g(x), \; x \in X$ if there exists $C>0$ such that $f(x) \leq C g(x)$ for all $x \in X$. If we want to emphasize that the constants depend on some additional data $y$, we will write $\ll_y$.

\subsubsection{Topological vector spaces} All the topological $\C$-vector spaces that we will consider in this paper are locally convex. We will use the notation $\widehat{\otimes}$ to denote the completed tensor product between two such spaces.

\subsubsection{Smooth functions} Let $\fg_\infty$ be the Lie algebra of $G(F_\infty)$, let $\cU(\fg_\infty)$ be the enveloping algebra of its complexification and let $\cZ(\fg_\infty)$ be the center of $\cU(\fg_\infty)$. 

By a level $J$ we mean a normal open compact subgroup of $K^\infty$. If $V$ is a representation of $G(\bA)$, we denote by $V^J$ its subspace of vectors invariant by $J$.

Let $V$ be a Fréchet space. We say that a function $\varphi : G(\bA) \to V$ is smooth if it is right-invariant by some level $J$ and if for every $g_f \in G(\bA_f)$, the function $g_\infty \in G(F_\infty) \mapsto \varphi(g_f g_\infty)$ is smooth in the usual sense (i.e. belongs to $C^\infty(G(F_\infty))$). We write $\rR$ (resp $\mathrm{L}$) for the actions by right-translation (resp. left-translation) of $G(\bA)$ and $\cU(\fg_\infty)$ on such smooth functions.

\subsubsection{Heights, growth of functions} \label{subsec:heights} We take a height $\aabs{\cdot}$ on $G(\bA)$ as in \cite{BPCZ}*{Section~2.4.1}. If $P$ is a standard parabolic subgroup of $G$, for any $g \in G(\bA)$ we define
\begin{equation*}
    \aabs{g}_P= \inf_{\delta \in M_P(F) N_P(\bA)} \aabs{\delta g}.
\end{equation*}

If $x=(x_1,\hdots,x_n) \in \bA^n$, we also have the height $\aabs{x}_{\bA^n}:=\prod_v \max(1,\valP{x_1}_v,\hdots,\valP{x_n}_v)$.

For all $N \in \R$, $X \in \cU(\fg_\infty)$ and any smooth function $\varphi : [G]_P \to \C$ we define
\begin{equation*}
    \aabs{\varphi}_{N,X}=\sup_{x \in [G]_P} \aabs{x}_P^N \valP{(\rR(X)\varphi)(x)}.
\end{equation*}
If $X=1$, we simply write $\aabs{\varphi}_{N}$.

\subsubsection{Schwartz functions} Let $\cS(G(\bA))$ be the space of Schwartz functions on $G(\bA)$ defined in \cite{BPCZ}*{Section~2.5.2}. It is a topological vector space, and an algebra for the convolution product $*$. For every level $J$, we denote by $\cS(G(\bA))^J$ its subalgebra of $J$-bi-invariant functions. For any $f \in \cS(G(\bA))$, we set $ f^\vee(g)=\overline{f(g^{-1})}$.

For every place $v$ (or more generally every finite subset $\tS \subset V_F$), we have the space of local Schwartz functions $\cS(G(F_v))$ (or $\cS(G(F_\tS)))$ defined in \cite{BP}*{Section~2.4}.

\subsubsection{Petersson inner-product} 
\label{subsec:petersson_innter}

Let us fix a standard parabolic subgroup $P$ of $G$ for the reminder of this section. We have the Hilbert space $L^2([G]_P)$ of square-integrable functions on $[G]_P$. More generally, for $N \in \R$ we denote by $L^2_N([G]_P)$ the space of functions on $[G]_P$ that are square-integrable with respect to the measure $\aabs{g}_P^N \rd g$. We denote by $L^2_N([G]_{P})^\infty$ its space of smooth vectors. It is given the topology described in \cite{BPCZ}*{Section~2.5.3}.

We will also consider $L^2([G]_{P,0})$ the space of functions on $[G]_P$ that transform by $\delta_P^{1/2}$ under left-translation by $A_P^\infty$ and such that the Petersson-norm
\begin{equation*}
    \aabs{\varphi}_{P,\Pet}^2=\langle \varphi,\varphi \rangle_{P,\Pet}:=\int_{[G]_{P,0}} \valP{\varphi(g)}^2 \rd g,
\end{equation*}
is finite. It is a Hilbert space.

\subsubsection{Functions of uniform moderate growth} For every $N \in \R$, let $\cT_N([G]_P)$ be the space of smooth functions $\varphi : [G]_P \to \C$ such that for every $X \in \cU(\fg_\infty)$ we have $\aabs{\varphi}_{-N,X} < \infty$. For every level $J$, we equip $\cT_N([G]_P)^J$ with the topology of Fréchet space induced by the family of semi-norms $(\aabs{\cdot}_{-N,X})_X$. Set
\begin{equation*}
    \cT([G]_P)=\bigcup_{N>0} \cT_N([G]_P).
\end{equation*}
This is the \emph{space of functions of uniform moderate growth on $[G]_P$}. It is equipped with a natural topology of locally Fréchet space.

We have a constant term map $ \varphi \in \cT([G]) \mapsto \varphi_P \in \cT([G]_P)$
defined by
\begin{equation}
\label{eq:constant_term_defi}
    \varphi_P(g)= \int_{[N_P]} \varphi(ng) \rd n , \quad \varphi \in \cT([G]), \quad g \in [G]_P.
\end{equation}

\subsubsection{Schwartz space of $[G]_P$}

Let $\cS([G]_P)$ be the space of smooth functions $\varphi : [G]_P \to \C$ such that for every $N \geq 0$ and $X \in \cU(\fg_\infty)$ we have $\aabs{\varphi}_{N,X} < \infty$. For every level $J$, we equip $\cS([G]_P)^J$ with the Fréchet topology induced by the family of semi-norms $(\aabs{\cdot}_{N,X})_{N,X}$, and $\cS([G]_P)$ with the locally convex direct limit topology. It is the \emph{Schwartz space} of $[G]_P$. 

\subsubsection{Harish-Chandra Schwartz space}
\label{subsec:HC_Schwartz}
Let $\cC([G]_P)$ be the Harish-Chandra Schwartz space of $[G]_P$ defined in \cite{BPCZ}*{Section~2.5.8}. The precise definition is not relevant to us, but we will make use of the following property.

\begin{lemma}[{\cite{Lap2}*{Equation~(9)}}]
    \label{lem:integration}
    The linear form
        \begin{equation*}
            \varphi \in \cC([G \times G]_{P \times P}) \mapsto \int_{[G]_P}\varphi(g) \rd g
        \end{equation*}
        is well-defined and continuous, where $[G]_P$ is embedded in $[G \times G]_{P \times P}$ diagonally.
\end{lemma}

\subsubsection{Relations between spaces of functions}

We summarize in a lemma the main properties of the spaces $\cT([G]_P)$, $\cS([G]_P)$ and $\cC([G]_P)$ that we will use.

\begin{lemma}
    \label{lem:functional_analysis}
    The following properties hold.
    \begin{enumerate}
        \item \cite{BPCZ}*{Equation~(2.5.7.6)}, \cite{BPCZ}*{Equation~(2.5.10.8)}. We have 
        \begin{equation}
            \cT([G]_P)=\bigcup_{N>0}  L^2_{-N}([G]_P)^\infty, \quad     \cS([G]_P) = \bigcap_{N>0} L^2_N([G]_P)^\infty.
        \end{equation}
        \item \cite{BPCZ}*{Equation~(2.5.5.4)} For every $N>0$ there exists $N'>0$ such that we have a continuous inclusion $L^2_{-N'}([G])^\infty \subset \cT_N([G])$.
        \item \cite{BPCZ}*{Section~2.5.10}, \cite{BPCZ}*{Section~2.5.8}, \cite{BPCZ}*{Section~2.5.7}. The space $\cS([G]_P)$ is dense in $\cT([G]_P)$, $\cC([G]_P)$ and $L^2_N([G]_P)^\infty$ for every $N \in \R$ (but in general not in $\cT_N([G]_P))$.
        \item \cite{BPCZ}*{Equation~(2.4.5.25)} The inclusion $\cC([G]_P) \subset \cT([G]_P)$ is continuous.
    \end{enumerate}
\end{lemma}

\subsection{Automorphic representations}

We keep the assumption that $F$ is a number field and that $G$ is connected reductive over $F$. We take $P$ a standard parabolic subgroup of $G$

\subsubsection{Automorphic forms}

We define the space of automorphic forms $\cA_P(G)$ to be the subspace of $\cZ(\fg_\infty)$-finite functions in $\cT([G]_P)$. Let $\cA_P^0(G)$ be the subspace of $\varphi \in \cA_P(G)$ such that
\begin{equation*}
    \varphi(ag)=\delta_P(g)^{1/2} \varphi(g), \quad a \in A_P^\infty, \quad g \in [G]_P.
\end{equation*}
If $P=G$ we simply write $\cA(G)$ and $\cA^0(G)$. These spaces receive the topology of \cite{BPCZ}*{Section~2.7.1}.

Let $\cA_{P,\disc}(G) \subset \cA_P^0(G)$ be the subspace of $\varphi$ such that the Petersson norm $\aabs{\varphi}_{P,\Pet}$ is finite. The spaces $\cA_{P}^0(G)$ and $\cA_{P,\disc}(G)$ receive actions of $G(\bA)$ and $\cS(G(\bA))$ by right-convolution.

\subsubsection{Automorphic representations} We define a discrete automorphic representation of $G(\bA)$ to be a topologically irreducible subrepresentation of $\cA_{\disc}(G)$. Let $\Pi_{\disc}(G)$ be the set of such representations. For $\pi \in \Pi_{\disc}(G)$, let $\cA_\pi(G)$ be the $\pi$-isotypic component of $\cA_{\disc}(G)$. Note that $\pi$ always has trivial central character on $A_G^\infty$. 

For $\pi \in \Pi_{\disc}(M_P)$, let $\cA_{P,\pi}(G)$ be the subspace of $\varphi \in \cA_{P,\disc}(G)$ such that for all $g \in G(\bA)$ the map $ m \in [M_P] \mapsto \delta_P(m)^{-1/2} \varphi(mg)$ belongs to $\cA_\pi(M_P)$. For any $\lambda \in \fa_{P,\C}^*$, set $\pi_\lambda=\pi \otimes \exp( \langle \lambda, H_{P}(\cdot) \rangle )$ and for $\varphi \in \cA_{P,\pi}(G)$ define $    \varphi_\lambda(g)=\exp( \langle \lambda, H_{P}(g) \rangle ) \varphi(g)$. The map $\varphi \mapsto \varphi_\lambda$ identifies $\cA_{P,\pi}(G)$ with a subspace of $\cA_P(G)$ denoted by $\cA_{P,\pi,\lambda}(G)$. We denote by $I_P(\lambda)$ the actions of $G(\bA)$ and $\cS(G(\bA))$ we obtain on $\cA_{P,\pi,\lambda}(G)$ by transporting those on $\cA_P(G)$. 

Let $\cA_{\cusp}(G) \subset \cA^0(G)$ be the subspace of $\varphi$ such that $\varphi_P=0$ for all standard proper parabolic subgroups $P$. Let $\Pi_{\mathrm{cusp}}(G)$ be the set of topologically irreducible subrepresentations of $\cA_{\mathrm{cusp}}(G)$. It is a subset of $\Pi_{\disc}(G)$. 

\subsubsection{Eisenstein series}
\label{subsubsec:Eisenstein}
Let $P \subset Q$ be standard parabolic subgroups of $G$. For any $\varphi \in \cA_{P,\disc}(G)$ and $\lambda \in \fa_{P,\C}^*$ we define
\begin{equation}
\label{eq:partial_Eisenstein}
    E^Q(g,\varphi,\lambda)=\sum_{\gamma \in P(F) \backslash Q(F)} \varphi_\lambda(\gamma g), \quad g \in G(\bA).
\end{equation}
This sum is absolutely convergent for $\Re(\lambda)$ in a suitable cone. It admits a meromorphic continuation to $\fa_{P,\C}^*$ by \cite{Lap}. If $Q=G$, we drop the exponent.

If $\varphi \in \cA_{P,\cusp}(G)$, we have the usual formula for the constant term of Eisenstein series
\begin{equation}
    \label{eq:constant_term_Eisenstein}
    E_Q(\varphi,\lambda)=\sum_{w \in W(P;Q)} E^Q(M(w,\lambda) \varphi,w \lambda),
\end{equation}
where $W(P;Q)$ is the subset of the Weyl group defined in \cite{MW95}*{Section~II.1.7}, and $M(w,\lambda)$ is the intertwining operator defined in \cite{MW95}*{Section~II.1.6}.

\subsubsection{Orthonormal bases}
\label{subsec:bases}
Let $\pi \in \Pi_{\disc}(M_P)$ such that $\cA_{P,\pi}(G)^J \neq \{0\}$. Let $\widehat{K}_\infty$ be the set of isomorphism classes of irreducible unitary representations of $K_\infty$. For any $\tau \in \widehat{K}_\infty$, let $\cA_{P,\pi}(G)^\tau$ be $\tau$-isotypic component of $\cA_{P,\pi}(G)$. For any level $J$, set $    \cA_{P,\pi}(G)^{\tau,J}=\cA_{P,\pi}(G)^{\tau} \cap \cA_{P,\pi}(G)^{J}$.

Let $\cB_{P,\pi}(\tau,J)$ be an orthonormal basis of $\cA_{P,\pi}(G)^{\tau,J}$ with respect to $\langle \cdot,\cdot \rangle_{P,\Pet}$. We then define $\cB_{P,\pi}(J)$ to be the union over $\tau \in \widehat{K}_\infty$ of the $\cB_{P,\pi}(\tau,J)$.

\subsubsection{Relative characters}
In this section, we prove a bound for sums along the bases $\cB_{P,\pi}(J)$. It will be used to define our relative characters.

    \begin{prop}
        \label{prop:strong_K_basis}
        Let $J$ be a level of $G$, let $f \in \cS(G(\bA))^J$. There exists $N>0$ such that the sum 
        \[
            \sum_{\pi \in \Pi_{\disc}(M_P)} \sum_{\varphi \in \cB_{P, \pi}(J)} I_P(\lambda, f) \varphi
            \otimes \overline{\varphi}
            \]
            is absolutely convergent in $\cT_N([G]_P) \otimeshat \overline{\cT_N([G]_P)}$. The convergence is uniform in $\lambda$ when $\Re(\lambda)$ lies in a fixed compact subset.
    \end{prop}
    
    \begin{proof}
        By \cite{Ch}*{Lemma~3.8.1.1}, there exists $N>0$ such that $\cA_{P,\disc}(G)^J \subset \cT_N([G]_P)$. Let $\aabs{\cdot}$ be a continuous semi-norm on this space. We can take it of the form $\aabs{\cdot}_{-N,X}$ for some $X \in \cU(\fg_\infty)$. Let $\pi \in \Pi_{\disc}(M_P)$ such that $\cA_{P,\pi}(G)^J \neq \{0\}$. By definition of our basis we want to bound 
        \begin{equation}
            \label{eq:to_bound}
            \sum_{\tau \in \widehat{K}_\infty} \sum_{\varphi \in \cB_{P, \pi}(\tau,J)} \aabs{I_P(\lambda,f)\varphi}\aabs{\varphi}.
        \end{equation}
        Let $\Omega_G$ and $\Omega_{K_\infty}$ are the Casimir operator of $G$ and $K_\infty$ respectively associated to the standard Killing form on $\fg_\infty$ (see \cite{Ch}*{Section~3.2.2}). Set $\Delta=\mathrm{Id}-\Omega_G+2\Omega_K$. 
        By \cite{Ch}*{Lemma~3.8.1.1}, there exist absolute constants $r \geq 1$ and $C>0$ (independent of $\pi$) such that \eqref{eq:to_bound} is bounded above by
        \begin{equation*}
            \sum_{\tau \in \widehat{K}_\infty} \sum_{\varphi \in \cB_{P, \pi}(\tau,J)} C\sum_{i,j=1}^r \aabs{\rR(\Delta^i)I_P(\lambda,f)\varphi}_{P,\Pet}\aabs{\rR(\Delta^j)\varphi}_{P,\Pet}.
        \end{equation*}
        Let $\lambda_\tau$ be the Casimir eigenvalue of $\tau$, and $\lambda_\pi$ of $\pi_\infty$. By \cite{Kn}*{Proposition~8.22}, we know that $\Delta$ acts on $\cA_{P,\pi}(G)^{\tau,J}$ by the scalar $1-\lambda_\pi+2\lambda_\tau$. It follows that there exists an absolute constant $B>0$ such that, for any $\tau \in \widehat{K}_\infty$ and $\varphi \in \cB_{P, \pi}(\tau,J)$, we have 
        \begin{equation*}
            \aabs{\rR(\Delta^i)\varphi}_{P,\Pet}\leq B^r(1+\lambda_\pi^2+\lambda_\tau^2)^r.
        \end{equation*}
        Assume that $\Re(\lambda)$ belongs to some compact set $\omega$. For any $R>0$, consider the operator $\Delta_R=R \times \mathrm{Id}-\Omega_G+2\Omega_{K_\infty}$. It acts on $\cA_{P,\pi_\lambda}(G)^{\tau,J}$ by the scalar $R-(\lambda,\lambda)-\lambda_\pi+2\lambda_\tau$, where we denote by $(\cdot,\cdot)$ the standard real quadratic form on $\fa_{0,\C}^*$. We can assume that $\cA_{P,\pi}(G)^{\tau,J} \neq \{0\}$, and by \cite{Mu2}*{Lemma~6.1} this implies that $\lambda_\tau \geq \lambda_\pi$. It follows that, up to increasing $B$, we can choose $R$ so that for any $\Re(\lambda) \in \omega$ we have $(1+\aabs{\lambda}^2+\lambda_\pi^2 + \lambda_\tau^2) \leq B\valP{R-(\lambda,\lambda)-\lambda_\pi+2\lambda_\tau}^2$. For any $N'$, for $f_{i,N'}=\rR(\Delta_R^{2N'})\mathrm{L}(\Delta^i)f$ we get
        \begin{align*}
            \aabs{\rR(\Delta^i)I_P(\lambda,f)\varphi}_{P,\Pet} &\leq B^{N'} (1+\aabs{\lambda}^2+\lambda_\pi^2+\lambda_\tau^2)^{-N'} \aabs{I_P(\lambda,f_{i,N'})\varphi}_{P,\Pet} \\
            &\leq B^{N'} (1+\lambda_\pi^2+\lambda_\tau^2)^{-N'} \aabs{I_P(\lambda,f_{i,N'})\varphi}_{P,\Pet}.
        \end{align*}
        But there exists a continuous semi-norm on $\cS(G(\bA))^J$ such that, for any $\Re(\lambda) \in \omega$, any $\varphi \in \cA_{P,\disc}(G)^J$, and any $1 \leq i \leq r$, we have $\aabs{I_P(\lambda,f_{i,N'})\varphi}_{P,\Pet} \leq \aabs{f} \aabs{\varphi}_{P,\Pet}$. To conclude, it remains to say that for $N'$ large enough we have a bound
        \begin{equation*}
            \sum_{\pi \in \Pi_{\disc}(M_P)} \sum_{\tau \in \widehat{K}_\infty} \dim(\cA_{P,\pi}(G)^{\tau,J}) (1+\lambda_\pi^2+\lambda_\tau^2)^{r-N'} < \infty.
        \end{equation*}
        This is \cite{Mu}*{Corollary~0.3}.
    \end{proof}

\subsection[Spectral decompositions]{Automorphic kernels and spectral decompositions}

\subsubsection{Cuspidal data and coarse Langlands decomposition}
\label{subsec:cuspi_data}

Let $\underline{\fX}(G)$ be the set of pairs $(M_P, \pi)$ where 
\begin{itemize}
    \item $P = M_P N_P$ is a standard parabolic subgroup of $G$,
    \item $\pi \in \Pi_{\cusp}(M_P)$ whose central character is trivial on $A_P^\infty$.
\end{itemize}
We say that two elements $(M_P, \pi)$ and $(M_Q, \tau)$ of $\underline{\fX}(G)$ are equivalent if there exists $w \in W(P, Q)$ such that $w.\pi = \tau$. We define a cuspidal datum to be an equivalence class of such $(M_P, \pi)$ and denote by $\fX(G)$ the set of all cuspidal data. If $\chi \in \fX(G)$ is represented by $(M_P, \pi)$ we define $\chi^\vee$ to be the cuspidal datum represented by $(M_P, \pi^\vee)$. 

Let $P$ be a standard parabolic subgroup of $G$. In \cite{BPCZ}*{Section~2.9.2}, to each $\chi \in \fX(G)$, a space $L^2_\chi([G]_P)$ is associated. It is the closure in $L^2([G]_P)$ of the space of pseudo-Eisenstein series induced from $\cA_{Q,\pi}(G)$ where the $(M_Q,\sigma)$ lie in the finite inverse image of $\chi$ under $\underline{\fX}(M_P) \to \fX(G)$. More generally, for $N \in \R$, $L_{N,\chi}^2([G]_P)$ is the closure of this same space in $L_N^2([G]_P)$. One can also extend these definitions to subsets $\fX \subset \fX(G)$ (see \cite{BPCZ}*{Section~2.9.2}). For any subset $\fX \subset \fX(G)$, set $\cS_{\fX}([G]_P)=\cS([G]_P) \cap L^2_{\fX}([G]_P)$, and for any $\cF \in \{L^2_N, \cT_N, \cT, \cC\}$, define $\cF_{\fX}([G]_P)$ to be the orthogonal of $\cS_{\fX^c}([G]_P)$ in $\cF([G]_P)$, where $\fX^c$ is the complement of $\fX$ in $\fX(G)$. We have a density result.
\begin{prop}[\cite{BPCZ}*{Section~2.9.5}]
\label{prop:density_schwartz}
    Let $\fX$ be a subset of $\fX(G)$. The space $\cS_{\fX}([G]_P)$ is dense in $L^2_{N,\fX}([G]_P)^\infty$ and $\cT_{\fX}([G]_P)$.
\end{prop}

Let $\cA_{P,\chi,\disc}(G)$ be the closed subspace of $\cA_{P,\disc}(G)$ of forms whose class belong to $L_\chi^2([G]_{P,0})$. We have an isotypical decomposition 
\begin{equation}
    \label{eq:chi_isotypic}
    \cA_{P,\chi,\disc}(G)=\widehat{\bigoplus_\pi} \cA_{P,\pi}(G).
\end{equation}
Let $J$ be a level of $G$. For any $\tau \in \widehat{K}_\infty$, let $\cB_{P,\chi}(\tau,J)$ be the union of the bases $\cB_{P,\pi}(\tau,J)$ where $\pi$ appears in \eqref{eq:chi_isotypic}. We then define $\cB_{P,\chi}(J)$ to be the union of the $\cB_{P,\chi}(\tau,J)$. 

\subsubsection{Regular cuspidal data}
\label{subsec:reg_cusp_data}
Let $\chi \in \fX(G)$ be a cuspidal datum represented by $(M_P,\pi)$. We say that $\chi$ is \emph{regular} if the only element in $w \in W(P,P)$ satisfying $w \pi \simeq \pi$ is $w=1$. It follows from the Langlands spectral decomposition theorem \cite{Langlands}, which describes the decomposition of $L^2([G])$ in terms of residues of Eisenstein series, that if $\chi$ is regular then the discrete representations $\pi$ appearing in \eqref{eq:chi_isotypic} are cuspidal.

\subsubsection{Decomposition of kernel functions}
\label{subsec:kernel_functions}

Let $f \in \cS(G(\bA))$. The right convolution $\rR(f)$ on $L^2([G])$ gives rise to a kernel $K_f$. For any $\chi \in \fX(G)$, the action of $\rR(f)$ on $L_\chi^2([G])$ is also an integral operator whose kernel is denoted by $K_{f,\chi}$. We now state a version of Langlands' spectral decomposition theorem for kernel functions. It is an extension of a result of \cite{Arthur3}*{Section~4} to Schwartz functions. In its statement, we add a subscript $x$ or $y$ to $\rR$ to indicate that this operator is applied in the variable $x$ or $y$. 

\begin{theorem}[{\cite{BPCZ}*{Lemma~2.10.2.1}}]   \label{thm:kernel_spectral_expansion}
    Let $J$ be a level and assume that $G$ has anisotropic center. There exist a continuous semi-norm $\aabs{\cdot}$ on $\cS(G(\bA))^J$ and an integer
    $N$ such that for all $X, Y \in \cU(\fg_{\infty})$, all $x, y \in G(\bA)$,
    and all $f \in \cS(G(\bA))^J$ we have 
        \[
        \begin{aligned}
        &\sum_{\chi \in \fX(G)} \sum_{P_0 \subset P}
        \int_{i \fa_{P}^{*}} \sum_{\tau \in \widehat{K}_\infty}
        \valP{\sum_{\varphi \in \cB_{P, \chi}(\tau,J)}
        \left( \mathrm{R}_x(X) E(x, I(\lambda, f) \varphi, \lambda) \right)
        \left( \mathrm{R}_y(Y) E(y, \varphi, \lambda) \right)} \rd  \lambda\\
        \leq &\aabs{\mathrm{L}(X) \mathrm{R}(Y)f}
        \aabs{x}_G^N \aabs{y}_G^N.
        \end{aligned}
        \]
        Moreover for all $x, y \in G(\bA)$ and all regular $\chi \in \fX(G)$ we have
    \begin{equation}
    \label{eq:spectral_expansion_kernel}
          K_{f, \chi}(x, y) = 
        \int_{i \fa_P^{*}} \sum_{\varphi \in \cB_{P, \chi}(J)}
        E(x, I(\lambda, f) \varphi, \lambda)
        \overline{E(y, \varphi, \lambda)} \rd \lambda.
    \end{equation}
      \end{theorem}

\part{Comparison of spectral expansions}
\label{part:corank_zero}

In Part~\ref{part:corank_zero}, we only work in the corank-zero situation (i.e. $V=W$). Our main goal is to complete the comparison of Liu's RTFs started in \cite{BLX1} and \cite{BLX2} by carrying out the analysis of the spectral sides. We compute the contributions of regular cuspidal data to the spectral expansions. The precise regularity conditions are given in \S\ref{subsec:reg_conditions_G} and \S\ref{sec:cusp_data_U}, and are tailored to include some cuspidal Eisenstein series and their base-changes. These contributions need to be regularized. This is done in \S\ref{sec:spectral_calculation_GL} for the ones on general linear groups, and \S\ref{sec:spectral_calculation_U} for unitary groups. Then in \S\ref{sec:Global_comparison} we input the results of \cite{BLX2} to compare the global relative characters associated to a (generic regular) Arthur packet on unitary groups to the one associated to its base-change: this is the content of Theorem~\ref{thm:global_identity}. Finally, in \S\ref{sec:proof_Eisenstein} we prove the generic GGP and Ichino--Ikeda conjectures from Theorem~\ref{thm:GGP-IY}.

\section{Groups, regularity conditions on automorphic representations and base-change}
\label{sec:notation}

In this section, we introduce some notations and notions relevant to this paper. The objective of Part~\ref{part:corank_zero} is to compare the two expansions of Liu's trace formula. The expansion on the side of general linear groups will be referred to as the \emph{$G$-side}, while for unitary groups it will be the \emph{$\U$-side}.

\subsection{Groups: $G$-side}

\subsubsection{Fields and automorphic characters}

Let $E/F$ be a quadratic extension of number fields. Write $\sfc$ for the non-trivial element in $\Gal(E/F)$. Let $\eta_{E/F}$ be the quadratic automorphic character of $\bA^\times$ associated to $E/F$ by global class field theory, and let $\mu$ be an automorphic (unitary) character of $\bA_E^\times$ that lifts $\eta_{E/F}$. Let $\tau \in E^-$ be an imaginary element (i.e. $\sfc(\tau)=-\tau$). Consider the automorphic characters of $\bA_E$ defined by $\psi_E(x)=\psi(\Tr_{E/F}(x))$ and $\psi^E(x)=\psi(\Tr_{E/F}(\tau x))$.

\subsubsection{Groups}

We denote by $E^n$ the $E$-vector space of $n$-dimensional column vectors with coordinates in $E$, and by $E_n$ the space of $n$-dimensional line vectors. We define similarly $\bA_E^n$ and $\bA_{E,n}$.

For any $k \geq 1$, set $G_k=\Res_{E/F} \GL_k$. We will use the following groups: 
\begin{equation*}
    G=G_n \times G_n, \quad H=G_n, \quad G'=\GL_{n,F} \times \GL_{n,F},
\end{equation*}
where $H$ is embedded diagonally in $G$, and $G'$ is embedded in $G$. We will also write $G'_n=\GL_{n,F}$. We write $(g_1,g_2)$ for the components of an element in $G$. We denote by $(T,B)$ the standard Borel pair of diagonal and upper triangular matrices in $G$. Let $N$ be the unipotent radical of $B$. We can write $T=T_{n} \times T_{n}$, $B=B_{n} \times B_{n}$ and $N=N_{n} \times N_{n}$. Let $e_n=(0,\hdots,0,1) \in E_n$ and let $P_n$ be the mirabolic subgroup of $G_n$ which consists of matrices whose last row is $e_n$. Set $P_{n,n}=P_n \times P_n$ seen as a subgroup of $G$. We set
\begin{equation*}
    G_+=G \times \Res_{E/F} \mathbf{A}_{E,n},
\end{equation*}
where by $\mathbf{A}_{E,n}$ we mean the $E$-affine space of $n$ dimensional row vectors, so that $G_+$ is an algebraic group over $F$.

If $A$ is any subgroup of $G$, set $A_H:=A \cap H$. In particular, $(T_{H},B_{H})$ is a standard Borel pair of $H$. We also set $A'=A \cap G'$.

Let $K$ be a maximal compact subgroup of $G(\bA)$ as in \S\ref{subsubsec:max_compact_prelim}. Then $K=K_n \times K_n$.

\subsection{Groups: $\U$-side}

\subsubsection{Skew-Hermitian spaces} Let $\cH$ be the set of isomorphism classes of $n$-dimensional non-degenerate skew-Hermitian spaces over $E/F$. If $\tS$ is a finite set of places of $F$, let $\cH_{\tS}$ be the set of isomorphism classes of $n$-dimensional non-degenerate skew-Hermitian spaces over $E_\tS=E \otimes_F F_{\tS}$. Let $\cH^{\tS}$ be the set of $V \in \cH$ such that $V \otimes_E E_v$ admits a self-dual lattice for all non-Archimedean $v \notin \tS$. These sets are both finite. If $V \in \cH$, we denote by $q_V$ its skew-Hermitian form.

We henceforth fix $V \in \cH$. Denote by $\U(V)$ the subgroup of $\GL(V)$ consisting of the $E$-linear unitary transformations of $V$. It is a $F$-algebraic group. Consider the $F$-vector space $\mathbb{V}:=\Res_{E/F} V$ equipped with the symplectic form $q_{\mathbb{V}}:=\Tr_{E/F} \circ q_V$. We fix a polarization $\bV=Y\oplus Y^\vee$, so that $Y$ and $Y^\vee$ are $n$-dimensional affine spaces over $F$.

\subsubsection{Groups}

We define the groups
\begin{equation*}
    \U_V:=\U(V) \times \U(V), \quad \U'_V=\U(V) \subset \U_V,
\end{equation*}
where the last embedding is diagonal inclusion. 

Let $S(V)$ be the \emph{Heisenberg group} of $V$, i.e. the $F$-linear group whose points are $\bV \times F$, with group law is given by the rule
\begin{equation}
    \label{eq:Heisenberg_group}
    (v_1,z_1).(v_2,z_2)=\left(v_1+v_2,z_1+z_2+\frac{1}{2}q_{\bV}(v_1,v_2) \right).
\end{equation}
It receives an action of $\U(V)$ on the first coordinate. We then define the \emph{Jacobi group} of $V$ to be
\begin{equation*}
    J(V):=S(V) \rtimes \U(V).
\end{equation*}

We finally define
\begin{equation*}
    \widetilde{\U}_V:=\U(V) \times J(V), \quad \U_{V,+}=\U_V \times Y^\vee \times Y^\vee.
\end{equation*}
If $g \in \widetilde{\U}_V$, we denote again by $g$ its projection on $\U_V$, and by $\tilde{g}$ its component in $J(V)$.

\label{subsubsec:max_compact}
Let $K_V$ be a maximal compact subgroup of $\U_V$ as in \S\ref{subsubsec:max_compact_prelim}. Assume that $\tS$ contains $V_{F,\infty}$ and the places that ramify in $E$. If $V \in \cH^\tS$, for every $v \notin \tS$, $\U_V$ admits a model over $\cO_{F_v}$ and we can take $K_{V,v}:=\U_{V}(\cO_{F_v})$ (the product of stabilizers of a self-dual lattice in $V \otimes_E E_v$). We write $K_V'$ for $K_V \cap \U_V'$.

\subsection{Measures}
\label{subsec:measure_part1}
For $\bG \in \{G, \U_V\}$, for each place $v$ we define in \cite{BLX2}*{Section~5.2} a local measure $\rd g_v$ on $\bG(F_v)$. It follows from \cite{BP}*{Section~2.5} that they are local components on the global Tamagawa measure on these groups from \S\ref{subsec:tamagawa_measure_new}, in the sense that
     \begin{equation}
     \label{eq:Tamagawa_facto}
        \rd g = (\Delta_{\bG}^*)^{-1} \prod_v \Delta_{\bG,v} \rd g_v.
    \end{equation}
We have the formulae for each place $v$
\begin{equation}
    \label{eq:delta_U_formula}
    \Delta_{\U(V)}=\prod_{i=1}^n L(i,\eta_{E/F}^i), \quad \Delta_{\U(V),v}=\prod_{i=1}^n L(i,\eta_{E_v/F_v}^i).
\end{equation}

If $X$ is an affine space over $F$, we equip $X(\bA)$ with the Tamagawa measure $\rd x=\prod_v \rd x_v$.

\subsection{Conditions on cuspidal data and Arthur parameters: $G$-side}
\label{subsec:reg_conditions_G}
Before diving in the spectral expansion of Liu's trace formula, we introduce some regularity conditions on Arthur parameters, i.e. on constituents (either discrete or continuous) of $L^2([G])$. Our goal here is to describe the spectral contributions attached to base changes of cuspidal representations of unitary groups with generic base-change. In our language, they will be \emph{discrete Hermitian Arthur parameters}. However, to prove Theorem~\ref{thm:GGP-IY} we will also need to deal with more general contributions that will be \emph{$(G,H,\overline{\mu})$-regular Hermitian}.

\subsubsection{Discrete Hermitian Arthur parameters}
\label{subsubsec:discreteHermitianArthur}

Let $\pi$ be an irreducible automorphic
representation of $G_k(\bA)$ for some $k \geq 1$. Let $\pi^\vee$ be the contragredient of
$\pi$, and set $\pi^\sfc = \pi \circ \sfc = \{ \varphi \circ \sfc \mid \varphi \in \pi\}$
and $\pi^* = (\pi^\sfc)^\vee$. We say that $\pi$ is conjugate self-dual if
$\pi\simeq \pi^*$.

We shall say that an irreducible automorphic representation $\Pi$ of $G_n$ is a \emph{discrete Hermitian Arthur parameter} of $G_n$ if there exist a partition $n = n_1 + \cdots + n_r$ and $\pi_i \in \Pi_{\cusp}(G_{n_i})$ for $i = 1, \hdots, r$, such that
\begin{itemize}
    \item $\Pi$ is isomorphic to the full parabolically induced representation
        $I_P^{G_n}(\pi_1 \boxtimes \hdots \boxtimes \pi_r)$ where $P$ is the
        standard parabolic subgroup of $G_n$ with Levi factor $G_{n_1} \times \hdots
        \times G_{n_r}$;
    \item $\pi_i$ is conjugate self-dual and the Asai
        $L$-function $L(s,\pi_i,\mathrm{As}^{(-1)^{n+1}})$ has a pole at
        $s=1$ for all $1 \leq i \leq r$;
    \item the representations $\pi_i$ are mutually non-isomorphic for $1
        \leq i \leq r$.
\end{itemize}
The integer $r$ and the automorphic representations $(\pi_i)_{1 \leq i \leq r}$ are unique (up to permutation), and we define the group $S_\Pi=(\Z/2\Z)^r$.

\subsubsection{Regular Hermitian Arthur parameter}
\label{subsec:reg_herm_art}
    
We shall say that an irreducible automorphic representation $\Pi$ of
$G_n(\bA)$ is a \emph{regular Hermitian Arthur parameter} of $G_n$ if
\begin{itemize}
    \item $\Pi$ is isomorphic to 
        $I_P^{G_n}((\tau_1 \boxtimes \tau_1^*) \boxtimes \hdots \boxtimes (\tau_s \boxtimes \tau_s^*) \boxtimes \Pi_0)$ where $P$ is
        a parabolic subgroup of $G_n$ with Levi factor $G_{m_1}^2 \times
        \hdots \times G_{m_s}^2 \times G_{n_0}$ and $n_0+2(m_1 + \hdots + m_r)=n$;
    \item $\Pi_0$ is a discrete Hermitian Arthur parameter of $G_{n_0}$;
    \item $\tau_i \in \Pi_{\cusp}(G_{m_i})$ for $1 \leq i \leq s$ and the representations $\tau_1, \hdots \tau_s, \tau_1^*, \hdots,  \tau_s^*$ are mutually non-isomorphic.
\end{itemize}
The representation $\Pi_0$ is then uniquely determined by $\Pi$, and is
called the \emph{discrete component} of $\Pi$. 

Let $w$ be the permutation matrix that exchanges the blocks $G_{m_i}$ corresponding to $\tau_i$ and $\tau_i^*$ for all $1 \leq i \leq s$. Set
\begin{equation}
\label{eq:fa_Pi}
    \fa_{\Pi,\C}^*:=\{ \lambda \in \fa_{P,\C}^* \; | \; w \lambda=-\lambda\}, \quad i\fa_{\Pi}^*:=\fa_{\Pi,\C}^* \cap i \fa_P^*.
\end{equation}
For any $\lambda \in \fa_{\Pi,\C}^*$, consider the full induced representation
\begin{equation*}
    \Pi_\lambda:=I_P^{G_n}(((\tau_1 \boxtimes \tau_1^*) \boxtimes \hdots \boxtimes (\tau_s \boxtimes \tau_s^*) \boxtimes \Pi_0)\otimes\lambda).
\end{equation*}
If $\lambda \in i\fa_{\Pi}^*$, then $\Pi_\lambda$ is irreducible.

Finally, if we write $\Pi_0$ as an induction like in \S\ref{subsubsec:discreteHermitianArthur}, we can consider the Levi subgroup of $G_n$
\begin{equation*}
    L_\pi=\prod_{i=1}^s G_{2 m_i} \times \prod_{j=1}^r G_{n_i}.
\end{equation*}
It contains $M_P$. We set
\begin{equation}
\label{eq:discrete_SPI_defi}
    S_{\Pi}:=S_{\Pi_0}=2^{\dim(\fa^*_{L_\pi})-\dim \fa_{P}^{L_\pi,*}}.
\end{equation}

\subsubsection{$(G,H,\overline{\mu})$-regularity}

Let $\Pi=\Pi_n \boxtimes \Pi_n'$ be an automorphic representation of $G$ of the form $\Pi_n=I_{P_n}^{G_n} (\pi_1 \boxtimes \cdots \boxtimes \pi_r)$ and $\Pi_n'=I_{P_n'}^{G_n} (\pi_1'\boxtimes \cdots \boxtimes \pi_{r'}')$, for some parabolic subgroups $P_n,P_n' \subset G_n$ and $\pi_1,\hdots,\pi_r,\pi_1',\hdots,\pi_{r'}'$ some cuspidal representations of smaller $G_l$'s. We shall say that
\begin{itemize}
    \item $\Pi$ is \emph{$G$-regular} if all the $\pi_1, \hdots, \pi_r$ are distinct and if all the $\pi_1', \hdots, \pi_{r'}'$ are distinct;
    \item $\Pi$ is \emph{$(H,\overline{\mu})$-regular} if for all $1 \leq i \leq r$ and $1\leq j \leq r'$ the representation $\overline{\mu} \pi_i$ is not isomorphic to the
contragredient of $\pi'_j$.
\end{itemize}
If $\Pi$ is both $G$ and $(H,\overline{\mu})$ regular, then it is said to be \emph{$(G,H,\overline{\mu})$-regular Arthur parameter}.

If both $\Pi_n$ and $\Pi_n'$ are regular Hermitian Arthur parameters, then $\Pi$ is said to be \emph{regular Hermitian}. In that case, define $i \fa_\Pi^* = i \fa_{\Pi_n} \oplus i \fa_{\Pi_n'}$ and $L_\pi=L_{\pi_n} \boxtimes L_{\pi_{n}'}$. This depends on the choice of the inducing data.

\begin{remark}  \label{rem:semi_discrete_regular}
Note that if $\Pi$ is regular Hermitian with either $\Pi_n$ or $\Pi_n'$ discrete, then it is necessarily $(G,H,\overline{\mu})$-regular as
otherwise some $L(s,\pi_i,\As^{(-1)^l})$ would have a pole at $s=1$ for
$l=n,n+1$ which is not possible. 
\end{remark}

\subsubsection{Conditions on cuspidal data}

We now translate these conditions on cuspidal data. Let $\chi \in \fX(G)$ be a cuspidal datum of $G$, represented by $(M_P,\pi)$. Set $\Pi:=I_P^G \pi$. Then we say that 
\begin{itemize}
    \item $\chi$ is \emph{$(G,H,\overline{\mu})$-regular} if $\Pi$ is a $(G,H,\overline{\mu})$-regular Arthur parameter;
    \item $\chi$ is Hermitian if $\Pi$ is a regular Hermitian Arthur parameter.
\end{itemize}
If $\chi$ is both, then it will be a \emph{$(G,H,\overline{\mu})$-regular Hermitian} cuspidal datum. In this case, we have the Levi subgroup $L_\pi$ and we will always assume that $P$ is of the shape described in \S\ref{subsec:reg_herm_art} so that $M_P \subset L_\pi$.

\subsection{Regularity conditions on cuspidal data: $\U$-side}
\label{sec:cusp_data_U}

Let $\chi_ = (\chi_V, \chi_V') \in \mathfrak{X}(\U_V)$ be a cuspidal datum
represented by $(M_Q, \sigma)$ where $M_Q$ is a Levi factor of a standard parabolic subgroup $Q$ of $\U_V$ and $\sigma \in \Pi_{\cusp}(M_Q)$. We have the following decompositions:
\begin{itemize}
\item $M_Q= M_V \times M'_V$, $M_V = G_{n_1} \times \cdots \times G_{n_s}
    \times \U(V_0)$, $M'_V = G_{n_1'} \times \cdots \times G_{n_{s'}'}
    \times \U(V_0')$;

\item $\sigma = \sigma_V \boxtimes \sigma'_V$, $\sigma_V=\tau_1 \boxtimes
    \cdots \boxtimes \tau_s \boxtimes \sigma_0$, $\sigma'_V=\tau_1' \boxtimes
    \cdots \boxtimes \tau_{s'}' \boxtimes \sigma_0'$;
\end{itemize}
where
\begin{itemize}
\item $V_0$ and $V_0'$ are non-degenerate skew-Hermitian vector spaces of dimension
    $n_0$ and $n_0'$ respectively, $2(n_1+\cdots+n_s)+n_0=
    2(n_1'+\cdots+n_{s'}')+n_0'=n$;

\item $\tau_i$ and $\tau_i'$ are cuspidal automorphic representations of the
    $G_{n_i}$ and $G_{n_i'}$ respectively;

\item $\sigma_0$ and $\sigma_0'$ are cuspidal automorphic representations
    of $\U(V_0)$ and $\U(V_0')$ respectively.
\end{itemize}
We shall say that $\chi$ is \emph{$\U_V$-regular} if $\chi_V$ and $\chi_V'$ are regular (see \S\ref{subsec:reg_cusp_data}), \emph{$(\U'_V,\overline{\mu})$-regular} if $\overline{\mu}\tau_i$ is neither isomorphic to $(\tau_j')^\vee$ nor $(\tau_j')^{\sfc}$, and \emph{$(\U_V,\U'_V,\overline{\mu})$-regular} if it is both. Regularity for $\chi_V$ means that the representations $\tau_1, \hdots, \tau_s, \tau_1^*, \hdots,
\tau_s^*$ are pairwise distinct and the same applies to $\chi_V'$.

\subsection{Base-change}

\subsubsection{Local base-change} Let $v$ be a place of $F$. Let ${}^L G_{n,v}$ and ${}^L \U(V)_v$ be the Langlands dual groups of $G_{n,v}$ and $\U(V)_v$ (seen as groups over $F_v$). We have a base change map $\mathrm{BC} : {}^L \U(V)_v \to {}^L G_{n,v}$ (see e.g. \cites{Mok,KMSW}). Because the local Langlands correspondence is known for general linear groups (by \cite{Langlands} and \cite{HT}) and for unitary groups (by \cite{Langlands}, \cite{Mok} and \cite{KMSW}), to any smooth irreducible representation $\sigma_v$ of $\U(V)(F_v)$ we can assign a smooth irreducible representation $\mathrm{BC}(\sigma_v)$ of $G_n(F_v)$ by functoriality. We extend this notion to representations of $\U_V(F_v)$. Note that $\mathrm{BC}(\sigma_v)$ determines the Langlands parameter of $\sigma_v$. 

If the place $v$ splits in $E$, then $\sigma_v$ is a representation of $\GL_n(F_v)$, and $\mathrm{BC}(\sigma_v)$ is $\sigma_{v} \boxtimes \sigma_{v}^\vee$ as a representation of $G_n(E_v)=\GL_n(F_v)^2$.

If the place $v$ is unramified, the Satake isomorphism induces a morphism of spherical Hecke-algebras 
\begin{equation}
    \label{eq:Hecke_BC}
    \BC : \cH(G(F_v),K_v) \to \cH(\U_V(F_v),K_{V,v}).
\end{equation}

\subsubsection{Weak base-change of automorphic representations}

We go back to the global setting. Let $Q
\subset \U(V)$ be a standard parabolic subgroup with Levi factor $M_Q$ isomorphic to $G_{n_1} \times \hdots \times G_{n_r} \times \U(V_0)$ where $V_0$ is a non-degenerate subspace of $V$. Let $\sigma \in \Pi_{\cusp}(M_Q)$. We write $\sigma=\tau_1 \boxtimes \hdots \boxtimes \tau_s \boxtimes \sigma_0$ according to this decomposition. 

We shall say that a regular Hermitian Arthur parameter $\Pi$ of $G_n$ is a \emph{weak base change} of $(Q,\sigma)$ if there exist a parabolic subgroup $P$ of $G_n$ with Levi factor $M_P=G_{n_1}^2 \times \hdots \times G_{n_s}^2 \times G_{n_0}$, a discrete Hermitian Arthur parameter $\Pi_0$ of $G_{n_0}$, and $\tau_i \in \Pi_{\cusp}(G_{n_i})$ for $1 \leq i \leq s$ such that
\begin{itemize}
    \item $\Pi$ is isomorphic to $I_P^{G_n}((\tau_1 \boxtimes \tau_1^*) \boxtimes \hdots \boxtimes (\tau_s \boxtimes \tau_s^*) \boxtimes \Pi_0)$,
    \item for almost all places $v$ of $F$ that split in $E$, the local component $\Pi_{0,v}$ is the split local base change of
        $\sigma_{0,v}$.
\end{itemize}
This implies that $\Pi_0$ is the discrete component of $\Pi$. If the second condition is satisfied for all place $v$ in a set $\tT$ of $F$, we shall say that $\Pi$ is a $\tT$ weak base change of $(Q,\sigma)$. We naturally extend weak base change to product of unitary groups. In particular, we also have the cuspidal representation $\sigma_0$ in this setting. 

It follows from the definition that if $\Pi$ is a $(G,H,\overline{\mu})$-regular Hermitian Arthur parameter of $G$ that is a weak base change of $(Q,\sigma)$, then $(Q,\sigma)$ represents a $(\U_V,\U_V',\overline{\mu})$-regular cuspidal datum.

\begin{remark}
    If we assume the results of \cite{Mok} and \cite{KMSW}, then by \cite{Ram}, $(G,H,\overline{\mu})$-regular Hermitian Arthur parameters $\Pi$ are exactly the (strong) base-changes of pairs $(Q,\sigma)$ that represent $(\U_V,\U_V',\overline{\mu})$-regular cuspidal data such that $\BC(\sigma_0)$ is generic.
\end{remark}

\section{Spectral expansion of Liu's trace formula on general linear groups}
\label{sec:spectral_calculation_GL}

\label{sec:spec_GLN}

In this section, we regularize the spectral contribution of $(G,H,\overline{\mu})$-regular Hermitian cuspidal data to the $G$-side of Liu's relative trace formula.

\subsection[The coarse spectral expansion on $\GL_n \times \GL_{n+1}$]{The coarse spectral expansion of Liu's trace formula on the \texorpdfstring{$G$}{G}-side}
\label{sec:coarse_spectral_Liu}

We first recall the main results of \cite{BLX1} on the coarse regularization of the RTF.

\subsubsection{Theta series}
\label{subsec:theta_series}

We define an automorphic character of $G'(\bA)$ by
\begin{equation}
    \label{eq:eta_defi}
          \eta_{n+1}(g_1, g_2) = \eta(\det g_1 g_2)^{n+1}, \quad (g_1,g_2) \in G'(\bA).
    \end{equation}
We define a representation $\mathrm{R}_{\overline{\mu}}$ of $H(\bA)$ realized on $\cS(\bA_{E,n})$ by the rule
\begin{equation*}
(\mathrm{R}_{\overline{\mu}}(h) \Phi)(x)=
\valP{ \det h}_E^{\frac{1}{2}} \overline{\mu}(h)
\Phi(xh), \quad h \in H(\bA_F), \quad x \in \bA_{E,n},
\end{equation*}
where we write $\overline{\mu}(h)$ for $\overline{\mu(\det(h))}$. Beware that $\det h$ is the determinant of $h$ seen as an element of $G_n(\bA)$ and not $G(\bA)$.

For every $\Phi \in \cS(\bA_{E,n})$, we define the Epstein--Eisenstein series of \cite{JS}*{Section~4.1}\begin{equation}
    \label{eq:theta_defi}
    \Theta(h, \Phi) =
    \sum_{x \in E_n} (\mathrm{R}_{\overline{\mu}}(h)\Phi)(x), \quad \Theta'(h, \Phi) =
    \sum_{x \in E_n\backslash\{0\}} (\mathrm{R}_{\overline{\mu}}(h)\Phi)(x), \quad h \in [H].
\end{equation}

    \begin{lemma}
        \label{lem:theta_moderate}
        The series $\Theta(h, \Phi)$ and $\Theta'(h, \Phi)$ are absolutely convergent. There exists $N>0$ such that $\Theta$ and $\Theta'$ define continuous mappings $\cS(\bA_{E,n}) \to \cT_{N}([H])$.
    \end{lemma}

    \begin{proof}
        For every $M>0$, there exists $N'>0$ and $N>0$ such that $\aabs{xh}_{\bA_{E,n}}^{-N'} \ll \aabs{x}_{\bA_{E,n}}^{-M} \aabs{h}^N$. Moreover, by \cite{BP}*{Proposition~A.1.1.(v)}, there exists $M>0$ such that $\sum_{x \in \bA_{E,n}} \aabs{x}_{\bA_{E,n}}^{-M}$ is finite. This concludes because for all $\Phi \in \cS(\bA_{E,n})$ we have $\sup_{x \in \bA_{E,n}} \valP{\Phi(x)}\aabs{x}_{\bA_{E,n}}^{N'} < \infty$ which defines a continuous norm on $\cS(\bA_{E,n})$.
    \end{proof}

    \subsubsection{Truncated kernels}
    \label{subsec:truncated_kernels_gln}

    For every $f_+ \in \cS(G_+(\bA))$, an automorphic kernel $K_{f_+}$ is defined in \cite[Section~5.3]{BLX1}. It is a function on $H(\bA) \times G(\bA)$, and defines an element in $K_{f_+} \in \cT([H\times G])$.
    
    Let $\chi$ be a cuspidal datum of $G$. In \cite[Section~5.3]{BLX1}, a kernel $K_{f_+,\chi}$ is also defined. It can be described as follows. First, for $f_+=f \otimes \Phi$ we set
    \begin{equation}
        \label{eq:kernel_f_+}
        K_{f \otimes \Phi,\chi}(h,g)=K_{f,\chi}(h,g) \Theta(h,g), \quad (h,g) \in H(\bA) \times G(\bA).
    \end{equation}
    In \cite[Section~5.3]{BLX1}, it is shown that the assignment $f \otimes \Phi \mapsto K_{f \otimes \Phi,\chi}$ extends by continuity to a map $f_+ \in \cS(G_+(\bA)) \to K_{f_+,\chi} \in \cT([H \times G])$. Alternatively, $K_{f_+,\chi}$ can be realized as a kernel of a certain operator by the process explained in \cite[Section~5.3]{BLX1}.

    By a truncation parameter, we mean an element $T \in \fa_{0}^{*}$. We shall say that $T$ is sufficiently positive if for a large enough $M$ (depending on the context) we have $\min_{\alpha \in \Delta_0^G} \langle \alpha,T \rangle >M$. This also defines $\lim_{T \to \infty}$. We take $\aabs{\cdot}$ a norm on $\fa_{0}$.

    In \cite[Section~5.3]{BLX1} and \cite[Section~5.5]{BLX1}, three truncated kernels $K_{f_+}^T$, $K_{f_+,\chi}^T$ and $k_{f_+,\chi}^T$ are defined. They are functions on $H(\bA) \times G'(\bA)$. We will also need two variants of Arthur truncation functions $F^{G'_{n+1}}(\cdot,T)$ and $F^{G_{n+1}}(\cdot,T)$. These are functions on $G'_n$ and $H$ respectively. The precise definitions are not relevant to us, we only need to know that for all $g' \in G'_n(\bA)$ and $h \in H(\bA)$ we have $\lim_{T \to \infty} F^{G'_{n+1}}(g',T)=\lim_{T \to \infty} F^{G_{n+1}}(h,T)=1$.
    
    We sum up the properties of the truncated kernels and of the associated regularized periods in the following theorem. 

    \begin{theorem}
        \label{thm:truncated_prop}
        The following properties hold for $T$ positive enough.
        \begin{enumerate}
            \item \cite[Theorem~5.8]{BLX1} The integrals 
            \begin{equation*}
                I^T(f_+)=\int_{[H]} \int_{[G']} K_{f_+}^T(h,g') \eta_{n+1}(g') \rd g' \rd h, \quad I^T_\chi(f_+)=\int_{[H]} \int_{[G']} K_{f_+,\chi}^T(h,g') \eta_{n+1}(g')\rd g'\rd h,
            \end{equation*}
            are absolutely convergent. Moreover, for every $r>0$ there exists a continuous semi-norm $\aabs{\cdot}$ on $\cS(G_+(\bA))$ such that for all $T$ positive enough
            \begin{equation*}
                \sum_{\chi \in \fX(G)} \int_{[G']} \int_{[H]} \valP{K_{f_+,\chi}^T(h,g')-F^{G'_{n+1}}(g_1') K_{f_+,\chi}(h,g') } \rd h\rd g' \leq \aabs{f_+} e^{-r \aabs{T}}.
            \end{equation*}
            \item \cite[Theorem~5.10]{BLX1} For $T$ sufficiently positive, the functions $I^T(f_+)$ and $I_{\chi}^T(f_+)$ are the
            restrictions of exponential-polynomial functions whose purely polynomial parts
            are constants. We denote them by $I(f_+)$ and $I_{\chi}(f_+)$ respectively. 
            Then $I$ and $I_\chi$ are continuous as distributions on $\cS(G_+(\bA))$ 
            and for any $f_+ \in \cS(G_+(\bA))$ we have
                \[
                \sum_{\chi \in \fX(G)} I_{\chi}(f_+) = I(f_+),
                \]
            where the sum is absolutely convergent.
            \item \cite[Proposition~5.13]{BLX1} The integral 
            \begin{equation*}
                i^T_\chi(f_+)=\int_{[H]} \int_{[G']} k_{f_+}^T(h,g') \eta_{n+1}(g')\rd g'\rd h
            \end{equation*}
            is absolutely convergent. Moreover for every $r>0$ there exists a continuous semi-norm $\aabs{\cdot}$ on $\cS(G_+(\bA))$ such that for all $T$ positive enough
            \begin{equation*}
                \sum_{\chi \in \fX(G)}\int_{[G']} \int_{[H]} \valP{k_{f_+,\chi}^T(h,g')-F^{G_{n+1}}(h) K_{f_+,\chi}(h,g') } \rd h\rd g' \leq \aabs{f_+} e^{-r \aabs{T}}.
            \end{equation*}
            \item \cite[Section~5.14]{BLX1} For $T$ sufficiently positive, the function $i^T_\chi(f_+)$ is the restriction of an exponential-polynomial function whose purely polynomial part
            is constant and equal to $I_\chi(f_+)$.
        \end{enumerate}
    \end{theorem}

    The second point of the theorem is the \emph{coarse spectral expansion} of Liu's trace formula on general linear groups.

\subsection{Extension of the Rankin--Selberg period}

In this section we consider the \emph{Rankin--Selberg} period, i.e.
\begin{equation}
    \label{eq:RS_period}
    \cP_{H}(f,\Phi)=\int_{[H]} f(h )\Theta'(h,\Phi)\rd h, \quad f \in \cS([G]), \quad \Phi \in \cS(\bA_{E,n}).
\end{equation}
This period plays an important role in \cite{JPSS83}, where the case where $f$ is a cuspidal automorphic form is considered. Our goal is to further extend $\cP_H(\cdot,\Phi)$ to Eisenstein series induced from $(G,H,\overline{\mu})$-regular Arthur parameters. The end result is Corollary~\ref{cor:RS_extension_Eisenstein}. We warn the reader that although we omit $\overline{\mu}$ in the notation $\cP_H$, there really is a dependence.

\subsubsection{Whittaker functionals}
\label{subsubsec:whittaker}

 We have two characters of $[N_{n}]$ given by
    \begin{equation}    \label{eq:psiN}
    \psi_{n,\pm}= \psi^E
    \left( \pm (-1)^n \sum_{i=1}^{n-1} u_{i,i+1} \right), \quad u \in [N_{n}],
    \end{equation}
and we define a character of $[N]$ by $\psi_N = \psi_{n,+} \boxtimes
\psi_{n,-}$. This is a generic character trivial on $[N_{H}]$. 

    For $f \in \mathcal{T}([G])$ set
    \begin{equation}    \label{eq:Whittaker}
    W^\psi(f,g):= \int_{[N]} f(ng) \overline{\psi_N}(n) \rd n, \quad g \in G(\bA).
    \end{equation}
    We write $W^{\overline{\psi}}$ if we integrate againt $\psi_N$ instead. 
    
    \subsubsection{Zeta functions}
    followsr $\Phi \in \cS(\bA_{E,n})$, we set
    \begin{equation}    \label{eq:Z_phi}
    Z(f,\Phi,\overline{\mu}, s)=
    \int_{N_{H}(\bA) \backslash H(\bA)} W^\psi(f,h)
    \valP{\det h}_E^{s} (\mathrm{R}_{\overline{\mu}}(h)\Phi)(e_n) \rd h,
    \end{equation}
    for every $s \in \C$ such that this expression converges absolutely. This is the \emph{Zeta function} associated to $f$ and $\Phi$ defined in \cite{JPSS83}

If $c$ is a real number, set $\mathcal{H}_{>c} := \{ s
\in \C \; | \; \Re(s)>c \}$, and if $c<C$, set $\mathcal{H}_{]c,C[} := \{ s
\in \C \; | \; c<\Re(s)<C \}$.

\begin{lemma}   \label{lem:convTau}
Let $N \geq 0$. There exists $c_N>0$ such that the following holds.
\begin{itemize}
    \item For every $f \in \mathcal{T}_N([G])$, $\Phi \in \cS(\bA_{E,n})$
        and $s \in \mathcal{H}_{>c_N}$, the defining integral of
        $Z(f,\Phi,\overline{\mu}, s)$ converges absolutely.

    \item For every $s \in \mathcal{H}_{>c_N}$, the bilinear form $(f,\Phi)
        \in \mathcal{T}_N([G]) \times \cS(\bA_{E,n})\mapsto
        Z(f,\Phi,\overline{\mu}, s)$ is separately continuous.

    \item For every $f \in \mathcal{T}_N([G])$ and $\Phi \in
        \cS(\bA_{E,n})$, the function $s \in \mathcal{H}_{>c_N} \mapsto
        Z(f,\Phi,\overline{\mu}, s)$ is holomorphic and bounded
        in vertical strips.
\end{itemize}
\end{lemma}

\begin{proof}
The proof is an immediate adaptation of~\cite{BPCZ}*{Lemma~7.1.1.1} using~\cite{BPCZ}*{Lemma~2.6.1.1.2} and the fact that, as $\Phi$ is Schwartz, for all $N>0$ there is a seminorm $\aabs{\cdot}$ such that
    \[
    \valP{\Phi(e_n t)} \leq \aabs{t_n}^{-N}_{\bA_E} \aabs{\Phi}, \quad t=(t_1, \hdots, t_n) \in T_n(\bA).
    \]
\end{proof}

\subsubsection{Rankin--Selberg periods}
\label{subsec:RS_FJ}

For every $f \in \cS([G])$ and $\Phi \in \cS(\bA_{E,n})$, we set
    \begin{equation}    \label{eq:Zn}
    Z_{n}(f,\Phi, \overline{\mu}, s) = \int_{[H]} f(h)
    \Theta'(h, \Phi)
    \valP{\det h}^{s} \rd h.
    \end{equation}
Since $f \in \cS([G])$ and $\valP{\Theta'(h, \Phi)} \ll \aabs{h}_H^{N}$ by Lemma~\ref{lem:theta_moderate}, we conclude by \cite{BP}*{Proposition~A.1.1~(vi)} that the
integral converges absolutely for all $s \in \C$ and yields an entire
function bounded on vertical strips. At $s=0$, we obtain $\cP_{H}(f,\Phi)$ the \emph{Rankin--Selberg} from \eqref{eq:RS_period}

\begin{remark}  \label{remark:mirabolic}
Our definition of $\Theta'(\cdot,\Phi)$ in \S\ref{subsec:theta_series} mirrors the Epstein--Eisenstein series
of~\cite{JS}*{Section~4}. We could have also considered $\Theta(\cdot,\Phi)$ which differs by the additional term $\overline{\mu}(h)\valP{\det h}^{\frac{1}{2}} \Phi(0)$. When $f \in \cS_{\chi}([G])$ with $\chi$
a $(H, \overline{\mu})$-regular cuspidal datum, we have for all $s \in \C$
    \[
    \int_{[H]} f(h) \overline{\mu}(h) \valP{\det h}^{s+\frac{1}{2}} \rd h=0
    \]
It follows that for such an $f$ the two definitions coincide.
\end{remark}

\subsubsection{Extension of Rankin--Selberg integrals}

Our main result on Rankin--Selberg periods is the following.

\begin{theorem} \label{thm:RS}
Let $\chi \in \mathfrak{X}(G)$ be a $(G,H,\overline{\mu})$-regular cuspidal datum.

\begin{enumerate}
    \item For every $f \in \mathcal{T}_\chi([G])$ and $\Phi \in
        \cS(\bA_{E,n})$, the function $s \mapsto
        Z(f,\Phi, \overline{\mu}, s)$, a priori defined on some
        right half-plane, extends to an entire function on $\C$ of finite order (independent from $f$ and $\Phi$) in vertical strips.

    \item For all $s \in \C$ and $\Phi \in \cS(\bA_{E,n})$, the restriction to $\cS_\chi([G])$
        of 
        \begin{equation*}
        f \in \cS([G])  \mapsto Z_{n}(f,\Phi,\overline{\mu}, s),
        \end{equation*}
        extends continuously to a linear form on
        $\mathcal{T}_\chi([G])$.
\end{enumerate}
Moreover, if we keep the same notation for these extensions, then for all $f \in
\mathcal{T}_{\chi}([G])$, $\Phi \in \cS(\bA_{E,n})$ and $s \in \C$ we have
    \begin{equation}    \label{eq:RS_equality}
    Z(f, \Phi, \overline{\mu}, s)=Z_n(f, \Phi,\overline{\mu}, s).
    \end{equation}
\end{theorem}

The key step in the proof is the following lemma which is the variant of the classical Rankin--Selberg unfolding argument.

\begin{lemma}
    \label{lem:RS_Schwartz}
    For fixed $f \in \cS_{\chi}([G])$ and $\Phi \in \cS(\bA_{E, n})$, for every $\Re s\gg0$ we have
    \begin{equation}   
        \label{eq:RS_equality_Schwartz}
        Z(f, \Phi, \overline{\mu}, s)=Z_n(f, \Phi,\overline{\mu}, s).
        \end{equation} 
    In particular, the function $s \mapsto Z(f,
    \Phi, \overline{\mu}, s)$ has analytic continuation to $s \in \C$.
\end{lemma}

\begin{proof}
    Let us introduce the following notation. Let $1 \leq r \leq n$ and denote by
$N_{r,n}$ the unipotent radical of the standard parabolic subgroup of $G_n$
with Levi factor $(G_r) \times (G_1)^{n-r}$. Set $N_r^G:=N_{r,n} \times
N_{r,n}$ and $N_{r,H} := N_r^G \cap H$. Let $P_r$ be the mirabolic subgroup of
$G_r$ whose last row is $e_r$ the last vector in the canonical basis of $E_r$.

For $f \in \cS([G])$ we define
    \[
    f_{N_r^G, \psi}(g) =
    \int_{[N_r^G]} f(ug) \overline{\psi_N}(u) \rd u,\quad g \in G(\bA).
    \]
For $s \in \C$ and $\Phi \in \cS(\bA_{E,n})$ we set
    \begin{equation*}
    Z_{r}(f,\Phi, \overline{\mu}, s) =
    \int_{P_r(F) N_{r,H}(\bA) \backslash H(\bA)}
    f_{N_r^G, \psi}(h) (\mathrm{R}_{\overline{\mu}}(h)\Phi)(e_n) |\det h|_E^{s}
    \rd h.
    \end{equation*}
For $r=1$ we get $Z_{1}(f,\Phi, \overline{\mu}, s)=Z(f,
\Phi, \overline{\mu}, s)$. The same argument as in the proof of Lemma~\ref{lem:convTau}
shows that for every $1 \leq r \leq n$ there exists $c_r>0$ such that the
following assertions hold.
\begin{itemize}
    \item For all $f \in \cS([G])$ and $\Phi \in \cS(\bA_{E,n})$, the
        expression defining $Z_{r}(f,\Phi,\overline{\mu},s)$ converges
        absolutely for $s \in \mathcal{H}_{>c_r}$.

    \item For all $s \in \mathcal{H}_{>c_r}$, $(f, \Phi) \mapsto
        Z_{r}(f,\Phi,\overline{\mu},s)$ is separately continuous.

    \item For every $f \in \cS([G])$ and $\Phi \in \cS(\bA_{E,n})$, the
        function $s \in \mathcal{H}_{>c_r} \mapsto
        Z_{r}(f,\Phi,\overline{\mu},s)$ is holomorphic and bounded in
        vertical strips.
\end{itemize}

The desired identity~\eqref{eq:RS_equality_Schwartz} then follows from the following claim: for all $1
\leq r \leq n-1$, for all $f \in \cS_{\chi}([G])$ and $\Phi \in \cS(\bA_{E,n})$
there exists $c>0$ such that for all $s \in \mathcal{H}_{>c}$ we have
    \begin{equation}    \label{eq:devissage}
    Z_{r+1}(f,\Phi, \overline{\mu}, s)=Z_{r}(f,\Phi,\overline{\mu}, s).
    \end{equation}

We now prove identity~\eqref{eq:devissage}. First for $r$, $f$ and $\Phi$, we
pick $c>0$ such that both $Z_{r+1}(f,\Phi,\overline{\mu},s)$ and
$Z_{r}(f,\Phi,\overline{\mu},s)$ are well-defined for $s \in
\mathcal{H}_{>c}$. Let $U_{r}$ be the unipotent radical of $P_r$. Then
$N_{r,n} = U_{r} N_{r+1,n}$. Set $U_{r}^G = U_{r} \times U_{r} \leq
G$ and $U_{r,H} = (U_{r}^G)_H \leq H$. We then compute 
    \begin{equation*}
    Z_{r+1}(f,\Phi,\overline{\mu}, s)
    = \int_{G_r(F) N_{r,H}(\bA) \backslash H(\bA)}
    (\mathrm{R}_{\overline{\mu}}(h)\Phi)(e_n)|\det h|_E^s  \int_{[U_{r,H}]}
    f_{N_{r+1}^G, \psi}(uh)  \rd u \rd h.
    \end{equation*}
By Fourier inversion on the compact abelian group $U_{r}^G(F)
U_{r,H}(\bA) \backslash U^G_{r}(\bA)$ we get
\begin{equation*}
    \int_{[U_{r,H}]}
    f_{N_{r+1}^G, \psi}(uh)  \rd u =\sum_{\gamma \in P_r(F) \backslash G_r(F) } f_{N_{r}^G, \psi}(\gamma h) + \int_{[U_{r}^G]} f_{N_{r+1}^G, \psi}(u h)\rd u.
\end{equation*}
Therefore we have
    \begin{equation*}
    Z_{r+1}(f,\Phi, \overline{\mu}, s)=
    Z_{r}(f,\Phi,\overline{\mu},s)+F_r(s),
    \end{equation*}
where
    \begin{equation*}
    F_r(s)=\int_{G_r(F) N_{r,H}(\bA) \backslash H(\bA)}
    (\mathrm{R}_{\overline{\mu}}(h)\Phi)(e_n) |\det h|_E^s  \int_{[U_{r}^G]}
    f_{N^G_{r+1},\psi}(uh)  \rd u \rd h.
    \end{equation*}

As our functions are holomorphic for $\Re(s) \gg 0$, it remains to show that
$F_r$ vanishes on some right half-plane of $\C$. Let $Q_r$ be the standard
parabolic subgroup of $G$ with Levi factor $L_r=(G_r \times G_{n-r})
\times (G_r \times G_{n-r})$, and set $Q_{r,H}=Q_r \cap H$. Let $\chi^L$ be the inverse
image of $\chi$ in $\mathfrak{X}(L_r)$. We write any element in $\fX(L_r)$
as $(\chi_1, \chi_2)$ where $\chi_1 \in \fX(G_{r}^2)$ and $\chi_2 \in
\fX(G_{n-r}^2)$. Then by~\cite{BPCZ}*{Section~2.9.6.10} we have the
decomposition
    \begin{equation}    \label{eq:decompo_cuspi}
    \cC_{\chi^L}([L_r])=\bigoplus_{(\chi_1,\chi_2) \in \chi^L}
    \mathcal{C}_{\chi_1}([G_r^2])
    \otimeshat \mathcal{C}_{\chi_2}([G_{n-r}^2]).
    \end{equation}
    Define for all $s \in \C$ and $k \in K_n$
    \begin{equation*}
    f_{Q_r,k,s} =
    \delta_{Q_r}^{-\frac{1}{2}+\frac{2s+1}{4n-4r}}
    \mathrm{R}(k) f_{Q_r}\Big|_{[L_r]},
    \end{equation*}
where $f_{Q_r}$ is the constant term of $f$ with respect to $Q_r$. As $\chi$
is $G$-regular,~\cite{BPCZ}*{Corollary~2.9.7.2} applies and for $\Re s>0$ and $k
\in K_n$ we have $(f_{Q_r,k,s})\Big|_{[L_r]} \in \mathcal{C}_{\chi^L}([L_r])$. Finally, for
$k \in K_n$, set $\Phi_{k,n-r}=(\mathrm{R}(k)\Phi)\Big|_{ \{0\} \times
\bA_{E,{n-r}}} \in \cS( \bA_{E,n-r})$. Then using the Iwasawa decomposition
$H(\bA)=Q_{r,H}(\bA)K_n$, we get just as in the proof
of~\cite{BPCZ}*{Proposition~7.2.0.2}:
    \begin{equation*}
    F_r(s)= \int_{K_n} \left( P_{G_r^\Delta}^{\overline{\mu}} \otimeshat
    \cZ_{n-r}
    \left( \cdot , \Phi_{k,n-r}, \overline{\mu},  s+2r\frac{2s+1}{4n-4r}\right)\right)
    (f_{Q_r,k,s})\rd k,
    \end{equation*}
where
\begin{itemize}
    \item $P_{G_r^\Delta}^{\overline{\mu}}$ stands for the period integral over
        the diagonal subgroup $G_r$ of $G_r^2$, against the
        character $\overline{\mu}$, which defines a continuous linear form on $\cC([G_r^2])$ by Lemma~\ref{lem:integration},

    \item $\cZ_{n-r}$ is $Z_{\psi^r}$ for
        $G_{n-r}$, with $\psi^r=\psi((-1)^{r}.)$. By Lemma~\ref{lem:convTau}, for every $\Phi' \in \cS(\bA_{E,n-r})$ it is a continuous linear form on $\mathcal{T}([G_{n-r}^2])$ provided
        that $s$ lies in some right half-plane of $\C$ (that can be chosen independently of $\Phi'$),

    \item the completed tensor product is taken relatively to the
        decomposition (\ref{eq:decompo_cuspi}), recalling the continuous inclusion $\mathcal{C}([G^2_{n-r}]) \subset
        \mathcal{T}([G^2_{n-r}])$ from Lemma~\ref{lem:functional_analysis}~(4).
\end{itemize}

Since $\chi$ is $(H, \overline{\mu})$-regular, for every preimage $(\chi_1,\chi_2)
\in \mathfrak{X}(G_r^2) \times \mathfrak{X}(G_{n-r}^2)$ with
$\chi_1=(\chi_1',\chi_1'')$ we have $\overline{\mu} \chi_1' \neq  \chi_1''^{\vee}$. Thus,
$P_{G_r^\Delta}^{\overline{\mu}}$ vanishes identically on
$\mathcal{C}_{\chi_1}([G_r^2])$, which implies that $F_r(s)=0$ whenever
$\Re(s) \gg 0$. This proves the identity~\eqref{eq:devissage} and hence
finishes the proof of the lemma.
\end{proof}

We can now end the proof of Theorem~\ref{thm:RS}

\begin{proof}[Proof of Theorem~\ref{thm:RS}]
We fix $\Phi$ and consider $Z_n$ and
$Z$ as (families of) linear forms in $f$. Let $w_0$ be the longest element in the Weyl group of $G$, which we identify with an element in $G(F)$. For $f \in \cT([G])$, set
$\widetilde{f}(g) = f(\tp{g}^{-1}) = f(w_0 \tp{g}^{-1}) \in \cT([G])$.

For $s \in \C$, we consider the linear form on $\cT([G])$ given by
    \begin{equation}
    \label{eq:Z_tilde_defi}
         f \mapsto \widetilde{Z}(f, \Phi, \overline{\mu}, s)
    := Z(\widetilde{f}, \widehat{\Phi}, \mu ,s), 
    \end{equation}
where we denote by $\widehat{\Phi}$ the Fourier transform of $\Phi$ defined as  
    \[
    \widehat{\Phi}(y) = \int_{\bA_{E, n}}
    \Phi(x) \psi_E(x \tp{y}) \rd x.
    \]
Note that if $f \in \cT_{\chi}([G])$ then so does $\widetilde{f}$. By
Lemma~\ref{lem:convTau} we see that $Z$ and
$\widetilde{Z}$ enjoy the following properties.
\begin{enumerate}
    \item For all $N>0$ there exists $C_N>0$ such that for all $s \in
        \mathcal{H}_{> C_N}$, $Z(\cdot, \Phi, \overline{\mu}, s)$
        and $\widetilde{Z}(\cdot, \Phi, \overline{\mu}, s)$ are
        continuous linear functionals on $\mathcal{T}_{N,\chi}([G])$.

    \item For all $N>0$ and $f \in \mathcal{T}_{N,\chi}([G])$, $s \mapsto
        Z(f,\Phi, \overline{\mu}, s)$ and $s \mapsto
        \widetilde{Z}(f,\Phi, \overline{\mu}, s)$ are holomorphic
        functions on $\mathcal{H}_{>C_N}$, bounded on vertical strips.

    \item By Lemma~\ref{lem:RS_Schwartz}, for $f \in
        S_\chi([G])$ the functions $s \mapsto Z(f,
        \Phi, \overline{\mu}, s)$ and $s \mapsto \widetilde{Z}(f,
        \Phi, \overline{\mu}, s)$ have analytic continuations to all $s \in \C$.

    \item By Remark~\ref{remark:mirabolic} we may replace $\Theta'$ by
        $\Theta$ and use the Poisson summation formula to obtain for any $f \in \cS_\chi([G])$ the functional equation 
            \[
            Z_n(f, \Phi, \overline{\mu}, s) =
            Z_n(\widetilde{f}, \widehat{\Phi}, \mu, -s).
            \]
        Therefore it follows from Lemma~\ref{lem:RS_Schwartz} that
           \begin{equation}
           \label{eq:functional_equation_Z}
                Z(f,\Phi, \overline{\mu}, s) =
            \widetilde{Z}(f, \Phi, \overline{\mu}, -s).
           \end{equation}
\end{enumerate}

Let $N'>0$. By Lemma~\ref{lem:functional_analysis}~(2), there exist $N>0$ such that we have a continuous inclusion $L^2_{-N'}([G])^\infty \subset \cT_N([G])$. By Proposition~\ref{prop:density_schwartz}, the space $\cS_\chi([G])$ is dense in $L^2_{-N',\chi}([G])^\infty$. It follows that the two functionals $f \mapsto Z(f,\Phi,
\overline{\mu}, s)$ and $f \mapsto \widetilde{Z}_\psi(f,\Phi, \overline{\mu}, s)$ defined on $S_\chi([G])$ satisfy the hypotheses of~\cite{BPCZ}*{Corollary~A.0.11.2}. This corollary shows the analytic continuation of $Z$ to $\C$ for all $f \in L^2_{-N',\chi}([G])^\infty$, and the continuous extension of $Z_n$ to $L^2_{-N',\chi}([G])^\infty$. By Lemma~\ref{lem:functional_analysis}~(1), this remains true if we replace $L^2_{-N',\chi}([G])^\infty$ by $\cT_\chi([G])$. Since
the equality~\eqref{eq:RS_equality} is true for $f \in \cS_\chi([G])$ and $\Re
(s) \gg 0$, it is true for $f \in \mathcal{T}_\chi([G])$ by density (Proposition~\ref{prop:density_schwartz}), and
finally for all $s$ by analytic continuation. That the order of $Z(f,\Phi,\overline{\mu},\cdot)$ is independent from $f$ and $\Phi$ follows from the third point of Lemma~\ref{lem:convTau} and the proof of~\cite{BPCZ}*{Corollary~A.0.11.2}. This concludes the proof.
\end{proof}

The linear form $Z$ given by Theorem~\ref{thm:RS} depends on a fixed Schwartz function $\Phi$, but it is possible to obtain uniform bounds when varying $\Phi$ as asserted by the following lemma.

\begin{lemma}
\label{lem:Z_continuous_Phi}
    Let $N>0$. There exist continuous semi-norms $\aabs{.}_N$ and $\aabs{.}$ on $\cT_{N,\chi}([G])$ and $\cS(\bA_{E,n})$ respectively such that
    \begin{equation*}
    \valP{Z(f,\Phi,\overline{\mu}, 0)} \ll \aabs{f}_N \aabs{\Phi}, \quad f \in \cT_{N,\chi}([G]), \quad \Phi \in \cS(\bA_{E,n}).
\end{equation*}
\end{lemma}

\begin{proof}
    This follows easily from the explicit formula for the analytic continuation of the Zeta function $Z(f,\Phi,\overline{\mu},s)$ in terms of integrals in \cite{BPCZ}*{Lemma~A.0.10.1}, and from Lemma~\ref{lem:convTau}.
\end{proof}

We now define the extension of the Rankin--Selberg period to $(G,H,\overline{\mu})$-regular Eisenstein series.

\begin{coro}
    \label{cor:RS_extension_Eisenstein}
    Let $\chi \in \fX(G)$ be a $(G,H,\overline{\mu})$-regular cuspidal datum represented by $(M_P,\pi)$. Then the restriction of the Rankin--Selberg period $\cP_{H}(\cdot, \Phi)$ to $\cS_{\chi}([G])$ extends continuously to $\cT_{\chi}([G])$. We denote this extension by $\cP_{H}^*( \cdot ,\Phi)$. This regularized period satisfies
    \begin{equation*}
        \cP_H^*(E(\varphi,\lambda),\Phi)=Z(E(\varphi,\lambda),\Phi,\overline{\mu},0), \quad \varphi \in \cA_{P,\pi}(G), \quad \lambda \in i \fa_P^*, \quad \Phi \in \cS(\bA_{E,n}).
    \end{equation*}
\end{coro}

\begin{proof}
    This is a direct application of Theorem~\ref{thm:RS} because $E(\varphi,\lambda) \in \cT_\chi([G])$. Note that both sides are defined by analytic continuation.
\end{proof}

\subsection{Extension of the Flicker--Rallis period}

In this section, we fix $\chi \in \mathfrak{X}(G)$ be a $G$-regular cuspidal datum, represented by
$(M_P,\pi)$. Put $\Pi = I_{P}^{G} \pi$ and for all $\lambda \in i
\mathfrak{a}_{P}^{G,*}$ set $\Pi_\lambda = I_{P}^{G} \pi_\lambda$. We recall the main result of \cite{BPC22} on the Flicker--Rallis period.

\subsubsection{Flicker--Rallis period} The Flicker--Rallis period is defined by 
    \begin{equation}    \label{eq:FR_period}
    \cP_{G'}(f) = \int_{[G']} f(g') \eta_{n+1}(g') \rd g', \quad f \in \cS([G]).
    \end{equation}
This integral is absolutely convergent and defines a continuous functional on
$\cS([G])$ by~\cite{BPCZ}*{Theorem~6.2.6.1}. 

\subsubsection{Spectral expansion of the period}

Let $J$ be a level of $G$. If $f \in \cS([G])^J$ and $\lambda \in i \mathfrak{a}_P^*$ we define 
    \begin{equation}    \label{eq:sumEisenstein}
    f_{\Pi_\lambda} = \sum_{\varphi \in \cB_{P,\pi}(J)}
    \langle f, E(\cdot,\varphi,\lambda) \rangle
    E(\cdot, \varphi, \lambda).
    \end{equation}
By Proposition~\ref{prop:strong_K_basis} and \cite{Lap}*{Theorem~2.2}, there exists an $N>0$ such that this series converges absolutely in $\mathcal{T}_N([G])$. Let $\mathcal{W}(\Pi_\lambda,\overline{\psi}_N)$ 
be the space of Whittaker functions of $\Pi_\lambda$ for $\overline{\psi}_N$ so that $W^{\overline{\psi}}(f_{\Pi_\lambda}) \in \mathcal{W}(\Pi_\lambda,\overline{\psi}_N)$.

For $\tS$ a sufficiently large finite set of places of
$F$ and $\lambda \in i \mathfrak{a}_P^*$, put
    \begin{equation}    \label{eq:defi_beta}
    \beta_{\eta}(W)  = (\Delta_{G'}^{\tS,*})^{-1}L^{\tS,*}(1,\Pi,\mathrm{As}^{n+1,n+1})
    \int_{N'(F_\tS) \backslash P_{n,n}'(F_\tS) }W(p_\tS) \eta_{n+1}(p_\tS) \rd p_\tS, \quad
    W \in \mathcal{W}(\Pi_\lambda, \overline{\psi}_N)
    \end{equation}
where we have set $\mathrm{As}^{n+1,n+1}=\mathrm{As}^{(-1)^{n+1}} \oplus \mathrm{As}^{(-1)^{n+1}}$. This
expression is absolutely convergent and independent of $\tS$ as soon as it is
sufficiently large by \cite{Fli}. 

We equip $i \mathfrak{a}_{M_P}^{L_\pi,*}$ with the measure
described in \S\ref{subsubsec:first_measure}. The following is~\cite{BPC22}*{Corollary~6.2.7.1}.

\begin{theorem}    \label{thm:FR_expansion}
If $\chi$ is Hermitian, then the
function $\lambda \in i \mathfrak{a}_{M_P}^{L_\pi, *} \mapsto
\beta_\eta(W^{\overline{\psi}}(f_{\Pi_\lambda}))$ is Schwartz, and the resulting map
    \[
    \cS_{\chi}([G]) \to \cS(i \mathfrak{a}_{M_P}^{L_\pi, *}), \quad
    f \mapsto \left( \lambda \mapsto \beta_{\eta}(W^{\overline{\psi}}(f_{\Pi_\lambda})) \right)
    \]
is continuous. Moreover, we have
    \begin{equation*}
    \cP_{G'}(f)= \left\{ \begin{aligned}
    & 2^{-\dim \mathfrak{a}_{L_\pi}^*} \int_{i \mathfrak{a}_{M_P}^{L_\pi, *}}
    \beta_{\eta}(W^{\overline{\psi}}(f_{\Pi_\lambda})) \rd \lambda,
    && \text{if $\chi$ is Hermitian}; \\
    &0, && \text{otherwise.}
    \end{aligned}\right.
    \end{equation*}
\end{theorem}

\subsection{Expansion of relative characters on the \texorpdfstring{$G$}{G}-side}
\label{sec:expansion_char_gln}

In this section, we let $\chi \in \fX(G)$ be a cuspidal datum of $G$ represented by $(M_P,\pi)$. Let $\lambda \in i
\mathfrak{a}_{P}^{G,*}$. We fix $J$ a level of $G$. We now combine Corollary~\ref{cor:RS_extension_Eisenstein} and Theorem~\ref{thm:FR_expansion} to write the spectral expansion of $I_{\chi}$ in terms of a relative character.

\subsubsection{The relative character \texorpdfstring{$I_{\Pi_\lambda}$}{}}
\label{subsec:I_PI_defi}
Assume that $\chi \in \mathfrak{X}(G)$ is $(G,H,\overline{\mu})$-regular and Hermitian. Thanks to Proposition~\ref{prop:strong_K_basis}, \cite{Lap}*{Theorem~2.2}, \cite{BPCZ}*{Lemma~2.10.1.1}, Theorem~\ref{thm:RS} and Theorem~\ref{thm:FR_expansion}, for every $f \in \cS(G(\bA))^J$ and $\Phi \in \cS(\bA_{E,n})$ we can consider the relative character 

\begin{equation}    \label{eq:IPi}
    I_{\Pi_\lambda}(f \otimes \Phi) =
    \sum_{\varphi \in \mathcal{B}_{P,\pi}(J)}
    Z(E(I_P(f,\lambda)\varphi,\lambda), \Phi, \overline{\mu}, 0)
    \overline{\beta_\eta(W^{\psi}(E(\varphi,\lambda)))}.
    \end{equation}
By continuity, we have the alternative description
\begin{equation}
    \label{eq:I_Pi_lambda_estimate}
        I_{\Pi_\lambda}(f \otimes \Phi)=\cP_{H}^*\left(g \mapsto
        \beta_{\eta}\left(
        W^{\overline{\psi}}\left(K_{f, \chi}(g, \cdot)_{\Pi^\vee_{-\lambda}} \right) \right), \Phi\right).
    \end{equation}
In particular, using Theorem~\ref{thm:RS} and Lemma~\ref{lem:Z_continuous_Phi}, we see that the distribution $I_{\Pi_\lambda}$ extends by continuity to $\cS(G_+(\bA))^J$ and that the linear functional $f_+ \mapsto \int_{i \fa_{M_P}^{L_\pi,*}} I_{\Pi_{\lambda}}(f_+) \rd \lambda$ is well defined and continuous. If $f_+=f \otimes \Phi$, by Theorem~\ref{thm:FR_expansion} we get
\begin{equation}
    \label{eq:integral_character}
    \int_{i \mathfrak{a}_{M_P}^{L_\pi, *}} I_{\Pi_\lambda}(f_+)
    \rd \lambda=\cP_{H}^*\left(g \mapsto \int_{i \mathfrak{a}_{M_P}^{L_\pi, *}}
    \beta_{\eta}\left(
        W^{\overline{\psi}}\left(K_{f, \chi}(g, \cdot)_{\Pi^\vee_{-\lambda}} \right) \right) \rd \lambda,\Phi\right).
\end{equation}

\subsubsection{Spectral expansion of \texorpdfstring{$I_\chi$}{the relative character}} 

Let $f \in \cS(G(\bA))^J$ and $\Phi \in \cS(\bA_E)$. Set
    \begin{equation*}
    K^{H}_{f \otimes \Phi, \chi}(g)
    = \int_{[H]} K_{f \otimes \Phi, \chi}(h, g)  \rd h,
    \quad
    K^{G'}_{f, \chi}(g)
    = \int_{[G']} K_{f, \chi}(g, g') \eta_{n+1}(g') \rd g',
    \quad g \in [G].
    \end{equation*}

\begin{prop}    \label{prop:partialSpect}
Assume that $\chi$ is $(G,H,\overline{\mu})$-regular. We have the following assertions.
\begin{enumerate}
\item If $\chi$ is not Hermitian, then $K^{G'}_{f, \chi}(g) = 0$ for every
    $g \in [G]$. Moreover $I_{\chi}(f \otimes \Phi) = 0$.

\item Assume that $\chi$ is Hermitian. Then
    \begin{equation}    \label{eq:distributionIterated}
    I_{\chi}(f \otimes \Phi) =
    \int_{[G']} K^H_{f \otimes \Phi,\chi}(g') \eta_{n+1}(g')
    \rd g' =\cP_{H}^* \left( K^{G'}_{f, \chi}, \Phi \right),
    \end{equation}
    where the middle integral is absolutely convergent.
\end{enumerate}
\end{prop}

\begin{proof}
Let us assume first that $\chi$ is not Hermitian. By definition we have
    \[
    K_{f, \chi}^{G'}(g) =
    \cP_{G'}\left( K_{f, \chi}(g, \cdot)\right).
    \]
Let $\chi^\vee \in \mathfrak{X}(G)$ be the $(G,H,\overline{\mu})$-regular cuspidal datum
represented by $(M_P,\pi^\vee)$. Since $\chi^\vee$ is not Hermitian, the linear
form $\cP_{G'}$ vanishes identically on $\cS_{\chi^\vee}([G])$ by Theorem~\ref{thm:FR_expansion}. It follows that $K_{f, \chi}^{G'}(g) = 0$. By
Theorem~\ref{thm:truncated_prop} we conclude that $\valP{i_{\chi}^T(f \otimes \Phi)} \ll e^{-r \aabs{T}}$ for some $r>0$ and all $T$ sufficiently positive. It follows from Theorem~\ref{thm:truncated_prop} that $I_{\chi}(f \otimes \Phi) = 0$.

We now assume that $\chi$ is Hermitian. By \cite{BPCZ}*{Lemma~2.10.1.1}, the family of functions $g' \mapsto K_{f,
\chi}(\cdot, g')$ is absolutely integrable over $[G']$ in
$\mathcal{T}_N([G])$ for some $N$. By Theorem~\ref{thm:RS}, the linear form
$\cP_{H}(\cdot, \Phi)$ extends continuously from $\cS_{\chi}([G])$ to
$\cT_{\chi}([G])$. Since
    \[
    K_{f \otimes \Phi, \chi}^H(g') =
    \cP_{H}\left( K_{f, \chi}(\cdot, g'), \Phi \right) =
    \cP_{H}^*\left( K_{f, \chi}(\cdot, g'), \Phi \right),
    \]
we conclude that the integral
    \begin{equation}    \label{eq:iterated_convergence_GL_kernel}
    \int_{[G']} K_{f \otimes \Phi, \chi}^H(g') \eta_{n+1}(g')\rd g'
    \end{equation}
is absolutely convergent, and equals
    \[
    \cP_{H}^* \left(\int_{[G']} K_{f \otimes \Phi, \chi}(\cdot, g')
    \eta_{n+1}(g') \rd g', \Phi \right)
    = \cP_{H}^*\left(K_{f, \chi}^{G'}, \Phi \right).
    \]
The absolute convergence of~\eqref{eq:iterated_convergence_GL_kernel}
and Theorem~\ref{thm:truncated_prop} imply that this integral
equals $I_{\chi}(f \otimes \Phi)$.
\end{proof}

\begin{theorem}
\label{thm:Ichi_spectral} Assume that $\chi \in \mathfrak{X}(G)$ is
$(G,H,\overline{\mu})$-regular and Hermitian. Then for all $f_+ \in \cS(G_+(\bA))^J$ we have
    \begin{equation}
    \label{eq:I_pi_spectral}
    I_\chi(f_+)
    =2^{-\dim \mathfrak{a}_{L_\pi}^*}
    \int_{i \mathfrak{a}_{M_P}^{L_\pi, *}} I_{\Pi_\lambda}(f_+)
    \rd \lambda,
    \end{equation}
where the integral on the right is absolutely convergent.
\end{theorem}

\begin{proof}
By Theorem~\ref{thm:truncated_prop}, the two sides of
\eqref{eq:I_pi_spectral} are continuous in $f_+$, and the RHS is absolutely
convergent. We may therefore assume that $f_+=f \otimes \Phi$. Let us
continue with the notation of Proposition~\ref{prop:partialSpect}. As $\chi^\vee$
is Hermitian and $(G,H,\overline{\mu})$-regular, Theorem~\ref{thm:FR_expansion} applies and yields for
all $g \in [G]$
    \begin{equation}    \label{eq:kernel_beta_expansion}
    K_{f, \chi}^{G'}(g)
    =\cP_{G'}(K_{f, \chi}(g, \cdot))=
    2^{-\dim \mathfrak{a}_{L_\pi}^*} \int_{i \mathfrak{a}_{M_P}^{L_\pi, *}}
   \beta_{\eta}\left(
        W^{\overline{\psi}}\left(K_{f, \chi}(g, \cdot)_{\Pi^\vee_{-\lambda}} \right) \right) \rd \lambda.
    \end{equation}
It now follows from \eqref{eq:integral_character},
Proposition~\ref{prop:partialSpect} and \eqref{eq:kernel_beta_expansion} that
    \begin{equation*}
    I_\chi(f \otimes \Phi) =\cP_{H}^*\left(K^{G'}_{f, \chi}, \Phi\right)
    =2^{-\dim \mathfrak{a}_{L_\pi}^*}
    \int_{i \mathfrak{a}_{M_P}^{L_\pi, *}} I_{\Pi_\lambda}(f_+)
    \rd \lambda.
    \end{equation*}
This proves the theorem.
\end{proof}

\section{Spectral expansion of Liu's trace formula on unitary groups}
\label{sec:spectral_calculation_U}

\label{sec:spec_U}

In this section, we fix $V$ a non-degenerate skew-Hermitian space over $E/F$. We take a polarization $\bV=Y \oplus Y^\vee$. We regularize the spectral contribution of $(\U_V,\U_V',\overline{\mu})$-regular cuspidal data to the $\U$-side of the Liu's relative trace formula.

\subsection{Weil representation and theta series}

\subsubsection{Local Heisenberg-Weil representations} \label{subsubsec:local_Heisenberg} Let $v$ be a place of $F$. We described in \S\ref{subsubsec:HW_reps} the Heisenberg representation $\rho^\vee_{\psi,v}$ of $S(V)(F_v)$ and the Weil representation $\omega^\vee_v$ of $\U(V)(F_v)$. The \emph{Heisenberg--Weil representation} $\nu^\vee_v$ of $J(V)(F_v)$ is defined by the rule
\begin{equation}
\label{eq:nu_defi1}
    \nu^\vee_v(h g)=\rho^\vee_{\psi,v}(h) \omega^\vee_{v}(g) , \quad  g \in \U(V)(F_v), \quad h \in S(V)(F_v).
\end{equation}
It is realized on the local Schwartz space $\cS(Y^\vee(F_v))$, and is unitary for the inner-product $\langle \cdot, \cdot \rangle_{L^2}$ given by integrating along $Y^\vee(F_v)$.

\subsubsection{Automorphic Weil representations and theta series} \label{subsubsec:automorphic_Weil} By choosing a basis, we have an $\cO_F$-integral model $Y^\vee$, so that at each finite place $v$ we have $1_{Y^\vee(\cO_v)} \in \cS(Y^\vee(F_v))$. We may therefore consider the restricted tensor products $\nu^\vee=\otimes_v' \nu^\vee_v$ and $\omega^\vee=\otimes_v' \omega^\vee_v$, both defined with respect to these unramified vectors. They are representations of $J(V)(\bA)$ and $\U(V)(\bA)$ respectively, both realized on $\cS(Y^\vee(\bA))$ and unitary for the global inner product $\langle \cdot, \cdot \rangle_{L^2}$. It decomposes as $\prod_v \langle \cdot, \cdot \rangle_{L^2,v}$ with the choice of measure made in \S\ref{subsec:measure_part1}.

By \cite{GGP}*{Section~23}, these representations are automorphic via the theta series
\begin{equation}
    \label{eq:Theta_unitary_definition}
     \theta^\vee(j, \phi) = \sum_{y^\vee \in Y^\vee(F)}
    \nu^\vee(j) \phi(y^\vee), \quad j \in J(V)(\bA), \quad
    \phi \in \cS(Y^\vee(\bA)).
\end{equation}

In \cite[Section~3.3.5]{BLX1}, we define a space $\cT([J(V)],\psi)$ of functions of uniform moderate growth on $J(V)$ with fixed central character. 

\begin{lemma}
    \label{lem:theta_moderate_u}
    There exists $N>0$ such that $\theta^\vee$ defines a continuous mapping $\cS(Y^\vee(\bA)) \to \cT_N([J(V)],\psi)$.
\end{lemma}

\begin{proof}
    If $V$ is anisotropic, then $[J(V)]$ is compact and there is nothing to prove. In general, $V$ can be written as $V=X \oplus V_0 \oplus X^*$ where $X \oplus X^*$ is split. In that case, we can use the formula for the Weil representation given in \S\ref{subsubsec:explicit_Weil}. This reduces to the case of $\U(V_0)$ by using the same arguments as in Lemma~\ref{lem:theta_moderate}.
\end{proof}

\subsection{Fourier--Jacobi periods for regular Eisenstein series}
\label{subsec:reg_FJ}

We now define an extension of the Fourier--Jacobi period for Eisenstein series induced from $(\U'_V,\overline{\mu})$-regular cuspidal data.

In \cite[Section~7.5]{BLX1}, we define a truncation operator $\Lambda_u^T$. It is a mapping from $\cT([\widetilde{\U}_V],\psi)$ the space of functions of uniform moderate growth on $[\widetilde{\U}_V]$ with fixed central character (see \cite[Section~3.3.5]{BLX1} for a precise definition) to a space $\cS^0([\U'_V])$ of continuous functions of rapid decay on $[\U'_V]$. 

Let $Q$ be a standard parabolic subgroup of $\U_V$, let $\sigma \in \Pi_{\cusp}(M_Q)$. If $\varphi \in \cA_{Q,\sigma}(\U_V)$, $\lambda \in i \fa_{Q}^*$ and $\phi \in \cS(Y^\vee(\bA))$, we can consider $ g \in [\widetilde{\U}_V] \mapsto E(g,\varphi,\lambda) \theta^\vee(\widetilde{g},\phi)$. By Lemma~\ref{lem:theta_moderate_u}, this functions belongs to $\cT([\widetilde{\U}_V],\psi)$ and is denoted by $E(\cdot,\varphi,\lambda) \cdot \theta^\vee(\cdot,\phi)$. By \cite{BP}*{Proposition~A.1.1~(vi)}, we may therefore apply $\Lambda_u^T$ and integrate the result over $[\U'_V]$. This defines a truncated Fourier--Jacobi period.

If the cuspidal datum $\chi$ represented by $(M_Q,\sigma)$ is $(\U_V',\overline{\mu})$-regular, then it turns out that the truncated Fourier--Jacobi period is independent of $T$.

\begin{prop}   \label{prop:IY_stability}
Let $T \in \fa_0$ be sufficiently positive. Assume that $\chi$ is $(\U'_V,\overline{\mu})$-regular. Then the integral
\begin{equation} \label{eq:regularized_Fourier_Jacobi}
    \int_{[\U'_V]}  \Lambda_u^T \left( E(\cdot,\varphi,\lambda)
    \cdot \theta^\vee(\cdot,\phi) \right)(h) \rd h
    \end{equation}
    is independent of $T$. We denote it by $\cP_{\U'_V}(\varphi,\phi,\lambda)$. 
\end{prop}

\begin{proof}
    The proof is very similar to \cite{BPCZ}*{Proposition~5.1.4.1}, and we sketch the main steps for the convenience of the reader. Let $T$ and $T'$ be two truncation parameters such that $T$ and $T+T'$ are sufficiently positive. In \cite[Equation~(7.16)]{BLX1}, it is proved that we have an unfolding identity of the form
    \begin{align*}
         &\int_{[\U'_V]}  \Lambda_u^{T+T'} \left( E(\cdot,\varphi,\lambda)
    \cdot \theta^\vee(\cdot,\phi) \right)(h) \rd h \\
    =&\sum_{R \in \cF_V}  \int_{K'_V}  \int_{[M_{R'}]}  e^{\langle -2\rho_{R'},H_{R'}(m)
    \rangle} \Gamma'_{R'}(H_{R'}(m)-T_{R'},T'_{R'})
    \Lambda^{R,T}_u \left( E_{R_{\U}}(\cdot,\varphi,\lambda)
    \cdot {}_R \theta^\vee(\cdot,\phi) \right)(mk) \rd m \rd k.
    \end{align*}
    Here, $\cF_V$ is a set of parabolic subgroups of $S(V)$ (see \cite[Section~4.3]{BLX1}), $R'$ is a standard parabolic subgroup of $\U_V'$ determined by $R$, $R_{\U}$ is $R' \times R'$ seen as a standard parabolic subgroup of $\U_V$, and ${}_R \theta^\vee$ is defined in \cite[Section~4.4]{BLX1}. The truncation operator $\Lambda^{R,T}_{u}$ and the $\Gamma_{R'}$ function are also defined in \cite[Section~7.5]{BLX1} and \cite[Proof of Theorem~7.6]{BLX1} (the $\Gamma$ function comes from \cite[Section~2]{Arthur1}). If $R=S(V)$, $\Gamma'_{R'}$ is identically equal to $1$ so that the corresponding term in the sum is constant as a function of $T'$. Therefore, it is enough to prove that the other terms vanish.

    Let $R \in \cF_V$ which is not $S(V)$. Then by the description of \cite[Section~4.3]{BLX1} we have $R' \neq \U_V'$ and therefore $R_{\U}\neq \U_V$. If we go back to the definition of $\Lambda^{R,T}_{u}$ and use the description of the mixed model given in \cite[Section~4.4]{BLX1} (see also \S\ref{subsubsec:explicit_Weil}), we see that for every $a \in A_{R'}^\infty$ we have
    \begin{equation*}
    \Lambda^{R,T}_u \left( E_{R_{\U}}(\varphi,\lambda) \cdot
    {}_R \theta^\vee(\phi)\right)(amk) = e^{\langle \lambda + \rho_{R_{\U}}, H_{R_{\U}}(a) \rangle} \overline{\mu}(a)|a|^{\frac{1}{2}}  \Lambda^{R,T}_u \left( E_{R_{\U}}(\varphi,\lambda) \cdot
    {}_R \theta^\vee(\phi)\right)(mk).
\end{equation*}
    By the properties of $\Gamma'_{R'}$ from~\cite{Arthur1}*{Section~2} and the formula for the constant term of Eisenstein series recalled in \eqref{eq:constant_term_Eisenstein}, it is enough to prove that for every $w \in W(Q;R_U)$ we have  
    \begin{equation} \label{eq:IY_left_to_show_vanishing}
    \int_{[M_{R'}]^1} \Lambda^{R,T}_u \left( E^{R_{\U}}(\cdot,M(w,\lambda)\varphi,w\lambda)
    \cdot {}_R\theta^\vee(\cdot,\phi) \right)(m) \rd m=0
    \end{equation}
     Note that the representation $w \pi$ is also $(\U'_V,\mu^{-1})$-regular, so that we may assume that $w=1$, which implies $Q \subset R_{\U}$.

Form the decomposition $M_{R'}=M_{R',\bullet} \times M_{R',*}$, where
$M_{R',\bullet}$ is isomorphic to a product of $\GL_{n_i}$ and $M_{R',*}$ is
a unitary group. We write similarly $M_{R_{\U}}=M_{R_{\U},\bullet}\times M_{R_{\U},*}$ so that $M_{R_{\U},\bullet} \cong M_{R',\bullet}^2$ and the embedding $M_{R',\bullet} \subset M_{R_{\U},\bullet}$ is the diagonal inclusion. Write $Q_1 \times Q_2=M_{R_{\U},\bullet} \cap Q$ and the restriction of the representation $\sigma$ to $M_{Q_1}(\bA) \times M_{Q_2}(\bA)$ as $\sigma=\sigma_1 \boxtimes \sigma_2$. It is enough to show \eqref{eq:IY_left_to_show_vanishing} if 
$\varphi|_{[M_{R_{\U}}]} = (\varphi_1 \otimes \varphi_2) \otimes \varphi_*$, where
$\varphi_1 \otimes \varphi_2 \in \cA_{{Q_1 \times Q_2},\sigma_1 \otimes \sigma_2}(M_{R_{\U},\bullet})$ and
$\varphi_* \in \cA_{{Q \cap M_{R_{\U},*}}}(M_{R_{\U},*})$. 

By the description of the parabolic subgroups of $J(V)$ in \cite[Section~4.3]{BLX1}, it is easily shown that the operator $\Lambda_u^{R,T}$ is a product of
\begin{itemize}
    \item the usual Arthur operator $\Lambda^T$ attached to the product of general linear groups $M_{R',\bullet}$, seen as an operator of the space $\cT([M_{R_{\U},\bullet}])$ acting on the second component (see~\cite{Arthur4}),
    \item the truncation operator $\Lambda_u^T$ attached to unitary group $M_{R',*} \subset M_{R_{\U},*}$, seen as an operator of the space $\cT([M_{R_{\U},*}])$.
\end{itemize}

For $m \in [M_{R'}]^1$ write $m=m_\bullet m_*$ according to our decomposition.
Using the description of mixed model from \cite[Section~4.4]{BLX1}, we see that ${}_R
\theta^\vee(m,\phi)=\overline{\mu}(m_\bullet) \cdot {}_{R}\theta^\vee
(m_*,\phi)$. Thus the integral in ~\eqref{eq:IY_left_to_show_vanishing} has
\begin{equation}
\label{eq:truncated_inner_product}
        \int_{[M_{R',\bullet}]^1} E(m_\bullet ,\varphi_{1,},\lambda)
    \Lambda^T E(m_\bullet,\varphi_{2},\lambda) \overline{\mu}(m_\bullet)
    \rd m_\bullet
\end{equation}
as a factor. The hypothesis $R_{\U} \neq \U_V$ implies that $M_{R',\bullet}$ is not trivial. By Langlands' formula for the inner product of truncated Eisenstein series applied to each general linear group (\cite{Arthur4}), \eqref{eq:truncated_inner_product} is zero as $\chi$ is $(\U'_V,\overline{\mu})$-regular. This concludes the proof.
\end{proof}

\begin{remark}
    \label{rem:cuspi_case}
    The truncation operator $\Lambda_u^T$ involves taking alternating sum of constant terms. In particular, if $\sigma=\sigma_V \boxtimes \sigma_V'$ with either $\sigma_V$ or $\sigma_V'$ cuspidal, we simply have $\Lambda_u^T E(\varphi,\lambda) \cdot \theta^\vee(\phi) =E(\varphi,\lambda)\cdot \theta^\vee(\phi) $. It follows that $\cP_{\U'_V}(\varphi,\phi,\lambda)$ is simply the integral of this function along $[\U'_V]$. 
\end{remark}

For any level $J$, by \cite{Lap}*{Theorem~2.2}, Lemma~\ref{lem:theta_moderate_u} and Theorem~\ref{thm:truncated_prop_u}, there exist continuous
semi-norms $\aabs{\cdot}$ and $\aabs{\cdot}_{Y^\vee}$ on
$\mathcal{A}_{Q,\sigma,}(\U_V)^J$ and $\cS(Y^\vee(\bA))$ respectively such
that
    \begin{equation}    \label{eq:cP_strong_continuity}
    \valP{\cP_{\U'_V}(\varphi,\phi,\lambda)} \ll \aabs{\varphi} \aabs{\phi}_{Y^\vee}, \quad \varphi \in \mathcal{A}_{Q,\sigma}(\U_V)^J, \quad \phi \in \cS(Y^\vee(\bA)).
    \end{equation}

\subsection[The coarse spectral expansion on unitary groups]{The coarse spectral expansion of Liu's trace formula on unitary groups}
\label{sec:coarse_spectral_u}

In this section, we recall the main results of \cite{BLX1} on the regularization of the $\U$-side of Liu's relative trace formula.

Let $\chi \in \fX(\U_V)$. Let $f_V \in \cS(\U_V(\bA))$, $\phi_1, \phi_2 \in \cS(Y^\vee(\bA))$. For $\chi_0 \in \{ \emptyset, \chi\}$, we have the kernel function
\begin{equation*}
    K_{f_V\otimes \phi_1 \otimes \phi_2,\chi_0}(g_1,g_2)=K_{f_V,\chi_0}(g_1,g_2) \theta^\vee(\widetilde{g}_1,\phi_1)\overline{\theta}^\vee(\widetilde{g}_2,\phi_2), \quad g_1,g_2 \in [\widetilde{\U}_V].
\end{equation*}
The assignment $f_V \otimes \phi_1 \otimes \phi_2 \mapsto K_{f_V\otimes \phi_1 \otimes \phi_2,\chi_0}$ extends by continuity to $\cS(\U_{V,+}(\bA))$. For $T$ sufficiently positive a modified kernel $K^T_{f_{V,+},\chi_0}$ is also defined in \cite[Section~7.3]{BLX1}. It is a function on $\widetilde{\U}_V(\bA)^2$.

For fixed $g_1$, the function $K_{f_{V,+},\chi_0}(g_1,\cdot)$ belongs to $\cT([\widetilde{\U}_V],\psi)$. We can therefore apply $\Lambda_u^T$ and we denote by $K_{f_{V,+},\chi_0}\Lambda_u^T(\cdot,\cdot) $ the functions in two variables thus obtained. We now sum up the principal properties of these two modified kernels.

\begin{theorem}
    \label{thm:truncated_prop_u}
    The following properties hold for any $T$ positive enough and any $f_{V,+} \in \cS(\U_{V,+}(\bA))$.
    \begin{enumerate}
        \item \cite[Proposition~7.5]{BLX1} The integral 
        \begin{equation*}
            J_{V,\chi_0}^T(f_+)=\int_{[\U'_V]} \int_{[\U'_V]}K_{f_{V,+},\chi_0}^T(h_1,h_2) \rd h_1 \rd h_2,
        \end{equation*}
        is absolutely convergent. 
        \item \cite[Theorem~7.6]{BLX1} For $T$ sufficiently positive, the function $J^T_{V,\chi_0}(f_{V,+})$ is the
        restriction of an exponential-polynomial function whose purely polynomial part
        is constant. We denote is by $J_{V,\chi_0}(f_{V,+})$. It defines a continuous distribution on $\cS(\U_{V,+}(\bA))$  and we have
            \[
            \sum_{\chi \in \fX(G_{V})} J_{V,\chi}(f_{V,+}) = J_V(f_{V,+}),
            \]
        where the sum is absolutely convergent.
        \item \cite[Section~7.7]{BLX1} The operator $\Lambda_u^T : \cT([\widetilde{\U}_V],\psi) \to \cS^0([\U'_V])$ is continuous. 
        \item \cite[Section~7.8]{BLX1} For any $r>0$, there exists a continuous semi-norm $\aabs{\cdot}$ on $\cS(\U_{V,+}(\bA))$ such that 
        \begin{equation*}
            \valP{J_{V,\chi_0}^T (f_{V,+})-\int_{[\U'_V]} \int_{[\U'_V]} K_{f_{V,+},\chi_0} \Lambda_u^T(h_1,h_2) \rd h_1 \rd h_2 } \leq \aabs{f_{V,+}} e^{-r \aabs{T}}.
        \end{equation*}
        In particular, the absolutely convergent integral
        \begin{equation*}
            \int_{[\U'_V]}\int_{[\U'_V]} K_{f_{V,+},\chi_0} \Lambda_u^T(h_1,h_2) \rd h_1 \rd h_2
        \end{equation*} 
        is asymptotic to an exponential-polynomial in $T$, whose purely polynomial term is a constant that equals $J_{V,\chi_0}(f_{V,+})$.
    \end{enumerate}
\end{theorem}

The second point of the theorem is the \emph{coarse spectral expansion} of Liu's trace formula on unitary groups.

\subsection{Expansion of relative characters on the $\U$-side}
\label{sec:expansion_spectral_u}
We henceforth assume that $\chi$ is a $(\U_V,\U'_V,\overline{\mu})$-regular cuspidal datum of $\U_V$. We fix $J$ a level of $\U_V$. We now write a spectral expansion of $J_{V,\chi}$ in terms of relative characters.

\subsubsection{The relative character \texorpdfstring{$J_{Q,\sigma}$}{}}
\label{subsec:J_Q_sigma}

By Proposition~\ref{prop:strong_K_basis}, \cite{Lap}*{Theorem~2.2} and \eqref{eq:cP_strong_continuity}, we can define for $\lambda \in  \mathfrak{a}_{Q,\C}^*$ in general position, $f_V \in \cS(\U_V(\bA))^J$ and $\phi_1, \phi_2 \in
\cS(Y^\vee(\bA))$ the relative character
    \begin{equation}    \label{eq:P_abs_conv}
    J_{Q,\sigma}(f_V \otimes \phi_1 \otimes \phi_2,\lambda) =
    \sum_{\varphi \in \mathcal{B}_{Q,\sigma}(J)}
    \mathcal{P}_{\U'_V}(I_Q(\lambda,f)\varphi,\phi_1,\lambda)
    \overline{\mathcal{P}_{\U'_V}(\varphi,\overline{\phi_2},\lambda)}.
    \end{equation}
For fixed $f_V$, $\phi_1$ and
$\phi_2$ it is meromorphic in $\lambda$ and holomorphic in the region $\lambda \in  i
\mathfrak{a}_Q^*$. For fixed $\lambda$, the map $f_V \otimes \phi_1 \otimes \phi_2 \mapsto J_{Q,\sigma}( f_V \otimes \phi_1 \otimes \phi_2,\lambda)$ extends by continuity to $\cS(\U_{V,+}(\bA))^J$. We need to bound it in families.

\begin{prop}    \label{prop:IY_sum_convergence}
There exist continuous semi-norms $\aabs{\cdot}$ on $\cS(\U_V(\bA))^J$ and $\aabs{\cdot}_{Y^\vee}$ on $\cS(Y^\vee(\bA))$ such that for any $f_V \in \cS(\U_V(\bA))^J$ and $\phi_1, \phi_2 \in \cS(Y^\vee(\bA))$
    \begin{equation}    \label{eq:sum_over_weyl_oribts}
    \sum_{(M_Q,\sigma)} \int_{ i \mathfrak{a}_Q^*}
    \valP{J_{Q,\sigma}( f_V \otimes \phi_1 \otimes \phi_2),\lambda} \rd \lambda
    \leq \aabs{f_V} \aabs{\phi_1}_{Y^\vee} \aabs{\phi_2}_{Y^\vee},
    \end{equation}
where $(M_Q,\sigma)$ ranges over a set of representatives of $(\U_V,\U'_V,\overline{\mu})$-regular cuspidal data
of $\U_V$. 
\end{prop}

\begin{proof}
Fix $\phi_1$ and $\phi_2$. We claim that the map $f_V \mapsto \left( ((Q,\sigma),\lambda) \mapsto J_{Q,\sigma}(f_V\otimes\phi_1\otimes\phi_2,\lambda) \right)$, valued in the space of $L^1$ functions with variables $((Q,\sigma),\lambda)$ (given the product of the counting measure and the Lebesgue measure), is continuous. To prove this, it is enough to show that for every $f_V \in \cS(\U_V(\bA))^J$ the sum 
\begin{equation}
\label{eq:selberg_trick_0}
    \sum_{(M_Q,\sigma)} \int_{ i \mathfrak{a}_Q^*}
    \valP{J_{Q,\sigma}(f_V \otimes \phi_1 \otimes \phi_2),\lambda} \rd \lambda
\end{equation}
is finite. Indeed, if it is, then our functional is the pointwise limit of a sequence of continuous forms on $\cS(\U_V(\bA))$, hence it is continuous by the closed graph theorem. 

Let $f_V \in \cS(\U_V(\bA))^J$. By Proposition~\ref{prop:IY_stability}, for $T$ positive enough \eqref{eq:selberg_trick_0} is
\begin{equation*}
    \sum_{(M_Q,\sigma)} \int_{ i \mathfrak{a}_Q^*} 
    \valP{\sum_{\varphi \in \mathcal{B}_{Q,\sigma}(J)} \int_{[\U'_V]^2}   \Lambda_u^T \left( E(I_Q(\lambda,f_V) \varphi,\lambda)
     \theta^\vee(\phi_1) \right)(h_1) \overline{\Lambda_u^T \left( E(\varphi,\lambda)
     \theta^\vee(\overline{\phi_2}) \right)}(h_2) \rd h_i} \rd \lambda.
\end{equation*} 
By Proposition~\ref{prop:strong_K_basis} and Theorem~\ref{thm:truncated_prop_u}, we can switch the last sum and integral, so that this expression is bounded above by
\begin{equation}
\label{eq:selberg_trick}
    \sum_{(M_Q,\sigma)} \int_{ i \mathfrak{a}_Q^*} 
    \int_{[\U'_V]^2}  \valP{\sum_{\varphi \in \mathcal{B}_{Q,\sigma}(J)}  \Lambda_u^T \left( E(I_Q(\lambda,f_V) \varphi,\lambda)
     \theta^\vee(\phi_1) \right)(h_1) \overline{\Lambda_u^T \left( E(\varphi,\lambda)
     \theta^\vee(\overline{\phi_2}) \right)}(h_2)} \rd h_i \rd \lambda.
\end{equation} 
By the Dixmier--Malliavin theorem from \cite{DM}, we may assume that $f_V=f_1 * f_2^{\vee}$, where $f_2^{\vee}(g)=\overline{f_2}(g^{-1})$. By change of basis, it follows from the Cauchy--Schwarz inequality that, up to replacing $\overline{\phi_2}$ by $\phi_2$, \eqref{eq:selberg_trick} is bounded above by the square root of the product over $i=1,2$ of
\begin{equation}
\label{eq:CS_convergence}
    \sum_{(M_Q,\sigma)} \int_{ i \mathfrak{a}_Q^*} \int_{[\U'_V] \times [\U'_V]} \sum_{\varphi \in \mathcal{B}_{Q,\sigma}(J)} \valP{\Lambda_u^T \left( E(I_Q(\lambda,f_i) \varphi,\lambda)
    \cdot \theta^\vee(\phi_i) \right)(h_i)}^2  \rd h_1 \rd h_2 \rd \lambda.
\end{equation}
Set $g_i=f_i*f_i^\vee$. Using again a change of variable and the continuity of $\Lambda_u^T$ (Theorem~\ref{thm:truncated_prop_u}) we have
\begin{equation}
\label{eq:positive_sum}
     \sum_{\varphi} \valP{\Lambda_u^T \left( E(I_Q(\lambda,f_i) \varphi,\lambda)
     \theta^\vee(\phi_i) \right) }^2= \Lambda_u^T \left(\sum_{\varphi } E(I_Q(\lambda,g_i) \varphi,\lambda)
     \theta^\vee(\phi_i) \otimes \overline{E(\varphi,\lambda)
     \theta^\vee(\phi_i)} \right) \Lambda_u^T,
\end{equation}
where we use $\Lambda_u^T F \Lambda_u^T$ to denote truncation in the left and right variables of a function $F$. Moreover, we can switch the integrals by positivity and \eqref{eq:CS_convergence} becomes 
\begin{equation}
\label{eq:Selberg_before_continuity}
    \vol([\U'_V]) \int_{[\U'_V]}  \sum_{(M_Q,\sigma)} \int_{ i \mathfrak{a}_Q^*}\Lambda_u^T \left(\sum_{\varphi \in \mathcal{B}_{Q,\sigma}(J)} E(I_Q(\lambda,g_i) \varphi,\lambda)
     \theta^\vee(\phi_i) \otimes \overline{E(\varphi,\lambda)
     \theta^\vee(\phi_i)} \right) \Lambda_u^T(h,h)   \rd \lambda \rd h.
\end{equation}

Set $ K_{g_i,\mathrm{reg}}=\sum_{\chi } K_{g_i,\chi}$, where the sum ranges over $(\U_V,\U'_V,\overline{\mu})$-regular cuspidal data of $\U_V$. By \cite{BPCZ}*{Lemma~2.10.1.1} and Theorem~\ref{thm:kernel_spectral_expansion} $ K_{g_i,\mathrm{reg}}$ is well-defined, belongs to $\cT([\U_V \times \U_V])$ and we have the spectral expansion
\begin{equation*}
     K_{g_i,\mathrm{reg}}(x,y) \theta^\vee(x,\phi_i) \theta(y,\phi_i) 
     = \sum_{(M_Q,\sigma)} \int_{ i \mathfrak{a}_Q^*}\sum_{\varphi \in \mathcal{B}_{Q,\sigma}(J)} E(x,I_Q(\lambda,g_i) \varphi,\lambda)\theta^\vee(x,\phi_i) \overline{E(y,\varphi,\lambda)
     \theta^\vee(y,\phi_i)}  \rd \lambda. 
\end{equation*}
The double integral $\sum_{(M_Q,\sigma)} \int_{ i \mathfrak{a}_Q^*}$ is absolutely convergent for fixed $x$ and $y$ by Theorem~\ref{thm:kernel_spectral_expansion}. Therefore, for fixed $h \in [\U'_V]$ we have
\begin{align*}
    &\sum_{(M_Q,\sigma)} \int_{ i \mathfrak{a}_Q^*}\Lambda_u^T \left(\sum_{\varphi \in \mathcal{B}_{Q,\sigma}(J)} E(I_Q(\lambda,g_i) \varphi,\lambda)
    \cdot \theta^\vee(\phi_i) \otimes \overline{E(\varphi,\lambda)
    \cdot \theta^\vee(\phi_i)} \right) \Lambda_u^T(h,h)   \rd \lambda  \\
    &= \Lambda_u^T \left(K_{g_i,\mathrm{reg}} \theta^\vee_x(\phi_i)\overline{\theta}^\vee_y(\phi_i) \right)\Lambda_u^T(h,h),
\end{align*}
where we write $\theta_x^\vee(\phi_i)$ (resp. $\theta_y^\vee(\phi_i)$) for $\theta^\vee(\phi_i)$ seen as a function of the first variable (resp. $\theta^\vee(\phi_i)$ as a function of the second variable). Indeed, it follows from the definition of $\Lambda_u^T$ given in \cite[Section~7.5]{BLX1} and from~\cite{Arthur3}*{Lemma~5.1} that for any $F \in \cT([\widetilde{\U}_V \times \widetilde{\U}_V])$, $\Lambda_u^T F \Lambda_u^T F(h,h)$ is a finite sum of left-translates of constant terms of $F$ (independent of $F$) evaluated at $h$. These constant terms are integrals over compact subsets and therefore commute with the absolutely convergent double integral $\sum_{(M_Q,\sigma)} \int_{ i \mathfrak{a}_Q^*}$. This implies that
\begin{equation*}
    \eqref{eq:Selberg_before_continuity} \leq  \vol([\U'_V]) \int_{[\U'_V]} \Lambda_u^T \left(K_{g_i,\mathrm{reg}} \theta^\vee_x(\phi_i)\overline{\theta}^\vee_y(\phi_i) \right)\Lambda_u^T(h,h) \rd h,
\end{equation*}
and this expression is finite by Theorem~\ref{thm:truncated_prop_u}. Therefore, \eqref{eq:selberg_trick_0} is also finite and this concludes the proof of the claim.

It remains to obtain uniform bounds on the $\phi_i$. As the $g_i$ depend only on $f_V$, we see that our bound is uniform in $\phi_1$ and $\phi_2$, in the sense that there exists a continuous semi-norm $\aabs{.}_{Y^\vee}$ on $\cS(Y^\vee(\bA))$ such that for every $f_V$
\begin{equation*}
    \sup_{\substack{\phi_i \in \cS(Y^\vee(\bA)) \\ \aabs{\phi_i}_{Y^\vee} \leq 1}}\sum_{(M_Q,\sigma)} \int_{ i \mathfrak{a}_Q^*}
    \valP{J_{Q,\sigma}( f_V \otimes \phi_1 \otimes \phi_2,\lambda)} \rd \lambda < \infty.
\end{equation*}
The bound \eqref{eq:sum_over_weyl_oribts} now follows from the uniform boundedness principle.
\end{proof}

As a byproduct of the proof, we obtain the following lemma.
\begin{lemma}
\label{lem:intermediate_regular_expansion}
    Let $f_{V,+} \in \cS(\U_{V,+}(\bA))^J$. Then for $T_1$ and $T_2$ sufficiently positive we have
    \begin{equation}
    \label{eq:selberg_equality}
        \int_{ i \mathfrak{a}_Q^*}
    J_{Q,\sigma}(f_{V,+},\lambda) \rd \lambda=\int_{[\U'_V] \times [\U'_V]} \Lambda_u^{T_1} \left(K_{f_{V,+},\chi} \right)\Lambda_u^{T_2}(x,y) \rd x \rd y.
    \end{equation}
\end{lemma}

\begin{proof}
As both sides are continuous in $f_+$ by Proposition~\ref{prop:IY_sum_convergence} and Theorem~\ref{thm:truncated_prop_u}, we may assume that $f_+ = f \otimes \phi_1
\otimes \phi_2$. The lemma now follows from the spectral expansion of the kernel from Theorem~\ref{thm:kernel_spectral_expansion} as all the necessary manipulations on the integrals are justified by the dominated convergence theorem using Proposition~\ref{prop:IY_sum_convergence} and its proof.
\end{proof}

\subsubsection{Spectral expansion of \texorpdfstring{$J_{V,\chi}$}{the relative character}}

We come to the main result of this chapter.

\begin{theorem} \label{thm:Jchi_spectral}
Let $\chi \in \mathfrak{X}(\U_V)$ be a $(\U_V,\U'_V,\overline{\mu})$-regular cuspidal datum, represented by
$(M_Q, \sigma)$. Then for all $f_{V,+} \in \cS(\U_{V,+}(\bA))^J$, we have
    \begin{equation}    \label{eq:JPi_spectral}
    J_{V,\chi}(f_{V,+})
    =
    \int_{i \mathfrak{a}^*_Q} J_{Q,\sigma}
    (f_{V,+},\lambda) \rd\lambda.
    \end{equation}
\end{theorem}

\begin{proof}
    By Lemma~\ref{lem:intermediate_regular_expansion}, for $T_1$ and $T_2$ positive enough, we have
    \[
    \int_{[\U'_V] \times [\U'_V]} (\Lambda^{T_1}_u K_{f_{V,+}, \chi} \Lambda^{T_2}_u)
    (x, y) \rd x \rd y =
    \int_{i \fa_Q^*} J_{Q, \sigma}( f_{V,+},\lambda) \rd \lambda.
    \]
As the RHS is independent of $T_1$ and $T_2$, we conclude that so is the LHS. By \cite{BPCZ}*{Lemma~2.10.1.1} and Theorem~\ref{thm:truncated_prop_u}, we have
\begin{equation*}
    \lim_{T_1 \to \infty } \int_{[\U'_V] \times [\U'_V]} (\Lambda^{T_1}_u K_{f_{V,+}, \chi} \Lambda^{T_2}_u)
    (x, y) \rd x \rd y = \int_{[\U'_V] \times [\U'_V]} K_{f_{V,+}, \chi} \Lambda^{T_2}_u
    (x, y) \rd x \rd y
\end{equation*}
The theorem now follows from Theorem~\ref{thm:truncated_prop_u}.
\end{proof}

\section{Global comparison of relative characters for regular cuspidal data}
\label{sec:Global_comparison}

\label{sec:comparison}

In this chapter, we compare the spectral contributions of matching regular cuspidal data in Liu's trace formulae, using the results of \cite{BLX1} and \cite{BLX2}. We then translate this in terms of relative characters in Theorem~\ref{thm:global_identity}.

\subsection{Comparison of the geometric sides}
\label{sec:comparison_geom}

In this section, we recall the main result of \cite{BLX2} on the comparison of the geometric side of Liu's trace formula. These results ultimately rely on a infinitesimal reduction of Liu's trace formula to the Jacquet--Rallis trace formula where the results of \cite{Zhang1}, \cite{Xue1}, \cite{Xue3}, \cite{Yun}, \cite{BP3}, \cite{Zhang3} and \cite{CZ} can be used.

\subsubsection{Local transfer of test functions} Let $\tS$ be a finite set of places of $F$. In \cite{BLX2}*{Section~4.3}, we introduce a notion of matching between test functions in $\cS(\U_+(F_\tS))$ and tuples of elements in $\prod_{V \in \cH_\tS} \cS(\U_{V,+}(F_\tS))$. It is defined as an equality of regular semi-simple orbital integrals along matching orbits. 

We say that a $f_{+,\tS} \in \cS(\U_+(F_\tS))$ is transferable if there exists a collection $(f_{V,+,\tS})_{V \in \cH_\tS} \in \prod_{V \in \cH_\tS} \cS(\U_{V,+}(F_\tS))$ that matches it. We say that a collection $(f_{V,+,\tS})_{V \in \cH_\tS}$ is transferable if there exists a $f_{+,\tS}$ that matches it, and that for $V \in \cH_\tS$ a single $f_{V,+,\tS}$ is transferable if the collection $(f_{V,+,\tS},0,\hdots,0)$ is transferable. We now state the main theorem on transferable functions of \cite{BLX2}. 

\begin{theorem}[\cite{BLX2}*{Theorem~4.3}]
    \label{thm:matching}
    The subset of transferable test functions in $\cS(G_+(F_\tS))$ and $\cS(\U_{V,+}(F_\tS))$ are dense. 
\end{theorem}

\subsubsection{Fundamental lemma}
\label{subsubsec:FL}
For almost every place $v$ of $F$, the situation is unramified in the sense of \cite{BLX2}*{Theorem~4.8}. For these places, we have a fundamental lemma. 

\begin{theorem}[\cite{BLX2}*{Theorem~4.8}]
    \label{thm:FL_FJ}
    For almost every place $v$, the functions $\Delta_{H(F_v)} \Delta_{G'(F_v)} 1_{G_+(\cO_v)}$ and $\Delta_{\U(V)(F_v)}^2 1_{\U_{V,+}(\cO_v)}$ match. Here the $\Delta$ factors are the local constants described in \S\ref{subsec:tamagawa_measure_new}.
\end{theorem}

\subsubsection{Global transfer}

We go back to the global setting. We say a test function $f_+ \in \cS(G_+(\bA))$ and a collection of test functions $(f_{V,+})_{V \in \cH}$, where $f_{V,+} \in \cS(\U_{V,+}(\bA))$, match, if there exists a finite set of places $\tS$ of $F$ containing all the Archimedean places and ramified places in $E$, such that:
\begin{itemize}
    \item $f_{V,+} = 0$ for $V \not \in \cH^{\tS}$;
    \item for each $V \in \cH^{\tS}$, $f_{V,+} = (\Delta_{\U(V)}^{\tS})^2 f_{V,+,\tS} \otimes 1_{\U_{V,+}(\cO^\tS)}$ where $ f_{V,+,\tS}
        \in \cS(\U_{V,+}(F_\tS))$;
    \item $f_+ = \Delta_H^{\tS,*} \Delta_{G'}^{\tS,*} f_{+,\tS} \otimes
        1_{G_+(\cO^\tS)}$, where $f_{+,\tS} \in \cS(G_+(F_\tS))$;
    \item $f_{+,\tS}$ and $(f_{V,+,\tS})_{V \in \cH_{\tS}}$ match.
\end{itemize}
By the fundamental lemma of Theorem~\ref{thm:FL_FJ}, matching is independent of the choice of $\tS$. 

\subsubsection{Matching of the geometric sides}

In \cite{BLX1}, we prove that the distributions $I$ and $J_V$ introduced in Theorem~\ref{thm:truncated_prop} and Theorem~\ref{thm:truncated_prop_u} admit geometric expansions. By comparing them for matching test functions, we prove the following result.

\begin{theorem}[\cite{BLX1}*{Theorem~6.7}, \cite{BLX1}*{Theorem~8.7}, \cite{BLX2}*{Theorem~4.10}]
    \label{thm:geom_comparison}
    If $f_+$ and $(f_{V,+})_V$ match, then we have
    \begin{equation*}
        I(f_+)=\sum_{V \in \cH}J_V(f_{V,+}).
    \end{equation*}
\end{theorem}

\subsection{Spectral characterization of local transfer}
\label{sec:spectral_local_transfer}

In this section, we present a spectral characterization of local transfer proved in \cite{BLX2}. It is the Fourier--Jacobi version of \cite{BPLZZ}*{Lemma~4.9} and is proved by a comparison of local trace formulae which is analogous to \cite{BP1}. We henceforth fix $v$ a place of $F$.

\subsubsection{Local spherical characters on $G$}
\label{subsec:local_spheri}

Let $\pi_v$ be a smooth irreducible tempered representation of $G(F_v)$. Write $\cW(\pi_v,\psi_{N,v})$ for its space of Whittaker functions with respect to the local component $\psi_{N,v}$ of the automorphic character $\psi_N$ of $N$ defined in \S\ref{subsubsec:whittaker}. It is given an invariant inner-product $\langle \cdot,\cdot \rangle_{\mathrm{Whitt},v}$ by integrating along $N(F_v) \backslash P_{n,n}(F_v)$ (well-defined by \cite{Jac2}), where we recall that $P_{n,n}$ is the product of the mirabolic subgroups of $G_n$. For $W \in \cW(\pi_v,\psi_{N,v})$ and $\Phi \in \cS(E_{n,v})$, we set
\begin{equation}    \label{eq:local_RS}
    Z_v(W, \Phi,\overline{\mu},0) = \int_{N_H(F_v) \backslash H(F_v)}
    W(h) \Phi(e_n h) \overline{\mu_v}(h)
    \valP{\det h}_{E_v}^{\frac{1}{2}} \rd h,
    \end{equation}
and 
\begin{equation}    \label{eq:local_FR}
    \beta_{\eta,v}(W) = \int_{N'(F_v) \backslash P_{n,n}'(F_v)} W(p) \eta_{n+1,v}(p) \rd p.
\end{equation}
That these integrals are well-defined follows from \cite{BLX2}*{Lemma~6.5} and from \cite{BP1}*{Lemma~2.15.1}. If $f \in \cS(G(F_v))$, we write $\pi_v(f)$ for the action of $f$ on $\pi_v$. Then we set
\begin{equation}
    \label{eq:local_spheri_char_gln}
    I_{\pi_v}(f\otimes \Phi)=\sum_W Z_v(\pi_v(f) W, \Phi,\overline{\mu},0) \overline{\beta}_{\eta,v}(W),
\end{equation}
where the sum ranges over an orthonormal basis of $\cW(\pi_v,\psi_{N,v})$. By \cite{BLX2}*{Lemma~6.6} the map $I_{\pi_v}$ extends by continuity to $\cS(G_+(F_v))$.

\subsubsection{Local Fourier--Jacobi periods} \label{subsubsec:local_FJ} Let $V \in \cH_{\{v\}}$. Let $\sigma_v$ be an irreducible tempered representation of $\U_V(F_v)$. We define the \emph{local Fourier--Jacobi period} as the integral
\begin{equation*}
    \cP_{\U'_V,v}(\varphi_1 \otimes \phi_1,\varphi_2 \otimes \phi_2):=\int_{\U'_V(F_v)} \langle \sigma_v(h) \varphi_1,\varphi_2 \rangle \langle \omega_v^\vee(h) \phi_1,\phi_2 \rangle_{L^2} \rd h, \quad \varphi_1, \varphi_2 \in \sigma_v, \quad \phi_1, \phi_2 \in \omega_v.
\end{equation*}
This integral is absolutely convergent by \cite{Xue2}*{Proposition~1.1.1}. If $\varphi_1=\varphi_2=\varphi$ and $\phi_1=\phi_2=\phi$, we simply write $\cP_{\U'_V,v}(\varphi,\phi)$.

\subsubsection{Local spherical characters on $\U_V$}
We define the character on the unitary side. We keep $V$ and $\sigma_v$. Let $f_{V} \in \cS(\U_V(F_v))$ and $\phi_1, \phi_2 \in \cS(Y^\vee(F_v))$. We set 
\begin{equation}
    \label{eq:local_spheri_char_u}
    J_{\sigma_v}(f_V \otimes \phi_1 \otimes \phi_2)=\sum_{\varphi \in \sigma_v} \cP_{\U'_V,v}(\sigma_v(f_V)\varphi \otimes \phi_1,\varphi \otimes \overline{\phi_2}),
\end{equation}
where the sum ranges over an orthonormal basis of $\sigma_v$. By \cite{BLX2}*{Lemma~6.2}, the map $J_{\sigma_v}$ extends by continuity to $\cS(\U_{V,+}(F_v))$.

\subsubsection{Matching of test functions as equality of spherical characters}

We keep $V \in \cH_{\{v\}}$. We denote by $\Temp_{\U_V'}(\U_V(F_v))$ the subset of smooth irreducible tempered representations $\sigma_v$ of $\U_V(F_v)$ such that the dimension of the space of Fourier--Jacobi functionals $\Hom_{\U'_V(F_v)}(\sigma_v \otimes \omega_v^\vee,\C)$ is $1$. 

\begin{theorem}[\cite{BLX2}*{Theorem~7.11}]
    \label{thm:spectral_transfer}
    Let $\kappa_V$ be the constant defined by \cite{BLX2}*{Theorem~7.10}. Let $f_+ \in \cS(G_+(F_v))$ and $f_{V,+} \in \cS(\U_{V,+}(F_v))$. Then $f_+$ and $f_{V,+}$ match if and only if for all $\sigma_v \in \Temp_{\U_V'}(\U_V(F_v))$ we have
    \begin{equation*}
        \kappa_V I_{\BC(\sigma_v)}(f_+)=J_{\sigma_v}(f_{V,+}).
    \end{equation*}
\end{theorem}

If $V$ is now a skew-Hermitian space over the global field $F$, we obtain for each place $v$ the constant $\kappa_{V,v}$. If we choose the local measures as in \S\ref{subsec:measure_part1}, it follows from \cite{BLX2}*{Theorem~7.10} that they are almost all equal to $1$ and that $\prod_v \kappa_{V,v}=1$.

\subsection{Comparison of the spectral sides}
\label{sec:comparison_spectral}

We now obtain the desired comparison of relative characters.

\subsubsection{Levels}
Let $\tS$ be a finite set of places of $F$ containing $V_{F,\infty}$ as well as the ramified places of $E/F$. Let $V \in \cH^\tS$. For each finite place $v \in \tS \backslash V_{F,\infty}$, we take $J_v \subset G(F_v)$ and $J_{V,v} \subset \U_{V}(F_v)$ some open compact subgroups. Set
    \begin{equation*}
        J:= \prod_{v \in \tS \setminus V_{F,\infty}} J_v \prod_{v \notin \tS} K_v,  \quad \text{and} \quad J_{V}:= \prod_{v \in \tS \setminus V_{F,\infty}} J_{V,v}' \prod_{v \notin \tS} K_{V,v}.
    \end{equation*}
We denote by $\cS(G_+(\bA))^J$ and
$\cS(\U_{V,+}(\bA))^{J_V}$ the subsets of $J$ bi-invariant (resp. bi-$J_{V}$ invariant) functions, where the compact subgroup acts on the first factor. 

\subsubsection{Multipliers}

Let $\tT$ be the union of $\tS \setminus V_{F,\infty}$ and of the set of all finite places of $F$ that are inert in $E$. Denote by $\cM^{\tT}(G(\bA))$ the algebra of $\tT$-multipliers defined in~\cite{BPLZZ}*{Definition~3.5} relatively to $\prod_{v \not \in \tT} K_v$. Any $m \in \cM^{\tT}(G(\bA))$ gives rise to a continuous linear operator $m*$ of $\cS(G(\bA))^J$. They satisfy the following property: for every irreducible smooth representation $\rho$ of $G(\bA)$, there exists a complex number $m(\rho)$ such that for all $f \in \cS(G(\bA))^J$ we have
\begin{equation}
\label{eq:mult_fundamental}
    \rho(m*f)=m(\rho)\rho(f).
\end{equation}

At the level of restricted tensor products, there exists a finite set of places $\tS'$ containing $V_{F,\infty}$ and disjoint from $\tT$ such that, for any $f = f_{\tS'} \otimes f^{\tS'} \in \cS(G(\bA))^J$, we have
\begin{equation}
\label{eq:multiplier_finite_place}
    m*f= m_{\tS'}*f_{\tS'} \otimes f^{\tS'}
\end{equation}
where $m_{\tS'}*$ is a continuous linear operator of
$\cS(G(F_{\tS'}))^{J_{\tS'}}$ the subalgebra of $\cS(G(F_{\tS'}))$ of Schwartz
functions which are $\prod_{v \in \tS' \setminus V_{F,\infty}} K_v$ bi-invariant. We extend any multiplier $m \in \cM^{\tT}(G(\bA))$ to $\cS(G_+(\bA))^J$ by acting trivially on the last factor of $\cS(G(\bA))^J \otimes \cS(\bA_{E,n})$ and using density.

For $V \in \cH^\tS$, we denote by $\cM^{\tT}(\U_V(\bA))$ the algebra of $T$-multipliers of $\U_V(\bA)$ relatively to $\prod_{v \not \in \tT} K_{V,v}$. We have similar statements as above, and in particular we extend any $m_V \in \cM^\tT(\U_V(\bA))$ to $\cS(\U_{V,+} (\bA))^{J_V}$.

\subsubsection{Global Arthur-packet}
\label{subsec:cuspi_arthur} 
Let $P$ be a standard parabolic subgroup of $G$ with standard Levi factor $M_P$. Let $\pi \in \Pi_{\cusp}(M_P)$ and set $\Pi=I_{P}^{G} \pi$. Assume that $\Pi$ is a $(G,H,\overline{\mu})$-regular Hermitian Arthur parameter and denote by $\chi_0 \in \fX(G)$ the cuspidal datum represented by the pair $(M_P,\pi)$. Let $\Pi_0$ be the discrete component of $\Pi$ (see \S\ref{subsec:reg_herm_art}). Enlarging $\tS$ and shrinking the $J_v$ if necessary, we assume that $\Pi$ admits non-zero $J$-invariant vectors.

For $V \in \cH^\tS$, we denote by $\fX_\Pi(\U_V) \subset \fX(\U_V)$ the set of cuspidal data represented by pairs $(M_{Q},\sigma)$ such that $\Pi$ is a $(V_F \setminus (\tT \cup V_{F,\infty}))$-weak base change of $(Q,\sigma)$ and $\sigma$ is unramified outside of $\tS$. It follows that such a $(M_{Q},\sigma)$ represents a $(\U_V,\U_V',\overline{\mu})$-regular cuspidal datum. Moreover, we have a natural isomorphism
\begin{equation}
\label{eq:bc}
    bc : \fa_{P}^{L_\pi,*} \to \fa_{{Q}}^*.
\end{equation}
Indeed, if $M_Q=G^\sharp \times \U$ where $G^\sharp$ is a product of linear groups and $\U$ a product of two unitary groups, then $G^\sharp \times G^\sharp$ is a factor of $M_P$ so that $\fa_{Q}^* \oplus \fa_{Q}^*$ is a subspace of $\fa_{P}^*$. The map \eqref{eq:bc} is now the inverse
of $x \mapsto(x,-x)$ whose image is $\fa_{P}^{L_\pi,*}$. In particular, $\fa_{P}^{L_\pi,*}$ is isomorphic to $\fa_\Pi^*$ defined in \eqref{eq:fa_Pi}. Via $bc$,
the pullback of the measure on $i\fa_{Q}^*$ is
$2^{\dim(\fa_{P}^{L_\pi,*})}$ times the measure on
$i\fa_{P}^{L_\pi,*}$.

\subsubsection{Comparison of global relative characters}

Let $V \in \cH^\tS$. For $\lambda \in \fa_{M_P,\C}^{L_\pi,*}$ in general position we define a relative character
\begin{equation}
    \label{eq:J_Pi_big}
        J_{V,\Pi}( f_{V,+},\lambda):=\sum_{(M_{Q},\sigma)} J_{Q,\sigma}(f_{V,+},bc(\lambda)), \quad f_{V,+} \in \cS(\U_{V,+}(\bA))^{J_V},
    \end{equation}
where the sum ranges over representatives of classes in $\fX_\Pi(\U_V)$. Note that it is independent of the chosen representatives by the functional equation of Eisenstein series. This sum is absolutely convergent with uniform convergence on compact subsets by Proposition~\ref{prop:strong_K_basis}. In particular, it yields a meromorphic function in $\lambda$ which is holomorphic on $i \fa_{P}^{L_\pi,*}$.

\begin{theorem}
\label{thm:global_identity}
    Let $f_+ \in \cS(G_+(\bA))^J$ be of the form $f_+ = \Delta_H^{\tS,*} \Delta_{G'}^{\tS,*} f_{+,\tS} \times \prod_{v \notin \tS} 1_{G_+(\cO_v)}$ where $f_{+,\tS} \in \cS(G_+(F_\tS))$. For every $V \in \cH^\tS$, let $f_{V,+} \in \cS(\U_{V,+}(\bA))^{J_{V}}$, and for every $V \notin \cH^\tS$ set $f_+^V=0$. Assume that $f_+$ and $(f_{V,+})_V$ match. Then for every $\lambda_0 \in i \fa^{L_\pi,*}_{P}$ we have
    \begin{equation}
    \label{eq:global_identity}
        \sum_{V \in \cH^\tS} J_{V,\Pi}(f_{V,+},\lambda_0) = |S_\Pi|^{-1} I_{\Pi_{\lambda_0}}(f_+).
    \end{equation}
\end{theorem}

\begin{proof}
    The proof is very similar to~\cite{BPC22}*{Theorem~7.1.6.1} and uses multipliers. We recall below the main steps. The following two lemmas are~\cite{BPC22}*{Lemma~7.1.7.1} and~\cite{BPC22}*{Lemma~7.1.7.2}.

    \begin{lemma}
    \label{lem:existence_mult_gln}
        Let $V \in \cH^\tS$ and $\lambda_0 \in i \fa_{P}^{L_\pi,*}$ be in general position. There exists a multiplier $m_V \in \cM^{\tT}(\U_V(\bA))$ such that we have the following conditions.
        \begin{enumerate}
            \item For all $f_V \in \cS(\U_V(\bA))^{J_V}$, the
                right convolution $m_V*f_V$ sends $L^2([\U_V])$ into
            \begin{equation*}
                \widehat{\bigoplus}_{\chi \in \fX_\Pi(\U_V)} L^2_\chi([\U_V]).
            \end{equation*}
            \item For $(M_{Q},\sigma) \in \fX_\Pi(\U_V)$, we have
                $m_V(I_{Q}^{\U_V} (\sigma_{bc(\lambda_0)}))=1$.
        \end{enumerate}
    \end{lemma}

     \begin{lemma}
         \label{lem:existence_mult_unitary}
        Let $\lambda_0 \in i \fa_{P}^{L_\pi,*}$. There exists a multiplier $m \in \cM^{\tT}(G(\bA))$ such that we have the following conditions.
        \begin{enumerate}
            \item For all $f \in \cS(G(\bA))^{J}$, the right
                convolution $m*f$ sends $L^2(G(F) A_G(\bA) \backslash
                G(\bA))$ into $L^2_{\chi_0}([G])$.
            \item We have $m(\Pi_{\lambda_0})=1$.
        \end{enumerate}
    \end{lemma}

    Let $\lambda_0 \in i \fa_{P}^{L_\pi,*}$ in general position, and take the multipliers $m$ and $m_V$ from Lemma~\ref{lem:existence_mult_gln} and Lemma~\ref{lem:existence_mult_unitary} for $V \in \cH^\tS$. By~\cite{BPLZZ}*{Lemma~4.12}, we may assume that the $m_V$ are base changes of $m$ in the sense of~\cite{BPLZZ}*{Equation~(4.7)}. This implies that for every large enough finite set of places $\tS'$ disjoint from $\tT$ and every $\sigma_{\tS'} \in \Temp_{\U_V'}(\U_V(F_{\tS'}))$ we have
    \begin{equation}
        \label{eq:base_change}
        m_{\tS'}(\BC(\sigma_{\tS'}))=m_{V,{\tS'}}(\sigma_{\tS'}).
    \end{equation}

    \begin{lemma}
    \label{lemma:multiplier_transfer}
        Let $f_+ \in \cS(G_+(\bA))^{J}$ and $f_{V,+} \in \cS(\U_{V,+}(\bA))^{J_{V}}$  for $V \in \cH^\tS$. Set $f_{V,+}=m_V* f_{V,+}=0$ for $V \notin \cH^\tS$, and assume that $f_+$ and $(f_{V,+})_V$ match. Then $m*f_+$ and $(m_V* f_{V,+})_V$ also match.
    \end{lemma}

    \begin{proof}
    For every $v \notin \tS$ that is inert we have $(m*f_+)_v=1_{G_+(\cO_v)}$. It follows from the fundamental lemma (Theorem~\ref{thm:FL_FJ}) that for every $V \notin \cH^\tS$ the function $m*f_+$ matches with $0=m_V*f_{V,+}$. Assume now that $V \in \cH^\tS$. Let $\tS'$ be a finite set of places such that $m$ and the $m_V$ act trivially outside of $\tS'$ as in \eqref{eq:multiplier_finite_place}. It is enough to show that $m$ and $m_V$ preserve the matching at the places in $\tS'$. By Theorem~\ref{thm:spectral_transfer}, if $f_{+,\tS'}$ and $f_{V,+,\tS'}$ match, then they have matching local relative characters. By \eqref{eq:mult_fundamental} and \eqref{eq:base_change}, $m_{\tS'}*f_{+,\tS'}$ and $m_{V,{\tS'}}*f_{V,+,\tS'}$ also have matching relative characters. It follows from Theorem~\ref{thm:spectral_transfer} again that they match, which concludes the proof.
    \end{proof}

By comparing the coarse spectral expansions on the $G$-side (Theorem~\ref{thm:truncated_prop}) and on the $\U$-side (Theorem~\ref{thm:truncated_prop_u}) for the matching test functions $m*f_+$ and $m_V*f_{V,+}$ (Theorem~\ref{thm:geom_comparison}), we obtain the
spectral identity
\begin{equation}
\label{eq:global_transfer_mult}
    \sum_{\chi \in \fX(G)} I_{\chi}(m*f_+)=\sum_{V \in \cH^\tS} \sum_{\chi \in \fX(\U_V)} J_{V,\chi}(m_V*f_+^V),
\end{equation}
By Theorem~\ref{thm:truncated_prop} and Lemma~\ref{lem:existence_mult_gln}
for the LHS, and by Theorem~\ref{thm:truncated_prop_u} and
Lemma~\ref{lem:existence_mult_unitary} for the RHS,
\eqref{eq:global_transfer_mult} reduces to
\begin{equation*}
    I_{\chi_0}(m*f_+)=\sum_{V \in \cH^\tS}  \sum_{\chi \in \fX_\Pi(\U_V)} J_{V,\chi}(m_V*f_{V,+}).
\end{equation*}
By Theorems~\ref{thm:Ichi_spectral} and~\ref{thm:Jchi_spectral}, and an easy
change of variables, \eqref{eq:global_transfer_mult} reads
\begin{equation}
    2^{-\dim(\fa^*_{L_\pi})} \int_{i \fa_{P}^{L_\pi, *}} I_{\Pi_\lambda}(m*f_+) \rd\lambda=2^{-\dim \fa_{P}^{L_\pi,*}}
 \sum_{V \in \cH^\tS} \int_{i \fa_{P}^{L_\pi, *}} J_{V,\Pi}( m_V*f_{V,+},\lambda) \rd\lambda, \label{eq:global_integral}
\end{equation}
where $(M_Q,\sigma)$ is a representative of $\chi \in \fX_\Pi(\U_V)$. Note that all the integrals are absolutely convergent by Theorem~\ref{thm:Ichi_spectral} and Proposition~\ref{prop:IY_sum_convergence}. Moreover, by \eqref{eq:discrete_SPI_defi} we have $|S_\Pi|=2^{\dim(\fa^*_{L_\pi})-\dim \fa_{P}^{L_\pi,*}}$. 

We take $v_1$ and $v_2$ two finite places outside of $\tT$ of distinct residual characteristics. Let $\tS^\circ$ be any finite set of finite places of $F$ containing $v_1$ and $v_2$ and disjoint from $\tT$. Therefore, any place of $\tS^\circ$ splits in $E$ and is unramified in the sense of \S\ref{subsubsec:FL}. Consider the map $\cT_{\tS^\circ}: \lambda \in i \fa_{\Pi}^* \mapsto \Pi_{\lambda,\tS^\circ}$, where we denote by $\Pi_{\lambda,\tS^\circ}$ the product of the local components of $\Pi_\lambda$ at the places in $\tS^\circ$. By our assumptions on $v_1$ and $v_2$, it has uniformly bounded fibers. Note that then $\cT_{\tS^{\circ}}^{-1}(\{\Pi_{\lambda_0,\tS^\circ}\}) \subset \cT_{\{v_1,v_2\}}^{-1}(\{\Pi_{\lambda_0,\tS^\circ}\})$ and that we know that this set is finite. By the strong multiplicity one theorem of \cite{Ram}, for every $\lambda$ such that $\Pi_\lambda \not\simeq \Pi_{\lambda_0}$ there exists a place $v$ outside of $\tT$ such that $\Pi_{\lambda,v} \not\simeq \Pi_{\lambda_0,v}$. Therefore, we can choose $\tS^\circ$ such that for any $\lambda \in \cT_{\tS^\circ}^{-1}(\{\Pi_{\lambda_0,\tS^\circ}\})$ we have $\Pi_\lambda=\Pi_{\lambda_0}$. By \cite{JS}*{Theorem~4.4} and because $\Pi$ is $G$-regular and $\lambda_0$ is in general position, this implies that $\pi_\lambda=\pi_{\lambda_0}$ and moreover $\lambda=\lambda_0$.

Let $\cA_{\tS^\circ}$ be the spherical Hecke algebra $\otimes_{v \in \tS^\circ} \cH(G(F_v),K_v)$. Let $g \in \cA_{\tS^\circ}$. For each $V \in \cH^\tS$, write $g_V$ for the unramified split base change $g_V=\otimes_{v \in \tS^\circ} \BC_{V,v}(g)$ defined in \eqref{eq:Hecke_BC}. It belongs to $\otimes_{v \in \tS^\circ} \cH(\U_V(F_v),K_{V,v})$. Then the functions $f_+*g$ and $f_{V,+}*g_V$ (where we act by convolution in the first variable) still satisfy the requirements of Theorem~\ref{thm:geom_comparison}. By definition we have
\begin{equation*}
    J_{V,\Pi}(m_V * (f_{V,+}*g_V),\lambda)=J_{V,\Pi}(m_V * f_{V,+},\lambda)\widehat{g}(\Pi_{\lambda,\tS^\circ}),
\end{equation*}
where $\widehat{g}(\Pi_{\lambda,\tS^\circ})$ is the scalar by which $g$ acts on the unramified representation $\Pi_{\lambda,\tS^\circ}$. Therefore, if we set
\begin{equation*}
    h : \lambda \in i \fa_{P}^{L_\pi,*} \mapsto I_{\Pi_\lambda}(m*f_+)- |S_\Pi|^{-1} \sum_{V \in \cH^\tS} J_{V,\Pi}(m_V * f_{V,+},\lambda),
\end{equation*}
we see that \eqref{eq:global_integral} reads
\begin{equation}
    \label{eq:int_is_zero}
    \int_{i \fa_{P}^{L_\pi,*}} h(\lambda)\widehat{g}(\Pi_{\lambda,\tS^\circ}) \rd\lambda=0.
\end{equation}
The unramified unitary dual of $G(F_{\tS^\circ})$ is compact for the topology induced from the space of unramified characters of $T(F_{\tS^\circ})$ by the Satake isomorphism, (see \cite{Ta88}) and by the Stone--Weierstrass theorem, the subspace $\{ \widehat{g} \; | \; g \in \cA_{\tS^\circ} \}$ is dense in its space of continuous functions. Because $h$ is absolutely integrable and the fibers of $\cT_{\tS^\circ}$ are uniformly bounded, we see by truncating \eqref{eq:int_is_zero} on compacts that for almost all $\lambda \in i \fa_{P}^{L_\pi,*}$ we have
\begin{equation*}
    \sum_{\mu \in \cT_{\tS^\circ}^{-1}(\{\Pi_{\lambda,\tS^\circ}\})} h(\mu)=0.
\end{equation*}
By continuity of $h$, this holds for all $\lambda$ and, by our choice of $\tS^\circ$, at $\lambda_0$ it reduces to $h(\lambda_0)=0$. Moreover, by Lemma~\ref{lem:existence_mult_gln} and Lemma~\ref{lem:existence_mult_unitary} we have $I_{\Pi_{\lambda_0}}(m*f_+)=I_{\Pi_{\lambda_0}}(f_+)$ and $J_{V,\Pi}(m_V * f_{V,+},\lambda_0)=J_{V,\Pi}(f_{V,+},\lambda_0)$. Our $\lambda_0$ was assumed to be in general position and we conclude the proof by continuity.
\end{proof}

\section{Proof of the Gan--Gross--Prasad conjecture: Eisenstein case}

\label{sec:proof_Eisenstein}

In this chapter, we prove the GGP and Ichino--Ikeda conjectures for $(\U_V,\U_V',\overline{\mu})$-regular Eisenstein series from Theorem~\ref{thm:GGP-IY}.

\subsection{The Gan--Gross--Prasad conjecture for regular Hermitian Arthur parameters}

\label{sec:proof_GGP}

\subsubsection{Statement of the result}

Let $\Pi=I_{P}^{G} \pi$ be a $(G,H,\overline{\mu})$-regular Hermitian Arthur parameter, and let $\lambda \in i\fa_{P}^{L_\pi,*}=i\fa_\Pi^*$. The statement that we have to prove is the following.
\begin{theorem} 
    \label{thm:ggp_intro}
    The following are equivalent:
    \begin{enumerate}
    \item $L(\frac{1}{2}, \Pi_\lambda \otimes  \overline{\mu}) \not=0$; 
    \item there exist $V \in \cH$, a standard parabolic subgroup $Q \subset \U_V$ and $\sigma \in \Pi_{\cusp}(M_Q)$ such that $\Pi$ is the weak base change of $(Q,\sigma)$ and the Fourier--Jacobi period $\cP_{\U_V'}(\cdot,\lambda)$ does not vanish identically on $\cA_{Q,\sigma}(\U_V) \otimes \omega^\vee$.
    \end{enumerate}
    \end{theorem}

Here, $\cP_{\U'_V}(\cdot,\lambda)$ is the regularized Fourier--Jacobi period defined in Proposition~\ref{prop:IY_stability}. The proof follows from the comparison of trace formulae carried out in Theorem~\ref{thm:global_identity}.

\subsubsection{A reformulation of Theorem~\ref{thm:ggp_intro}}

By continuity, the relative character $I_{\Pi_\lambda}$ defined in \eqref{eq:IPi} is non-zero if and only if it is non-zero on pure tensors, which happens if and only if the Zeta function $Z(\cdot, \cdot, \overline{\mu}, 0)$ and $\beta_\eta$ are non-zero on $\Pi_\lambda \otimes \cS(\bA_{E,n})$ and $\cW(\Pi_\lambda,\overline{\psi}_N)$ respectively. By~\cite{JPSS83} and~\cite{Jac04}, $Z(\cdot, \cdot, \overline{\mu}, 0)$ is non-zero if and only if $L(\frac{1}{2},\Pi_\lambda \otimes \overline{\mu}) \neq 0$. By~\cite{GK75}, ~\cite{Jac10}*{Proposition~5} and~\cite{Kem15}, $\beta_\eta$ is always non-zero on $\cW(\Pi_\lambda,\overline{\psi}_N)$.

On the unitary side, let $V \in \cH$. Let $Q$ be a standard parabolic subgroup of $\U_V$ and let $\sigma \in \Pi_{\cusp}(M_Q)$. Assume that $\Pi$ is the weak base change of $(Q,\sigma)$. It follows from \eqref{eq:P_abs_conv} that $J_{Q,\sigma}(\cdot,bc(\lambda))$ is non-zero if and only if the period $\cP_{\U_V'}(\cdot,\cdot,bc(\lambda))$ is non-zero on $\cA_{Q,\sigma}(\U_V) \otimes \cS(Y^\vee(\bA))$. Therefore, Theorem~\ref{thm:ggp_intro} amounts to the equivalence of the following two statements.
\begin{itemize}
    \item[(A)] The distribution $I_{\Pi_\lambda}$ is non-zero.
    \item[(B)] There exists a $V$, a parabolic subgroup $Q$ of $\U_V$ and $\sigma \in \Pi_{\cusp}(M_Q)$ such that $\Pi$ is the weak base change of $(Q,\sigma)$ and $J_{Q,\sigma}(\cdot,bc(\lambda))$ is non-zero.
\end{itemize}

\subsubsection{Proof of $(A)\implies (B)$}
Let $\tS$ be a sufficiently large finite set of places of $F$ containing the Archimedean and ramified ones. For each $v \in \tS \setminus V_{F,\infty}$ take $J_v \subset G(F_v)$ to be an open compact subgroup, and for each $v \notin \tS$ let $J_v$ be a maximal open compact subgroup of $G(F_v)$ Assume that $I_{\Pi_\lambda}$ is non-zero on $\cS(G_+(\bA))^J$ (which implies that $\Pi_\lambda$ has non-zero $J$-invariant vectors). As $I_{\Pi_\lambda}$ is continuous, by the existence of transfer in the non-Archimedean case and its approximation in the Archimedean case (Theorem~\ref{thm:matching}) and the fundamental Lemma (Theorem~\ref{thm:FL_FJ}), up to enlarging $\tS$ and shrinking $J$ we may choose $f_+ \in \cS(G_+(\bA))^J$ and a collection of $f_{V,+} \in \cS(G_{V,+}(\bA))^{J_{V}}$ for $V \in \cH^\tS$ which satisfy the hypotheses of Theorem~\ref{thm:global_identity} and such that $I_{\Pi_\lambda}(f_+) \neq 0$. The result now follows from the comparison of global relative trace formulae in Theorem~\ref{thm:global_identity}.

\subsubsection{Proof of $(B)\implies (A)$}
Choose $\tS$ as before, $V \in \cH^\tS$, $J_{V,v}' \subset \U_V(F_v)$ an open compact subgroup, $Q$ a standard parabolic subgroup of $\U_V$, $\sigma \in \Pi_{\cusp}(M_Q)^{J_V}$ and $g_{V,+}=g_{V} \otimes \phi_1 \otimes \phi_2 \in \cS(G_{V,+}(\bA)^{J_V}$ such that the cuspidal datum represented by $(M_Q,\sigma)$ belongs to $\mathfrak{X}_\Pi(\U_V)$ and $J_{Q,\sigma}(g_{V,+},bc(\lambda))\neq 0$. Set $f_{V,+}=g_V*g_{V}^\vee \otimes \phi_1 \otimes \overline{\phi_1}$. We get $J_{Q',\sigma'}(f_{V,+},bc(\lambda)) \geq 0$ for all pairs $(Q',\sigma')$ in $\mathfrak{X}_\Pi(\U_V)$, and moreover $J_{Q,\sigma}(f_{V,+},bc(\lambda))> 0$. This yields $J_{V,\Pi}(f_{V,+},\lambda)>0$. For $V' \in \cH^\tS$ different from $V$, set $f_{V',+}= 0$. Then we have 
\begin{equation*}
    \sum_{V' \in \cH^{\tS}} J_{V',\Pi}(f_{V',+},\lambda)>0.
\end{equation*}
As the LHS is continuous, by the existence of non-Archimedean transfer and the approximation of smooth-transfer (Theorem~\ref{thm:matching}) and the fundamental Lemma (Theorem~\ref{thm:FL_FJ}), we may assume that there exists a level $J$ of $G$ and $f_+ \in \cS(G_+(\bA))^{J}$ such that $f_+$ and $(f_{V',+})$ satisfy the hypotheses of Theorem~\ref{thm:global_identity} and $\sum_{V' \in \cH^{\tS}} J_{V',\Pi}(f_{V',+},\lambda) \neq 0$. The result now follows again from the comparison of global relative trace formulae in Theorem~\ref{thm:global_identity}.

\subsection[The Ichino--Ikeda conjecture]{The Ichino--Ikeda conjecture for regular Hermitian Arthur parameters}
\label{sec:proof_II}

\subsubsection{Statement of the result}

We fix the skew-Hermitian space $V$. Let $Q$ a standard parabolic subgroup of $\U_V$, and $\sigma \in \Pi_{\cusp}(M_Q)$. We write the factorization into local components $\sigma=\otimes'_v \sigma_v$ and assume that all the $\sigma_v$ are tempered. For any $\lambda \in i \fa_Q^*$, set $\Sigma_\lambda=I_Q^{\U_V}\sigma_\lambda$, so that for every $v$ the local component $\Sigma_{\lambda,v}$ is tempered. 

We assume that $(Q,\sigma)$ admits a weak base-change $\Pi$ which is a $(G,H,\overline{\mu})$-regular Hermitian Arthur parameter.

Set $\mathrm{As}^{n,n}=\mathrm{As}^{(-1)^{n}} \oplus \mathrm{As}^{(-1)^{n}}$. Consider the global $L$ factor 
\begin{equation*}
    \cL(s,\Sigma_\lambda)=(s-\frac{1}{2})^{-\dim \fa_\Pi^*}\frac{L(s,\Pi_\lambda \otimes \overline{\mu})}{L(s+\frac{1}{2},\Pi_\lambda,\mathrm{As}^{n,n})},
\end{equation*}

\begin{remark}
\label{rk:L_pole}
    If $\pi_1$ and $\pi_2$ are irreducible automorphic representations of $G_{n_1}(\bA)$ and $G_{n_2}(\bA)$ respectively with $n_1+n_2=n$, and if $P$ is a standard parabolic subgroup of $G_n$ with Levi factor $G_{n_1} \times G_{n_2}$, then we have the induction formula
    \begin{equation}
    \label{eq:induction_formula}
        L(s,I_P^{G_n}(\pi_1 \boxtimes \pi_2),\mathrm{As}^{(-1)^n})=L(s,\pi_1 \times \pi_2 ^{c})     L(s,\pi_1,\mathrm{As}^{(-1)^n})L(s,\pi_2,\mathrm{As}^{(-1)^n}).
    \end{equation}
    By \cite{BPCZ}*{Section~4.1.2}, if $\pi$ is an automorphic cuspidal representation which is non-isomorphic to $\pi^*$, then $L(s,\pi,\mathrm{As}^{(-1)^n})$ is regular at $s=1$. Moreover, if $\Pi_0$ is a discrete Hermitian Arthur parameter of $G_{n_0}$ with $n-n_0$ even, induction on \eqref{eq:induction_formula} and Remark \ref{rem:semi_discrete_regular} show that $L(s,\Pi_0,\mathrm{As}^{(-1)^n})$ is also regular at $s=1$. Therefore, our hypotheses on our Hermitian Arthur parameter $\Pi$ imply that the meromorphic function $L(s,\Pi_\lambda,\mathrm{As}^{n,n})$ has a pole of order $\dim \fa_\Pi^*$ at $s=1$. In particular, $\cL(s,\Sigma_\lambda)$ is holomorphic at $s=\frac{1}{2}$. 
\end{remark}

We denote by $\cL(s,\Sigma_{\lambda,v})$ the corresponding local $L$-factor, so that for $\Re(s)$ large enough we have
\begin{equation*}
    \cL(s,\Sigma_\lambda)=(s-\frac{1}{2})^{-\dim \fa_\Pi^*}\prod_{v \in V_F} \cL(s,\Sigma_{\lambda,v}),
\end{equation*}
As $\Sigma_{\lambda,v}$ is tempered, $\cL(s,\Sigma_{\lambda,v})$ has neither a pole nor a zero at $s=\frac{1}{2}$. 

Let $v$ be a place of $F$. In \S\ref{subsubsec:local_FJ}, we have introduced a local period $\cP_{\U_V',v}$ on $\Sigma_{\lambda,v} \otimes \omega^\vee_{v}$. It depends on a choice of measure on $\U_V'(F_v)$, and of inner products on $\Sigma_{\lambda,v}$ and $\omega^\vee_{v}$. We use the measure $\rd g_v$ defined in \S\ref{subsec:measure_part1}. Note that $ \rd g = (\Delta_{\bG}^*)^{-1} \prod_v \Delta_{\bG,v} \rd g_v$. For the inner product on $\Sigma_{\lambda,v}$, we take the one arising from a factorization of $\langle\cdot,\cdot \rangle_{Q,\Pet}$. For the one on $\omega^\vee_v$, we take a factorization of the inner product $\langle \cdot,\cdot \rangle_{L^2}$ on $\omega^\vee$ given by integration on $Y^\vee(\bA)$ using the Tamagawa measure. We now define for $\varphi_v \in \Sigma_{\lambda,v}$ and $\phi_v \in \omega^\vee_{v}$ the local normalized period
\begin{equation*}
    \cP^\sharp_{\U_V',v}(\varphi_v,\phi_v,\lambda):=\cL(\frac{1}{2},\Sigma_{\lambda,v})^{-1}  \cP_{\U_V',v}(\varphi_v,\phi_v,\lambda).
\end{equation*}

The version of the Ichino--Ikeda conjecture that we prove is the following.

\begin{theorem}
    \label{thm:II_regular_intro}
    Let $\Pi$ and $(Q,\sigma)$ be as above. For $\lambda \in i \fa_\Pi^*$ and every non-zero and factorizable $\varphi =\otimes_v' \varphi_v \in \cA_{Q,\sigma}(\U_V)$ and $\phi=\otimes_v' \phi_v \in \omega^\vee$, we have
    \begin{equation}
    \label{eq:IY_regular_intro}
        |\cP_{\U_V'}(\varphi,\phi,\lambda)|^2=|S_\Pi|^{-1} \cL(\frac{1}{2},\Sigma_\lambda) \prod_v \cP^\sharp_{\U_V',v}(\varphi_v,\phi_v,\lambda).
    \end{equation}
\end{theorem}

\begin{remark}
    It follows from the unramified Ichino--Ikeda conjecture \cite{Xue2}*{Proposition~1.1.1} that almost all factors in the RHS are equal to $1$.
\end{remark}

The proof once again follows from the comparison of trace formulae carried out in Theorem~\ref{thm:global_identity}.

\subsubsection{Factorization of relative characters on the $G$-side}

Let $\varphi \in \Pi$ and $\lambda \in i \fa_P^*$. Write $W^\psi(E(\varphi,\lambda))=\otimes_v W_v$ where $W_v \in \cW(\Pi_{\lambda,v},\psi_{N,v})$ satisfies $W_v(1)=1$ for almost all $v$. Recall that we have equipped these local Whittaker spaces with an inner product $\langle \cdot,\cdot \rangle_{\mathrm{Whitt},v}$ in \S\ref{subsec:local_spheri}. By~\cite{BPCZ}*{Theorem~8.1.2.1}, there exists $\tS$ a finite set of places of $F$ such that
\begin{equation*}
    \langle \varphi,\varphi \rangle_{P,\mathrm{Pet}}=(\Delta_{G}^{\tS,*})^{-1}L^{\tS,*}(1,\Pi_{\lambda},\mathrm{Ad}) \prod_{v \in \tS} \langle W_v,W_v \rangle_{\mathrm{Whitt},v}.
\end{equation*}
Moreover, up to enlarging $\tS$, by the unramified computations of \cite{Zhang2}*{Section~3} and \cite{JS}, and by \cite{Fli}, we have for $\Phi=\otimes_v \Phi_v \in \cS(\bA_{E,n})$ 
\begin{align*}
     Z(W^\psi(E(\varphi,\lambda)), \Phi, \overline{\mu}, 0)&=(\Delta_{G}^{\tS,*})^{-1}(\Delta_H^{\tS,*})^{-1} L^{\tS}(\frac{1}{2},\Pi_\lambda \otimes \overline{\mu})\prod_{v \in \tS} Z_v(W_v,\Phi_v,\overline{\mu},0), \\
     \beta_\eta(W^\psi(E(\varphi,\lambda)))&=(\Delta_{G'}^{\tS,*})^{-1}L^{\tS,*}(1,\Pi_{\lambda},\mathrm{As}^{n+1,n+1}) \prod_{v \in \tS} \beta_{\eta,v}(W_v).
\end{align*}
It now follows from the equality of $L$ functions $L(s,\Pi_{\lambda},\mathrm{Ad})=L(s,\Pi_{\lambda},\mathrm{As}^{n+1,n+1})L(s,\Pi_{\lambda},\mathrm{As}^{n,n})$ and from the definition of the local spherical character in \eqref{eq:local_spheri_char_gln} that for a factorizable test function $f_+= (\Delta_H^{\tS,*} \Delta_{G'}^{\tS,*}) \prod_{v \in \tS}f_{+,v} \times \prod_{v \notin \tS} 1_{G_+(\cO_v)}$ we have
\begin{equation}
\label{eq:global_facto}
    I_{\Pi_\lambda}(f_+)=\frac{L^{\tS}(\frac{1}{2},\Pi_\lambda \otimes \overline{\mu})}{L^{\tS,*}(1,\Pi_\lambda,\mathrm{As}^{n,n})} \prod_{v \in \tS} I_{\Pi_\lambda,v}(f_{+,v}).
\end{equation}

\subsubsection{Unramified computations of spherical characters on unitary groups}

Up to enlarging $\tS$, by the unramified local Ichino--Ikeda conjecture \cite{Xue2}*{Proposition~1.1.1}, for $v \notin \tS$ and $f_{V,+,v}=f_{V,v} \otimes \phi_{1,v} \otimes \phi_{2,v}=1_{G_{V,+}(\cO_{F,v})}$ we have 
\begin{equation}
\label{eq:unram_J}
    J_{\Sigma_{\lambda,v}}(f_{V,+,v})=\Delta_{\U(V),v}^{-2} \frac{L(\frac{1}{2},\Pi_{\lambda,v} \otimes \overline{\mu})}{L(1,\Pi_{\lambda,v},\mathrm{As}^{n,n})}.
\end{equation}
Note that the square factor $\Delta_{\U(V),v}^{-2}$ comes from the fact that if $\varphi^\circ$ is a spherical vector in $\Sigma_{\lambda,v}$, we have $\Sigma_{\lambda,v}(f_{V,v}) \varphi^\circ= \Delta_{\U(V),v}^{-2} \varphi^\circ$ by our choice of measure as $\vol(\U_V(\cO_v),\rd g_v)=\Delta_{\U(V),v}^{-2}$.

\subsubsection{End of the proof}

If there exists a place $v \in \tS$ such that $\Sigma_{\lambda,v}$ doesn't have any non-zero continuous $\U_V'(F_v)$-invariant functional, then both sides of \eqref{eq:IY_regular_intro} are automatically zero. We now assume that it is not the case, i.e. that for all $v$ we have $\Sigma_{\lambda,v} \in \Temp_{\U_V'}(\U_V(F_v))$. It follows from \cite{Xue2}*{Proposition~1.1.1} that the product of distributions $\prod_{v \in \tS} J_{\Sigma_{\lambda,v}}$ doesn't vanish identically. By the unramified computation \eqref{eq:unram_J} and by the same argument as~\cite{Zhang2}*{Lemma~1.7}, Theorem~\ref{thm:II_regular_intro} is now equivalent to the following assertion: for every factorizable test function $f_{V,+} \in \cS(G_{V,+}(\bA))$ of the form $f_{V,+}=(\Delta_{\U(V)}^{\tS})^2 \prod_{v \in \tS} f_{V,+,v} \times \prod_{v \notin \tS} 1_{G_{V,+}(\cO_v)}$, we have
\begin{equation}
\label{eq:reformulation}
    J_{Q,\sigma}(f_{V,+},bc(\lambda))=\valP{S_\Pi}^{-1} \frac{L^{\tS}(\frac{1}{2},\Pi_\lambda \otimes \overline{\mu})}{L^{\tS,*}(1,\Pi_\lambda, \mathrm{As}^{n,n})}\prod_{v \in \tS}J_{\Sigma_{\lambda,v}}(f_{V,+,v}).
\end{equation}

Let $f_{V,+}$ be as in \eqref{eq:reformulation}. For each $V' \neq V$, set $f_{V',+}=0$. By Theorem~\ref{thm:matching} and because both sides of \eqref{eq:reformulation} are continuous, we may assume that for all $v \in \tS$ there exists $f_{+,v} \in \cS(G_+(F_v))$ such that $f_{+,v}$ matches with $(f_{V',+,v})_{V'}$. We now set $f_+= (\Delta_H^{\tS,*} \Delta_{G'}^{\tS,*}) \prod_{v \in \tS}f_{+,v} \times \prod_{v \notin \tS} 1_{G_+(\cO_v)}$. By Theorem~\ref{thm:global_identity}, we get
 \begin{equation}
 \label{eq:J_Pi=I_Pi}
 \sum_{(M_{Q'},\sigma') \in \fX_\Pi(\U_V)} J_{Q',\sigma'}(f_{V,+},bc(\lambda))= |S_\Pi|^{-1} I_{\Pi_\lambda}(f_+).
\end{equation}
By definition $J_{Q',\sigma'}$ is zero unless for every place $v$ the dimension of the space of Fourier--Jacobi functionals $\Hom_{\U'_V(F_v)}(I_{Q'}^{G_{V}} \sigma'_v \otimes \omega^\vee_v,\C)$ is $1$. By the local Gan--Gross--Prasad conjecture from \cite{GI2} and \cite{Xue6}, there is at most one element in each $L$-packet such that this is the case. But by the compatbility of base change with respect to the local Langlands correspondence, all the $I_{Q'}^{G_{V}} \sigma'_v$ arising from \eqref{eq:J_Pi=I_Pi} must be in the same $L$-packet at all places $v$. Therefore, all the terms in this sum must disappear except possibly $J_{Q,\sigma}(f_{V,+},bc(\lambda))$, so that \eqref{eq:J_Pi=I_Pi} reduces to
\begin{equation}
\label{eq:global_equality_characters}
    J_{Q,\sigma}(f_{V,+},bc(\lambda)) = |S_\Pi|^{-1} I_{\Pi_\lambda}(f_+).
\end{equation}
By Theorem~\ref{thm:spectral_transfer}, for all place $v$ there exist constants $\kappa_{V,v} \in \C^\times$ such that $\prod_{v \in \tS} \kappa_{V,v}=1$ and
\begin{equation}
\label{eq:local_equality_characters}
J_{\Sigma_\lambda,v}(f_{V,+,v})=\kappa_{V,v}I_{\Pi_{\lambda,v}}(f_{+,v}).
\end{equation}
Therefore, \eqref{eq:reformulation} follows from the global equality \eqref{eq:global_equality_characters}, the factorization \eqref{eq:global_facto} and the local equality \eqref{eq:local_equality_characters}. This concludes the proof of Theorem~\ref{thm:II_regular_intro}.

\part{Fourier--Jacobi periods in positive corank}
\label{part:positive}

In Part~\ref{part:positive}, we prove the GGP and Ichino--Ikeda conjectures for Fourier--Jacobi periods in positive corank (Theorems~\ref{thm:GGP} and \ref{thm:II}). The bulk of our work is devoted to the proof of unfolding identities. For local periods, it is carried out in \S\ref{sec:local_FJ}, while for the global ones it is the focus of \S\ref{sec:global_FJ}. We finally explain in \S\ref{sec:proofs_positive} how this reduces our main results to Theorem~\ref{thm:GGP-IY} proved in Part~\ref{part:corank_zero}.

\section{Groups and representations}

In this section, we describe precisely the setting of the Gan--Gross--Prasad conjecture in arbitrary corank. We will freely use the notations and notions introduced in \S\ref{sec:notation} for the corank-zero case. 

\subsection{Fourier--Jacobi groups}

\subsubsection{Skew-Hermitian spaces, unitary groups} \label{subsubsec:parabolic_subgroup}  Let $(V,q_V)$ be a $n$-dimensional nondegenerate skew-Hermitian space over $E/F$. Let $W \subset V$ be a subspace of dimension $m$ such that $n=m+2r$. We will refer to the integer $r$ as the \emph{corank} of the pair $W \subset V$. We will also always assume that $W^\perp$ is split. This means that there exists a basis $\{ x_i, \; x_i^*  \; | \; i=1,\hdots,r \}$ of $W^\perp$ as an $E$-vector space such that for all $1 \leq i,j \leq r$ we have
\begin{equation}
\label{eq:hyperbolic}
    q_V(x_i,x_j)=0, \quad q_V(x_i^*, x_j^*)=0, \quad q_V(x_i,x_j^*)= \delta_{i,j}.
\end{equation}
Set $X=\mathrm{span}_E(x_1,\hdots,x_r)$ and $X^*=\mathrm{span}_E(x_1^*,\hdots,x_r^*)$, so that $V=X \oplus W \oplus X^*$. We will also always assume that the corank $r$ is greater than $1$.

We will identify the unitary group $\U(W)$ as the subgroup of $\U(V)$ of $g \in \U(V)$ that act by the identity on $W^\perp$. Furthermore, we define the following products
\begin{equation*}
    \U_V:=\U(V) \times \U(V), \quad {\cU_W}:=\U(V) \times \U(W), \quad \U_V':=\U(V) \subset \U_V,
\end{equation*}
where the last embedding is the diagonal injection.

Let $P(X)=M(X)N(X)$ be the maximal parabolic subgroup of $\U(V)$ stabilizing $X$, where $M(X)$ is the Levi subgroup stabilizing $X^*$ and $N(X)$ is the unipotent radical of $P(X)$. We identify $M(X)$ with $G_r \times \U(W)$ by restricting to $X \oplus W$, where we recall that $G_r=\Res_{E/F} \GL_r$. We can write elements of $P(X)$ as $m(a) g_W  n(b) n(c)$ where $g_W \in
\U(W)$, $a \in \GL(X)$, $b \in \Hom(W, X)$, $c \in \mathrm{Herm}(X^*, X)$, and
    \[
    m(a) = \begin{pmatrix} a \\ & 1_{W} \\ && (a^*)^{-1} \end{pmatrix},
    \quad
    n(b) = \begin{pmatrix} 1_{X} & b & -\frac{1}{2} b b^* \\
    & 1_{W} & -b^* \\ && 1_{X^*} \end{pmatrix},
    \quad
    n(c) = \begin{pmatrix} 1_{X} & 0 &c \\ & 1_{W} & 0 \\
    && 1_{X^*} \end{pmatrix},
    \]
where we define
    \begin{equation}
        \label{eq:Herm_defi}
    \mathrm{Herm}(X^*, X) = \{ c \in \Hom(X^*, X) \mid c^* = -c \}.
    \end{equation}
Here $a^* \in \GL(X^*)$, $b^* \in \Hom(X^*, W)$ and $c^* \in
\Hom(X^*, X)$ are defined by
    \[
    q( ax, x^*)_{V} = q(x, a^* x^* )_{V}, \quad
    q(b v, x^*)_{V} = q( v, b^* x^*)_V, \quad
    q( c x^*, y^* )_V = q( x^*, c^* y^* )_V,
    \]
for $v \in V$, $x \in X$ and $x^*, y^* \in X^*$.

The subgroup of $\U(V)$ of $g \in \U(V)$ stabilizing $X$ and $X^*$ and being trivial on $W$ is identified with $G_r$ by restricting to $X$. Let $(T_r,B_r)$ be the standard Borel pair of diagonal and upper triangular matrices of $G_r$ with respect to $(x_1,\hdots,x_r)$. We denote by $P_r$ the mirabolic subgroup of $G_r$

For each $1 \leq k \leq r$, we embed the group $G_k$ into $G_r$ as the group of automorphisms of $\mathrm{span}_E(x_1,\hdots,x_k)$. We have the standard Borel pair $(T_k,N_k)$.

\subsubsection{Heisenberg, Jacobi and Fourier--Jacobi groups}
We take $\bW=Y\oplus Y^\vee$ a polarization of the symplectic space $(\bW,\Tr_{E/F} \circ q_V)$. We also obtain a polarization $\bV=(X \oplus Y) \oplus (X^* \oplus Y^\vee)$. Write $q(\cdot,\cdot)_\bV$ for the symplectic pairing of $\bV$. We will consider $X$ and $X^*$ as $E$-affine spaces, and $Y$ and $Y^\vee$ as $F$-affine spaces. 

We denote by $S(W)$ the Heisenberg group of $W$, and by $ J(W):=S(W) \rtimes \U(W)$ its Jacobi group. We have an injective morphism $h : S(W) \to \U(V)$ (see e.g. \cite{Boi}*{Equation~(2.1.5.1)}) so that we identify $S(W)$ as a subgroup of $\U(V)$. 

Let $U_{r-1}$ and $U_r$ be the unipotent radicals of the parabolic subgroups of $\U(V)$ stabilizing the flags $\mathrm{span}_E(x_1) \subset \mathrm{span}_E(x_1,x_2) \subset \hdots \subset \mathrm{span}_E(x_1,\hdots,x_{r-1})$, and the full flag $\mathrm{span}_E(x_1) \subset \hdots \subset X$ respectively. Define the \emph{Fourier--Jacobi group}
\begin{equation*}
    \cH_W:=U_{r-1} \rtimes J(W) \subset \cU_W,
\end{equation*}
where the embedding is the product of the inclusion $\cH_W \subset \U(V)$ and the projection $\cH_W \to \U(W)$. Note that $U_{r-1} \rtimes S(W)=U_r$, so that we have the alternative definition $\cH_W=U_r \rtimes \U(W)$. We will write $\cH'_W$ if we consider it as a subgroup of $\U_V'$ via the inclusion. 

Set $L:=\cU_W \times G_r$. It is a Levi subgroup of the parabolic $\mathrm{U}(V) \times P(X)$ of $\U_V$. Define
\begin{equation*}
    \cH_L=\cH_W \times N_r \subset L.
\end{equation*}

\subsection{Heisenberg--Weil representations}

\subsubsection{Definition}
Define a morphism $\lambda : U_{r-1} \to \bG_a$ by
\begin{equation}
\label{eq:lambda_N}
     \lambda(u)=\Tr_{E/F} \left(\sum_{i=1}^{r-1}  q_V(u(x_{i+1}),x_{i}^*) \right), \; u \in U_{r-1}.
\end{equation}
Note that the action of $J(W)$ by conjugation on $U_{r-1}$ is trivial on $\lambda$, so that we may extend $\lambda$ to $\cH_W=U_{r-1} \rtimes J(W)$.

Let $v$ be a place of $F$. Consider the character
\begin{equation}
\label{eq:psi_U}
    \psi_{U,v}(h):=\psi_{E,v} (\lambda(h)), \quad h \in \cH_W(F_v).
\end{equation}
Note that $U_r$ contains $N_r$. For each $1 \leq k \leq r$, we denote by $\psi_{k,v}$ the restriction of $\psi_{U,v}$ to $N_k(F_v)$. 

We have the local Heisenberg representation $\rho^\vee_{\psi,v}$ of $S(W)(F_v)$ defined in \S\ref{subsubsec:local_Heisenberg}. We also have the Weil representations $\omega_{W,v}^\vee$ and $\omega_{V,v}^\vee$ of $\U(W)(F_v)$ and $\U(V)(F_v)$ respectively from \S\ref{subsubsec:local_Heisenberg}. The first is realized on $\cS(Y^\vee(F_v))$, and the second on $\cS(X^*(E_v) \oplus Y^\vee(F_v))$. They are all defined with respect to the pair $(\psi_v^{-1},\mu_v^{-1})$.

We can define the representation $\nu^\vee_v$ of $\cH_W(F_v)$ by the rule
\begin{equation}
\label{eq:nu_defi2}
    \nu^\vee_v(u h g_W)=\overline{\psi_{U,v}}(u) \rho^\vee_{\psi,v}(h) \omega^\vee_{W,v}(g_W) , \quad u \in U_{r-1}(F_v), \quad g_W \in \U(W)(F_v), \quad h \in S(W)(F_v).
\end{equation}
This is the \emph{Heisenberg--Weil representation} of $\cH_W(F_v)$.

We will also consider $\nu_{L,v}^\vee$ the representation of $\cH_L(F_v)$ defined by
    \begin{equation}
        \label{eq:nu_L}
        \nu_{L,v}^\vee:=\nu_{v}^\vee \otimes \psi_{r,v}.
    \end{equation}

\subsubsection{Explicit formulae in the mixed model} 
\label{subsubsec:explicit_Weil}

Let $\phi \in \nu^\vee_v$. By \cite{Kudla}, for $y^\vee_0 \in Y^\vee(F_v)$, $y \in Y(F_v)$, $y^\vee \in Y^\vee(F_v)$ and $z \in F_v$ we have
    \begin{equation}
        \label{eq:Heisenberg_action}
        \rho_{\psi,v}^\vee((y+y^\vee,z))\phi(y_0^\vee)=\overline{\psi_v}\left(z+q(y_0^\vee,y) _{\bW}+\frac{1}{2}q(y^\vee,y )_{\bW}\right)\phi(y_0^\vee+y^\vee).
    \end{equation}

Let $\Phi \in \omega_{V,v}^\vee$. We identify $\Phi$ with a Schwartz function on $X^*(F_v)$ valued in $\omega_{W,v}^\vee$. If $a \in G_r(F_v)$, we write $\overline{\mu_v}(a)$ for $\overline{\mu_v}(\det a)$. Let $x^* \in X^*(F_v)$. With the same elements as in \S\ref{subsubsec:parabolic_subgroup}, we have the formulae

    \begin{align}
    \omega_{V,v}^\vee(g_W)\Phi(x^*) &= \omega_{W,v}^\vee(g_W)(\Phi(x^*)) \label{eq:weil_formula_P1},\\
    \omega^\vee_{V,v}(m(a)) \Phi(x^*) & = \overline{\mu_v}(a) \valP{\det a}^{\frac{1}{2}}
    \Phi(a^* x^*) \label{eq:weil_formula_P2},\\
    \omega^\vee_{V,v}(n(b))\Phi(x^*) &= \rho_{\psi,v}^\vee((b^* x^*, 0)) (\Phi(x^*)) \label{eq:weil_formula_P3} ,\\
    \omega_{V,v}^\vee(n(c))\Phi(x^*) & = \overline{\psi}_{E,v}(q(cx^*, x^* )_{V}) \Phi(x^*) \label{eq:weil_formula_P4}.
    \end{align}
    
Consider the map
    \begin{equation}
        \label{eq:Phi_W_defi}
        \Phi \in \omega_{V,v}^\vee \mapsto \Phi_{Y^\vee} := \Phi(x_r^*) \in \nu_v^\vee.
    \end{equation}
    Using the formulae \eqref{eq:Heisenberg_action} to \eqref{eq:weil_formula_P4}, we see that we have for $n_r, n_r' \in N_r(F_v), \; n \in N(X)(F_v), \; g_W \in \mathrm{U}(W)(F_v)$, and $ g_V \in \mathrm{U}(V)(F_v)$ the relation 
    \begin{equation}
        \label{eq:nu_relation}
            \psi_{r,v}(n_r') \overline{\psi}_{r,v}(n_r) (\omega_{V,v}^\vee(n_r' ng_Wg_V) \Phi)_{Y^\vee}=\nu_{L,v}^\vee (n_rng_W,n_r')(\omega_{V,v}^\vee(g_V) \Phi)_{Y^\vee}.
        \end{equation}

\subsubsection{Automorphic Heisenberg--Weil representations} As in \S\ref{subsubsec:automorphic_Weil}, we obtain automorphic representations $\omega_V^\vee$, $\omega_W^\vee$, $\nu^\vee$ and $\nu_L^\vee$ of $\U(V)(\bA)$, $\U(W)(\bA)$, $\cH_W(\bA)$ and $\cH_L(\bA)$ by taking restricted tensor product. We also have the automorphic characters $\psi_U$ and $\psi_r$.

\begin{remark}
    At split places $v$, the representation $\omega^\vee_{V,v}$ is realized on $\cS(X_0^*(E_v) \oplus Y^\vee(F_v))=\cS(F_v^n)$ with the action described in \S\ref{subsubsec:explicit_Weil}, which comes from our choice of global polarization. In particular, it is not the usual left-regular representation twisted by $\mu \valP{\cdot}^{1/2}$ considered for example in \cite{GGP}*{Section~13}. 
\end{remark}

\subsection{Measures}
\label{subsec:final_measures}

For every algebraic group $\bG$ over $F$ we have considered so far, we let $\rd g$ be the Tamagawa measure on $\bG(\bA)$ defined in Subsection~\ref{subsec:tamagawa_measure_new}. We fix a factorization $\rd g = \prod_v \rd g_v$ on $\bG(\bA)$ such that for almost all $v$ the volume of $\bG(\cO_v)$ is $1$. Note that this differs from the choice made in \S\ref{subsec:tamagawa_measure_new} and \S\ref{subsec:measure_part1}, but it is more convenient for this section (in particular the unramified computations of \S\ref{subsec:unram_unfold}).

For unipotent groups, it is convenient to make a specific choice for the local measures $\rd g_v$. For every place $v$ of $E$, we set $\rd_{\psi_v} x_v$ to be the unique Haar measure on $E_v$ that is self-dual with respect to $\psi_{E,v}$. This yields a measure $\rd x = \prod_v \rd_{\psi_v} x_v$ on $\bA_E$. More generally, for every $k$ we equip $N_k(F_v)$ with the product measure
\begin{equation*}
    \rd n_v= \prod_{1 \leq i < j \leq k} d_{\psi_v} n_{i,j},
\end{equation*}
and set $\rd n = \prod_v \rd n_v$. Then it is well known that $\rd n_v$ gives volume $1$ to $N_k(\cO_v)$ for almost all $v$, and that $\rd n$ is indeed the Tamagawa measure on $N_k(\bA)$.

For every $k$, the measure $\rd_{\psi_v} g_v$ on $G_k(F_v)$ from \S\ref{subsec:tamagawa_measure_new} can be taken to be
\begin{equation*}
    \rd_{\psi_v} g_v := \frac{\prod_{1 \leq i,j \leq k} \rd_{\psi_v} g_{i,j}}{\valP{\det g}_v^k}.
\end{equation*}
We denote by $\upsilon(G_{k,v}) >0$ the quotient of Haar measures $\rd g_v (\rd_{\psi_v} g_v)^{-1}$. It follows from \eqref{eq:Delta} that for $\tS$ any sufficiently large finite set of places of $F$ we have $\upsilon(G_{k,v})=\Delta_{G_k,v}$ for $v \notin \tS$, and that
\begin{equation}
\label{eq:normalization_measures}
    \prod_{v \in \tS} \upsilon(G_{k,v})=(\Delta_{G_k}^{\tS,*})^{-1}.
\end{equation}

By integration along $X^*(\bA) \oplus Y^\vee(\bA)$ and $Y^\vee(\bA)$ against the Tamagawa measure, we obtain invariant inner products on $\omega_V^\vee$ and $\nu^\vee$ which we both denote by $\langle \cdot , \cdot\rangle_{L^2}$. For every place $v$, we have the local versions $\langle \cdot , \cdot \rangle_{L^2,v}$.

\section{Local Fourier--Jacobi periods}

\label{sec:local_FJ}
In this section, we start by discussing the local theory of Fourier--Jacobi periods $\cP_{\cH_W,v}$. This extends the period $\cP_{\U'_V,v}$ from \S\ref{subsubsec:local_FJ} to the positive corank setting. In that case, the naive integral of matrix coefficients doesn't converge and a more refined construction is needed. We also prove some unfolding results to connect $\cP_{\cH_W,v}$ to its corank-zero analogue $\cP_{\U'_V,v}$. In particular, we prove the tempered intertwining property of Theorem~\ref{thm:temp_int}. 

We henceforth fix $v$ a place of $F$. Because we will always be in the local setting in this section, we erase the place $v$ in this section, so that $E/F$ is either a quadratic extension of local fields (the \emph{inert case}), or $E=F \times F$ (the \emph{split case}). Moreover, if $\bG$ is an algebraic group over $F$, we will simply write $\bG$ for $\bG(F)$.

\subsection{Definition of Fourier--Jacobi periods}
\label{sec:FJ_periods}

\subsubsection{Spaces of tempered functions}
\label{subsubsec:Harish_Chandra_Schwartz}
Let $\bG$ be a connected reductive group over $F$. Let $\varsigma$ be a (class of) logarithmic height function on $\bG$, as in~\cite{BP2}*{Section~1.2}. Note that if $\bG' \leq \bG$ we may take $\varsigma_{| \bG'}$ as the logarithmic height function on $\bG'$, hence the absence of reference to the group in the notation. Denote by $\Xi^\bG$ the Harish-Chandra special spherical function on $\bG$ from \cite{Wald}*{Section~II.1}. Let $\cC^w(\bG)$ be the space of tempered functions on $\bG$ (see~\cite{BP1}*{Section~2.4}). We will not use the precise definition in this paper, but for all $f \in \cC^w(\bG)$ there exists $d>0$ such that 
  \begin{equation}
        \label{eq:HC_Schwartz}
            \valP{f(g)} \ll \Xi^\bG(g) \varsigma(g)^{d}, \quad g \in \bG;
    \end{equation}
The topological vector space $\cC^w(\bG)$ contains $C_c^\infty(\bG)$ as a dense subset.

By \cite{CHH88}, it is known that if $\pi \in \Temp(\bG)$, then the matrix coefficients of $\pi$ belong to $\cC^w(\bG)$.

\subsubsection{Estimates}
\label{subsec:Defi_FJ_periods}

In this section, we gather some estimates on matrix coefficients of the Weil representations and of the Harish-Chandra special spherical function.

\begin{lemma}
\label{lem:estimatesnew}
We have the following assertions.
\begin{enumerate}
    \item There exists $\varepsilon>0$ such that 
\begin{equation*}
   \valP{\langle\omega^\vee_V(g_V) \Phi_1, \Phi_2\rangle_{L^2}} \ll_{\Phi_1,\Phi_2} e^{- \varepsilon \varsigma(g_V)}, \quad g_V \in \U(V), \quad \Phi_1, \Phi_2 \in \omega^\vee_V, 
\end{equation*}
    \item There exists $\varepsilon>0$ such that 
\begin{equation*}
  \valP{ \langle \nu^\vee(g_Wu) \phi_1, \phi_2 \rangle_{L^2}} \ll_{\phi_1,\phi_2} e^{- \varepsilon \varsigma(g_W)}, \quad u \in U_r, \quad g_W \in \U(W), \quad \phi_1, \phi_2 \in \nu^\vee.
\end{equation*}
\item For any non-degenerate skew-Hermitian space $V_0$ and for all $\varepsilon >0$ we have 
\begin{equation*}
    \int_{\U(V_0)} \Xi^{\U(V_0)\times \U(V_0)}(h) e^{-\varepsilon \varsigma(h)} \rd h < \infty.
\end{equation*}
\item For all $\delta>0$ there exists $\varepsilon >0$ such that for every $\phi_1, \phi_2 \in \nu^\vee$ we have
\begin{equation*}
    \int_{\cH_W} \Xi^{\cU_W}(h) e^{\varepsilon \varsigma(h)} \valP{ \langle \nu^\vee(h) \phi_1, \phi_2 \rangle_{L^2}} (1+\valP{\lambda(h)})^{-\delta} \rd h < \infty,
\end{equation*}
where we recall that $\lambda$ was defined in \eqref{eq:lambda_N}.
\end{enumerate}
\end{lemma}

\begin{proof}
    Let us prove the first and second points together. As the Heisenberg representation acts by translation and multiplication by a unitary character by \eqref{eq:Heisenberg_action}
    it is enough to prove the following: if $V_0$ is any non-degenerate skew-Hermitian space of dimension $n_0$ with a polarization $\bV_0 = Y_0 \oplus Y^\vee_0$ then there exists $\varepsilon>0$ such that for all $\Phi_1, \Phi_2 \in \cS(Y_0^\vee)$ we have
\begin{equation}
\label{eq:estimate_goal}
    \valP{ \langle \omega^\vee_{V_0}(g)\Phi_1, \Phi_2(\cdot+y^\vee) \rangle_{L^2}} \ll e^{- \varepsilon \varsigma(g)}, \quad g \in \U(V_0), \quad y^\vee \in Y_0^\vee,
\end{equation}
where $\omega^\vee_{V_0}$ is the Weil representation of $\U(V_0)$ associated to $(\psi^{-1},\mu^{-1})$. Note that if we can prove that \eqref{eq:estimate_goal} holds for one polarization $\bV_0 = Y_0 \oplus Y_0^\vee$, then it holds for all. Indeed, change of Lagrangian induces an isomorphism of $\omega_{V_0}^\vee$ which is also an isomorphism the Heisenberg representation. As the latter is irreducible, the inner products are equal up to multiplication by a non-zero constant. 

Assume first that $E=F\times F$ so that $\U(V_0)=\GL_{n_0}$. Identify $V_0 \cong E^{n_0}=F^{n_0} \times F^{n_0}$ and take $Y^\vee_0=\{0\} \times F^{n_0}$. By the Cartan decomposition it is enough to show that (\ref{eq:estimate_goal}) holds for $\alpha=\mathrm{diag}(\alpha_1, \hdots, \alpha_{n_0})$ with $\valP{\alpha_1} \leq  \valP{\alpha_s} \leq 1 \leq \valP{\alpha_{s+1}} \leq \hdots \leq \valP{\alpha_{n_0}}$. But we have 
\begin{equation*}
\langle\omega^\vee_{V_0}(\alpha)\Phi_1,\Phi_2(\cdot+y^\vee)\rangle_{L^2}=\overline{\mu}(\alpha)\valP{ \det \alpha}^{\frac{1}{2}} \int_{F^{n_0}} \Phi_1(\alpha^* x^\vee) \overline{\Phi_2(y_0^\vee+x^\vee)} \rd x^\vee.
\end{equation*}
Write $(x_i)$ the coordinates of any $x^\vee$ in the canonical basis, and $(y_i)$ the ones of $y^\vee$. Choose $d>0$ such that $\int_{F} \frac{1}{1+\valP{x}^d} \rd x < \infty$ and let $P=\prod_i (1+\valP{x_i}^d)$ seen as a polynomial function on $Y^\vee_0$. Then $\sup_{x^\vee} \valP{\Phi_i(x^\vee)P(x^\vee)} < \infty$ and
\begin{equation*}
    \valP{\langle \omega^\vee_{V_0}(\alpha)\Phi_1,\Phi_2(\cdot+y_0^\vee)\rangle_{L^2}} \ll \prod_{i=1}^s \valP{\alpha_i}^{\frac{1}{2}}\prod_{i=s+1}^{n'} \valP{\alpha_i}^{-\frac{1}{2}} \prod_{i=1}^{s}\int_{F} \frac{1}{1+\valP{x+y_i}^d} \rd x \prod_{i=s+1}^{n'}\int_{F} \frac{1}{1+\valP{x}^d} \rd x.
\end{equation*}
The product of integrals is finite and does not depend on $y_0^\vee$, so that $\varepsilon=\frac{1}{2}$ works.

In the case where $E/F$ is a quadratic extension, let $(z_1, \hdots, z_{r_0}, z_1^*, \hdots, z_{r_0}^*)$ be a maximal hyperbolic family of $V_0$, that is a family such that we have $ q(z_i,z_j)_{V_0}=q(z_i^*, z_j^* )_{V_0}=0$ and $ q(z_i,z_j^*)_{V_0}= \delta_{i,j}$, and set $Z:=\mathrm{span}_E(z_1,\hdots,z_{r_0})$ and $Z^*=\mathrm{span}_E(z_1^*,\hdots,z_{r_0}^*)$. By the Cartan decomposition, it is enough to prove that (\ref{eq:estimate_goal}) holds for $\alpha=\mathrm{diag}(\alpha_1, \hdots, \alpha_{r_0}) \in \GL(Z) \subset \U(V_0)$, with $\valP{\alpha_1} \leq \hdots \leq \valP{\alpha_{r_0}} \leq 1$. Let $Y_1 \oplus Y_1^\vee$ be a polarization of $(Z\oplus Z^*)^\perp$. We take the polarization $(Z \oplus Y_1) \oplus (Z^* \oplus Y_1^\vee)$ of $\bV_0$, and do the same proof as in the split case, using the description of the mixed model given in~\eqref{eq:weil_formula_P2}.

The third point is \cite{Wald}*{Lemma~II.1.5} in the $p$-adic case and \cite{Va}*{Proposition~31} in the Archimedean case.

We finally show the fourth point. The proof is very close to \cite{BP2}*{Lemma~6.5.1~(iii)}, so we only recall the main steps. Since $\varsigma(g_W u ) \ll \varsigma(g_W)+ \varsigma(u)$ for $g_W \in \U(W)$ and $u \in U_r$, by (2) and (3) applied to $V_0=W$ and the equality $\cH_W=U_r \rtimes \U(W)$, it is enough to prove that for all $\delta>0$ and $\varepsilon_0>0$ there exists $\varepsilon>0$ such that 
\begin{equation*}
    I^0_{\varepsilon, \delta}(g_W):= \int_{U_r} \Xi^{\cU_W}(u g_W) e^{\varepsilon \varsigma(u)} (1+\valP{\lambda(u)})^{-\delta} \rd u
\end{equation*}
is absolutely convergent for all $g_W \in \U(W)$ and satisfies
\begin{equation*}
    I^0_{\varepsilon, \delta}(g_W) \ll \Xi^{\U(W) \times \U(W)}(g_W) e^{\varepsilon_0 \varsigma(g_W)}.
\end{equation*}
The proof in \cite{BP2}*{Lemma~6.5.1~(iii)} introduces two intermediate integrals, for $b>0$:
\begin{align*}
       &I^0_{\varepsilon, \delta, \leq b}(g_W):=\int_{U_r}  1_{\varsigma(u) \leq b} \Xi^{\cU_W}(ug_W) e^{\varepsilon \varsigma(u)} (1+\valP{\lambda(u)})^{-\delta}\rd u, \\
     &I^0_{\varepsilon, \delta, >b}(g_W):=\int_{U_r}  1_{\varsigma(u) > b} \Xi^{\cU_W}(u g_W) e^{\varepsilon \varsigma(u)} (1+\valP{\lambda(u)})^{-\delta}\rd u.
\end{align*}
On the one hand it is easy to see as in the proof of loc. cit. that there exists $d>0$ such that 
\begin{equation}
\label{eq:compact_estimate}
    I^0_{\varepsilon, \delta, \leq b}(g_W) \ll e^{\varepsilon b} b^d \Xi^{\U(W) \times \U(W)}(g_W),
\end{equation}
for all $\varepsilon >0$, $b>0$ and $g_W \in \U(W)$. But on the other hand, there exists $\alpha >0$ such that $\Xi^{\cU_W}(g_1 g_2) \ll e^{\alpha \varsigma(g_2)} \Xi^{\cU_W}(g_1)$ for all $g_1, g_2 \in \cU_W$. Therefore
\begin{equation}
\label{eq:non_compact_estimate}
     I^0_{\varepsilon, \delta, > b}(g_W) \ll e^{\alpha \varsigma(g_W) - \sqrt{\varepsilon}b} \int_{U_r} \Xi^{\cU_W}(u) e^{(\varepsilon + \sqrt{\varepsilon}) \varsigma(u)} (1+\valP{\lambda(u)})^{-\delta}\rd u,
\end{equation}
for all $g_W \in \U(W)$ and $b>0$. That the last integral converges for $\varepsilon$ small is a consequence of \cite{BP2}*{Lemma~B.3.2} as $\lambda$ is the restriction of a non-degenerate additive character of $N_{\mathrm{min}}$ the unipotent radical of a minimal parabolic of $\U(V)$ which stabilizes a maximal flag obtained by completing $\mathrm{span}_E(x_1) \subset \mathrm{span}_E(x_1,x_2) \subset \hdots \subset X$ with lines in $W$. The result now follows for a small $\varepsilon$ by the same trick as in \cite{BP2}*{Lemma~6.5.1~(iii)}.
\end{proof}

\begin{prop}
    \label{prop:extension_continuity}
    We have the following assertions.
    \begin{enumerate}
    \item For every $\phi_1$ and $\phi_2 \in \nu^\vee$, the linear form
    \begin{equation*}
    f \in C_c^\infty(\cU_W) \mapsto \int_{\cH_W} f(h) \langle\nu^\vee(h) \phi_1, \phi_2\rangle_{L^2} \rd h
    \end{equation*}
    extends by continuity to $\cC^w(\cU_W)$.
    \item The linear form
    \begin{equation*}
        f \in C_c^\infty(G_r) \mapsto \int_{N_r} f(n_r) \psi_r(n_r) \rd n_r
    \end{equation*}
    extends by continuity to $\cC^w(G_r)$.
    \end{enumerate}
    \end{prop}

    \begin{proof}
        For the first point, we fix the one-parameter subgroup parametrized by $a : t \in \bG_m \mapsto \mathrm{diag}(t^{r-1},t^{r-2}, \hdots, 1) \in \GL(X) \subset \U(V)$. By the formulae \eqref{eq:Heisenberg_action} and \eqref{eq:weil_formula_P2} we have for all $t \in F^\times$ and $h \in \cH_W$
        \begin{equation*}
            \langle \nu^\vee(a(t)ha(t)^{-1}) \phi_1, \phi_2 \rangle_{L^2}=\overline{\psi_U}(t \lambda(h)) \langle \nu^\vee(h) \phi_1, \phi_2 \rangle_{L^2}.
        \end{equation*}
        The proof is now exactly the same as in \cite{BP2}*{Proposition~7.1.1}, using Lemma~\ref{lem:estimatesnew} (4) instead of (iii) of \cite{BP2}*{Lemma~6.5.1}.

        The second point is proved as in \cite{BP2}*{Proposition~7.1.1} thanks to \cite{BP2}*{Lemma~B.3.2}.
    \end{proof}

    \subsubsection{Local Fourier--Jacobi periods}
    \label{subsec:local_FJ_periods_defi}

    Let us denote by
    \begin{equation*}
        \cP_{\cH_W} : f \otimes \phi_1 \otimes \phi_2 \in \cC^w(\cU_W) \otimes \nu^\vee \otimes \nu^\vee \to \C, \quad \cP_{N_r} : f \in \cC^w(G_r) \to \C 
    \end{equation*}
    the linear forms obtained by extending the integrals over $\cH_W$ and $N_r$, which exist by Proposition~\ref{prop:extension_continuity}. The first is the \emph{Fourier--Jacobi period}, while the second is the \emph{Whittaker period}. 

For $\phi_1, \phi_2 \in \nu^\vee_L$, set
\begin{equation*}
    \mathcal{P}_{\cH_L}(f \otimes \phi_1 \otimes \phi_2) := \int_{\cH_L} f(h) \langle\nu^\vee_L(h) \phi_1, \phi_2\rangle_{L^2} \rd h, \quad f \in C_c^{\infty}(L),
\end{equation*}
which we extend by continuity to $\cC^w(L)$. Note that $\mathcal{P}_{\cH_L}=\mathcal{P}_{\cH_W} \otimes \mathcal{P}_{N_r}$.

For $(\mathbb{G},\mathbb{H}) \in \{ (\U_V, \U_V'), (\cU_W,\cH_W),(L,\cH_L) \}$, if $\sigma$ is a smooth irreducible tempered representation of $\mathbb{G}$ equipped with an invariant inner product $\langle\cdot,\cdot\rangle$, then for every $\varphi_1, \varphi_2 \in \sigma$ the map $c_{\varphi_1,\varphi_2} : g \mapsto \langle \sigma(g) \varphi_1, \varphi_2\rangle$ belongs to $\mathcal{C}^w(\mathbb{G})$ by \cite{CHH88}. We set for $\Phi_i \in \omega^\vee_V$, $\nu^\vee$ or $\nu^\vee_L$
\begin{equation}
\label{eq:integral_of_coefficients}
    \mathcal{P}_{\mathbb{H}}(\varphi_1 \otimes \Phi_1, \varphi_2 \otimes \Phi_2):= \mathcal{P}_{\mathbb{H}}(c_{\varphi_1,\varphi_2} \otimes \Phi_1 \otimes \Phi_2).
\end{equation}
If $\varphi_1=\varphi_2$ and $\Phi_1=\Phi_2$, we will simply write $\mathcal{P}_{\mathbb{H}}(\varphi, \Phi)$ for $\mathcal{P}_{\mathbb{H}}(\varphi_1 \otimes \Phi_1, \varphi_2 \otimes \Phi_2)$. We also define, for $\tau$ a tempered representation of $G_r$ and $\varphi \in \tau$, the Whittaker period $\cP_{N_r}(\varphi):=\cP_{N_r}(c_{\varphi,\varphi})$.

\subsection{Unfolding equalities}
\label{sec:unfold_eq}

\subsubsection{Representations}
\label{subsubsec:reps}
Let $\sigma_V$, $\sigma_W$ and $\tau$ be smooth irreducible representations of $\U(V)$, $\U(W)$ and $G_r$ respectively, equipped with invariant inner products $\langle \cdot, \cdot \rangle$. Define $\sigma=\sigma_V \boxtimes \sigma_W$, an irreducible representation of $\cU_W$. For $s \in \C$, set $\tau_s=\tau \valP{\det}^s$ and $\Sigma_s=I_{P(X)}^{\U(V)} \sigma_W \boxtimes \tau_s$ the normalized parabolic induction, equipped with its canonical inner product $\langle \cdot, \cdot \rangle$ given by integration along $P(X) \backslash \U(V)$. These representations are fixed for the rest of this section. 

\subsubsection{Unfolding of functionals}
Let
\begin{equation*}
    \cL_{\cH_L} \in \Hom_{\cH_L}(( \sigma \boxtimes \tau) \otimes \nu^\vee_L, \C).
\end{equation*}
By multiplicity one results of~\cite{Sun}, \cite{GGP} and \cite{LiuSun}, we know that $\cL_{\cH_L}$ factors as $\cL_{\cH_L}=\cL_{N_r} \otimes \cL_{\cH_W}$ with $\cL_{N_r} \in \Hom_{N_r}(\tau , \overline{\psi_r})$ and $\cL_{\cH_W} \in \Hom_{\cH_W}(\sigma \otimes \nu^\vee,\C)$.

\begin{prop}
\label{prop:corank_reduction}
    There exists $c \in \R$ such that for $\Re(s)>c$ the functional 
    \begin{equation}
    \label{eq:L_U_defi}
        \cL_{\U_V',s}:  \varphi_V \otimes \varphi_s \otimes \Phi \in  \sigma_V \otimes \Sigma_s\otimes \omega^\vee_V \mapsto \int_{\cH_W' \backslash \U_V'} \cL_{\cH_L}(\sigma_V(h) \varphi_V \otimes \varphi_s(h) , (\omega^\vee_V(h)\Phi)_{Y^\vee}) \rd h
    \end{equation}
    converges absolutely. If all representations are tempered, we may take $c=-\frac{1}{2}$. Furthermore, for $\Re(s)>c$ the following assertions are equivalent.
\begin{enumerate}
    \item There exist $\varphi_V, \varphi_W, \varphi_{\tau}$ and $\phi$ such that $\cL_{\cH_L}(\varphi_V \otimes \varphi_W \otimes \varphi_\tau ,  \phi) \neq 0$.
    \item There exist $\varphi_V, \varphi_s$ and $\Phi$ such that $\cL_{\U_V',s}(\varphi_V \otimes \varphi_s , \Phi) \neq 0$.
\end{enumerate}
\end{prop}

\begin{proof}
    The convergence can be proved as in~\cite{BPC22}*{Proposition~8.6.1.1} by the Iwasawa decomposition $\U_V'=T_r \cH_W' K_V$, where $K_V$ is a maximal compact subgroup, using \eqref{eq:weil_formula_P2}, and Lemma~\ref{lem:estimatesnew} (1) and (2). The implication $(2)\implies(1)$ is automatic. For $(1)\implies(2)$, we see that for $\Re(s)>c$ we have by \eqref{eq:weil_formula_P2}, 
    \begin{align*}
         &\cL_{\U_V',s}(\varphi_V \otimes \varphi_s , \Phi)= \\
         &\int_{P(X) \backslash \U(V)} \int_{N_r \backslash G_r}  \cL_{\cH_L}\left((\sigma_V(gh)\otimes \tau(g))(\varphi_V \otimes \varphi_{s}(h)) , (\omega^\vee_V(h)\Phi)(g^* x_r^*)\right) (\valP{\cdot}^{s+\frac{1}{2}} \overline{\mu} \delta^{-\frac{1}{2}}_{P(X)})(g) \rd g \rd h.
    \end{align*}
    As $\Sigma_s$ is stable by multiplication by $C^\infty(P(X) \backslash \U(V))$, it is enough to show that there exists $\varphi_V \in \sigma_V$, $\varphi_W \in \sigma_W$, $\varphi_\tau \in \tau$, $\Phi \in \cS(X^*)$ and $\phi \in \nu^\vee$ such that
    \begin{equation*}
        \int_{N_r \backslash G_r} \Phi(g^* x_r^*)  \cL_{\cH_L}\left(\sigma_V(g) \varphi_V \otimes \varphi_W  \otimes \tau(g) \varphi_\tau ,  \phi \right) (\valP{\cdot}^{s+\frac{1}{2}} \overline{\mu} \delta_{P(X)}^{-\frac{1}{2}})(g)\rd g \neq 0.
    \end{equation*}
    Recall that $P_r$ is the mirabolic subgroup of $G_r$. This condition can be rewritten as
    \begin{equation*}
         \int_{P_r \backslash G_r} \Phi(g^* x_r^*) \int_{N_r \backslash P_r} \cL_{\cH_L}\left(\sigma_V(pg) \varphi_V  \otimes \varphi_W \otimes\tau(pg) \varphi_\tau , \phi \right) (\valP{\cdot}^{s+\frac{1}{2}} \overline{\mu} \delta_{P(X)}^{-\frac{1}{2}})(pg)\valP{p}^{-1} \rd p \rd g \neq 0.
    \end{equation*}
    The map $g \in P_r \backslash G_r \mapsto  g^* x_r^* \in X^*$ induces an embedding $\cS(P_r \backslash G_r) \hookrightarrow \cS(X^*)$, so it is enough to prove that we have elements such that 
    \begin{equation*}
         \int_{N_{r-1} \backslash G_{r-1}} \cL_{\cH_W}(\varphi_W \otimes \sigma_V(g) \varphi_V , \phi) \cL_{N_r}(\tau(g) \varphi_\tau) (\valP{\cdot}^{s-\frac{1}{2}} \overline{\mu} \delta_{P(X)}^{-\frac{1}{2}})(g) \rd g \neq 0.
    \end{equation*}
    By~\cite{GK75}*{Theorem~6} and~\cite{Kem15}*{Theorem~1}, for every $f \in C_c^\infty(N_{r-1} \backslash G_{r-1},\overline{\psi_r})$ there exists $\varphi_\tau \in \tau$ such that $\cL_{N_r}(\tau(g) \varphi_\tau)=f(g)$ for every $g \in N_{r-1} \backslash G_{r-1}$. The claim now follows from the non-vanishing of $\cL_{\cH_W}$.
\end{proof}

\subsubsection{Unfolding of tempered periods}

The following is \cite{Boi}*{Proposition~6.3.1}.
\begin{prop}
\label{prop:tempered_computations}
    Assume that $\sigma_V$, $\sigma_W$ and $\tau$ are all tempered. Then for $\varphi^1_V, \varphi^2_V \in \sigma_V$,  $\varphi^1_\Sigma, \varphi^2_\Sigma \in \Sigma$,  and $\Phi_1, \Phi_2 \in \omega^\vee_V$ we have
    \begin{equation}
        \mathcal{P}_{\U_V'}(\varphi^1_V \otimes \varphi^1_\Sigma \otimes \Phi_1, \varphi^2_V \otimes \varphi^2_\Sigma \otimes \Phi_2) \nonumber 
        =\int_{(\cH_W' \backslash \U_V')^2} \frac{\mathcal{P}_{\cH_L}\left(
        (\sigma_V(h_i) \varphi^i_V \otimes \varphi^i_\Sigma(h_i) \otimes (\omega^\vee_V(h_i) \Phi_i)_{Y^\vee})_{i=1,2}\right)}{\upsilon(G_r)} \rd h_i. \label{eq:to_prove_temp}
    \end{equation}
\end{prop}

\subsubsection{Non vanishing of Fourier--Jacobi periods}

We now deduce from our unfolding identities the following proposition which is an explicit version of the Gan--Gross--Prasad conjecture for tempered representation. It is an extension of \cite{Xue2}*{Proposition~1.1.1} which worked for $r=0$.

\begin{theorem}
    \label{thm:explicit_GGP_FJ}
    Let $\sigma_V$ and $\sigma_W$ be smooth irreducible tempered representations of $\U(V)$ and $\U(W)$ respectively. For every $\varphi_V \in \sigma_V$, $\varphi_W \in \sigma_W$ and $\phi \in \nu^\vee$ we have 
    \begin{equation*}
        \cP_{\cH_W}(\varphi_V \otimes \varphi_W , \phi) \geq 0.
    \end{equation*}
    Moreover, if $\Hom_{\cH_W}(\sigma_V \otimes \sigma_W \otimes \nu^\vee,\C) \neq \{0\}$, then $\cP_{\cH_W}$ is not identically zero.
\end{theorem}

\begin{proof}
    For all $f \in C_c^\infty(\cU_W)$ satisfying $f(g)=\overline{f(g^{-1})}$, we have $\cP_{\cH_W}(f\otimes \phi \otimes \phi) \in \R$. Therefore it follows by continuity that $\cP_{\cH_W}(\varphi_V \otimes \varphi_W , \phi)$ is real. By multiplicity one for Whittaker functionals (\cite{Ro}) and by \cite{SV}*{Theorem~6.3.4}, for any smooth irreducible tempered representation $\tau$ of $G_r$ (hence generic by \cite{Zel}*{Theorem~9.7}) we have $\cP_{N_r}= \cL_{N_r} \otimes \overline{\cL_{N_r}}$ for some $\cL_{N_r} \in \Hom_{N_r}(\tau,\overline{\psi_r})$ non-zero. Using Proposition \ref{prop:tempered_computations} and following the same steps as in the proof of Proposition~\ref{prop:corank_reduction} in reverse order, one shows that the existence of $\varphi_V, \varphi_W$ and $\phi$ such that $\cP_{\cH_W}(\varphi_V \otimes \varphi_W , \phi)<0$ implies the existence of $\varphi_\Sigma \in \Sigma:=I_{P(X)}^{\U(V)} \sigma_W \boxtimes \tau$ and $\Phi \in \omega^\vee_V$ satisfying
    \begin{equation*}
        \cP_{\U_V'}(\varphi_V \otimes \varphi_\Sigma , \Phi)<0.
    \end{equation*}
    This leads to a contradiction by \cite{Xue2}*{Proposition~1.1.1~(2)}, and the first point follows. Note that \cite{Xue2} uses the local Theta correspondence to reduce the positivity of $\cP_{\U_V'}$ to that of the Bessel period which is known to satisfy this property by \cite{SV}*{Theorem~6.2.1}. 

    For the second point, assume that $\cP_{\cH_W}=0$. By Proposition \ref{prop:tempered_computations}, this implies that $\cP_{\U_V'}=0$, which in turn implies that $\Hom_{\U_V'}(\sigma_V \otimes \Sigma \otimes \omega^\vee_V,\C)=\{0\}$ by \cite{Xue2}*{Proposition~1.1.1~(2)}. It now follows from Proposition~\ref{prop:corank_reduction} that $\Hom_{\cH_W}(\sigma_V \otimes \sigma_W \otimes \nu^\vee,\C) = \{0\}$.
\end{proof}

\subsection{Unramified unfolding}
\label{subsec:unram_unfold}

We now assume that the situation is unramified in the sense of \cite{Boi}. This will happen as soon as the place $v$ which we have picked at the beginning of \S\ref{sec:local_FJ} lies outside of the a finite $\tS$ of places of the global field $F$. We have a maximal compact subgroup $K_V$ as in \S\ref{subsubsec:max_compact}.

We still drop the place $v$ from the notation. Let $\sigma_{V}$, $\sigma_{W}$ and $\tau$ be smooth unramified irreducible representations of $\U(V)$, $\U(W)$ and $G_{r}$ respectively. Set $\Sigma:=I_{P(X)}^{\U(V)} \sigma_{W} \boxtimes \tau$ and $\Sigma_{s} = I_{P(X)}^{\U(V)} \sigma_{W} \boxtimes \tau_{s}$. Because of our unramified assumption, the space $(\nu^\vee)^{K_{V} \cap \cH_W'}$ has dimension one. 

\begin{theorem}
\label{thm:L_function_equality}
Let $\varphi_{V}^\circ \in \sigma_{V}^{K_{V}}$, $\varphi_{s}^\circ \in \Sigma_{s}^{K_{V}}$ and $\phi^\circ \in (\nu^\vee)^{K_{V} \cap \cH_W}$. Set $\Phi^\circ := 1_{X^*(\cO_{E})} \otimes \phi^\circ \in \omega^\vee_{V}$. For $\Re(s)$ sufficiently large we have
\begin{equation}
    \cL_{\U_V',s}(\varphi_{V}^\circ \otimes \varphi_{s}^\circ , \Phi^\circ)=  \frac{L(\frac{1}{2}+s,\tau \times \sigma_{V} \otimes \overline{\mu})}{L(1+s,\tau^{\mathsf{c}} \times \sigma_{W})L(1+2s,\tau,\mathrm{As}^{(-1)^{m}})} \cL_{\cH_L}(\varphi_{V}^\circ \otimes\varphi_s^\circ(1) ,  \phi^\circ). \label{eq:un_unfolding}
\end{equation}
where $\tau^{\mathsf{c}}=\tau \circ \mathsf{c}$ and $m$ is the dimension of $W$. Moreover, if $\sigma_{V}$, $\sigma_{W}$ and $\tau$ are tempered, \eqref{eq:un_unfolding} holds for $\Re(s)>-\frac{1}{2}$.
\end{theorem}

\begin{proof}
    By \cite{Boi}*{Proposition~7.3.1.1}, we have the unramified unfolding equality for $\Re(s)$ large enough
    \begin{align}
\label{eq:L_unfold}
       &\sum_{\lambda_r \in \Lambda_{r}} \cL_{\cH_L}\left(\sigma_{V}(\lambda_r)\varphi^\circ_{V}  \otimes \varphi_{s}^\circ(\lambda_r) , (\omega^\vee_{V}(\lambda_r)\Phi^\circ)_{Y^\vee} \right)(\delta_{P(X)}^{\frac{1}{2}} \delta_{P}^{-1})(\lambda_r) \nonumber  \\
       =  &\frac{L(\frac{1}{2}+s,\tau \times \sigma_{V} \otimes \overline{\mu})}{L(1+s,\tau^{\mathsf{c}} \times \sigma_{W})L(1+2s,\tau,\mathrm{As}^{(-1)^{m}})}\cL_{\cH_L}(\varphi_{V}^\circ \otimes \varphi_{s}^\circ(1) ,\phi^\circ),
\end{align}
where $P$ is the parabolic subgroup of $\U(V)$ stabilizing the flag $0 \subset E x_1 \subset \hdots \subset X$ and we have set $\Lambda_{r}=T_r(F) / T_r(\cO_F)$. Note that with our choices of measures, the volumes $\mathrm{vol}(K_{V})$ and $\mathrm{vol}(K_{V} \cap \cH_W)$ are $1$. The result now follows from the definition of $\cL_{\U_V',s}$ given in \eqref{eq:L_U_defi} and the Iwasawa decomposition $\U(V)=T_r \cH'_W K_{V}$.
\end{proof}

\section{Global Fourier--Jacobi periods}

\label{sec:global_FJ}

We now go back to the global setting and prove an unfolding identity.

\subsection{Definition of global periods}
\label{subsubsection:global_FJ_notations}
For $(\bH,\omega^\vee,Z^\vee) \in \{ (\U'_V,\omega^\vee_V,X^* \oplus Y^\vee), (\cH_W,\nu^\vee,Y^\vee),(\cH_L,\nu^\vee_L,Y^\vee)\}$, write the theta series
\begin{equation*}
    \theta^\vee_\bH(h,\Phi)=\sum_{z \in Z^\vee(F)} (\omega^\vee(h)\Phi)(z), \quad h \in \bH(\bA),  \quad \Phi \in \omega^\vee.
\end{equation*}
In the first case, this is the series $\theta^\vee$ from \eqref{eq:Theta_unitary_definition}. We will also make use of
\begin{equation*}
    \theta_W^{V,\vee}(h,\Phi):= \sum_{y^\vee \in  Y^\vee(F)} (\omega^\vee_{V}(h) \Phi)_{Y^\vee}(y^\vee), \quad h \in \U'_V(\bA), \quad \Phi \in \omega^\vee_V.
\end{equation*}
These theta series are of moderate growth by variants of Lemma~\ref{lem:theta_moderate_u}.

By the unfolding identity \eqref{eq:nu_relation}, we have for $u \in U_{r-1}(\bA), \; n \in N(X)(\bA), \; \gamma \in P_r(F)$ (the mirabolic group of $G_r$), $g_W \in \U(W)(\bA)$ and $g \in \U(V)(\bA)$ the formula 
\begin{equation}
\label{eq:theta_relation}
    \overline{\psi}_{U}(u)\theta_W^{V,\vee}(n \gamma g_Wg,\Phi)=\theta_{\cH_W}^{\vee}(un,(\omega^\vee_V(\gamma g_Wg)\Phi)_{Y^\vee}).
\end{equation}
For triples ${(\bG,\bH,\omega^\vee) \in \{(\U_V,\U'_V,\omega^\vee_V), (\cU_W,\cH_W,\nu^\vee),(L,\cH_L,\nu^\vee_L)\}}$, consider the integral
\begin{equation}
    \label{eq:global_period_defi}
    \cP_{\bH}(\varphi,\Phi):=\int_{[\bH]} \varphi(h) \theta_{\bH}^\vee(h,\Phi) \rd h, \quad \varphi \in \cA(\bG), \quad \Phi \in \omega^\vee,
\end{equation}
at least if the integral converges. Finally, for $\varphi \in \cA(G_r)$ set 
\begin{equation*}
        \cP_{N_r}(\varphi):= \int_{[N_r]} \varphi(h) \psi_r(n) \rd n.
\end{equation*}
Then for $\varphi=\varphi_{V,W} \otimes \varphi_{r} \in \cA(L)$ we have $ \cP_{\cH_L}(\varphi , \phi)=\cP_{N_r}(\varphi_r)\cP_{\cH_W}(\varphi_{V,W}, \phi)$.

\subsection{Unfolding of global periods}

Let $\sigma_V \in \Pi_{\cusp}(\U(V))$ and $\sigma_W \in \Pi_{\cusp}(\U(W))$. Let $\tau$ be an Arthur parameter of $G_r(\bA)$. For $s \in \C$ we set $\Sigma_{s}:=I_{P(X)}^{\U(V)} (\tau_s \boxtimes \sigma_W)$. Write tensor product decompositions $\sigma_V=\otimes'_v \sigma_{V,v}$, $\sigma_W=\otimes'_v \sigma_{W,v}$ and $\tau=\otimes'_v \tau_{v}$.

We prove the global counterpart of Proposition~\ref{prop:tempered_computations}.

\begin{prop}
\label{prop:global_corank_reduction}
Let $\varphi_V \in \cA_{\sigma_V}(\U(V))$, $\varphi \in \cA_{P(X),\tau \boxtimes \sigma_W}(\U(V))$ and $\Phi \in \omega^\vee_V$. There exists $c>0$ such that for all $s$ with $\Re(s)>c$ we have the equality
\begin{equation*}
    \cP_{\U'_V}\left(\varphi_V \otimes E_{P(X)}^{\U(V)}(\varphi,s) , \Phi \right) = \int_{\cH_W'(\bA) \backslash \U'_V(\bA)} \cP_{\cH_L} \left( \mathrm{R}(h)(\varphi_V \otimes \varphi_s) , (\omega^\vee(h) \Phi)_{Y^\vee} \right) \rd h,
\end{equation*}
where the right-hand side is absolutely convergent.
\end{prop}

\begin{remark}
    \label{rem:same_period}
    In \S\ref{subsec:reg_FJ}, we have defined a regularized period $\cP_{\U'_V}$ defined on Eisenstein series induced from automorphic cuspidal representations of Levi subgroups of $\U_V$ representing a $(\U_V,\U'_V,\overline{\mu})$-regular cuspidal data. Because $\sigma_V$ is cuspidal, the construction applies here. Moreover, by Remark~\ref{rem:cuspi_case} the period in this case reduces to the absolutely convergent integral along $[\U'_V]$.
\end{remark}

\begin{proof}
For $\Re(s)$ large enough, the Eisenstein series is also absolutely convergent and we have
\begin{align}
\label{eq:global_period}
     \cP_{\U'_V}\left(\varphi_V \otimes E_{P(X)}^{\U(V)}(\varphi,s) , \Phi \right)
    &=\int_{[\U(V)]}\sum_{ \gamma \in P(X)(F) \backslash \U(V)(F)} \varphi_s( \gamma g) \varphi_V(g) \theta_{\U'_V}^\vee(g,\Phi) \rd g \nonumber \\
    &=\int_{M(X)(F)N(X)(\bA) \backslash \U(V)(\bA)} \varphi_s(g)  \int_{[N(X)]}\varphi_V(ng) \theta_{\U'_V}^\vee(ng,\Phi) \rd n \rd g.
\end{align}
For a fixed $g$ the inner integral over $[N(X)]$ is given by
\begin{equation*}
    \int_{[N(X)]}\varphi_V(ng) \sum_{x^* \in X^*(F)} \sum_{y^\vee \in Y^\vee(F)}\omega^\vee_V(ng) \Phi(x^*,y^\vee) \rd n.
\end{equation*}
Note that this triple integral is absolutely convergent. It follows from the description of the mixed model given in~\eqref{eq:weil_formula_P2} and~\eqref{eq:weil_formula_P3} that $N(X)$ only acts on the $y^\vee$ coordinate, so that for every $x^*$ the map $n \mapsto  \sum_{y^\vee \in Y^\vee(F)}\omega^\vee_V(ng) \Phi(x^*,y^\vee)$ is $N(X)(F)$ left-invariant, and that $n \mapsto  \sum_{y^\vee \in Y^\vee(F)}\omega^\vee_V(ng) \Phi(0,y^\vee)$ is constant on $[N(X)]$. Then by cuspidality of $\varphi_V$ and~\eqref{eq:weil_formula_P2} we have
\begin{align*}
    \int_{[N(X)]}\varphi_V(ng) \theta_{\U'_V}^\vee(ng,\Phi) \rd n&
    =\sum_{x^* \in X^* \setminus \{0\}} \int_{[N(X)]}\varphi_V(ng) \sum_{y^\vee \in Y^\vee(F)}\omega^\vee_V(ng) \Phi(x^*,y^\vee) \rd n \\
    &=\sum_{\gamma \in P_r(F) \backslash G_r(F)} \int_{[N(X)]}\varphi_V(n\gamma g) \theta_W^{V,\vee}(n \gamma g, \Phi) \rd n,
\end{align*}
where the last equality is obtained thanks to the change of variables $n \mapsto \gamma n \gamma^{-1}$. Now fix $\gamma$. Set $X_{r-1}=\mathrm{span}_E(x_1,\hdots,x_{r-1}) \subset X$. Let $N(X_{r-1})$ be the unipotent radical of the parabolic of $\U(V)$ stabilizing $X_{r-1}$. We have $N(X)=S(W) \ltimes (N(X) \cap N(X_{r-1}))$. Then
\begin{equation}
\label{eq:global_Fourier}
    \int_{[N(X)]}\varphi_V(n\gamma g) \theta_W^{V,\vee}(n \gamma g, \Phi) \rd n  
    =\int_{[S(W)]}\int_{[N(X) \cap N(X_{r-1})]} \varphi_V(nh \gamma g)
    \theta_W^{V,\vee}(n h\gamma g, \Phi) \rd n \rd h. 
\end{equation}
For $h$ fixed, consider $p \in [P_r] \mapsto \int_{[N(X) \cap N(X_{r-1})]} \varphi_V( nph \gamma g) \rd n$. This is well defined as $P_r(F)$ normalizes $(N(X) \cap N(X_{r-1}))(\bA)$. But $N_r \ltimes (N(X) \cap N(X_{r-1}))=U_{r-1}$  and $S(W)$ commutes with $P_r$ so that by cuspidality of $\varphi_V$ we have the Fourier expansion along $[N_r]$
\begin{equation}
\label{eq:cuspidal_fourier}
\int_{[N(X) \cap N(X_{r-1})]} \varphi_V( nh \gamma g) \rd n=\sum_{\gamma' \in N_r(F) \backslash P_r(F)} \int_{[U_{r-1}]} \varphi_V( u  h \gamma'\gamma g) \overline{\psi_{U}(u)} \rd u.
\end{equation}
As $ U_{r-1} \rtimes S(W) =U_r$, using \eqref{eq:theta_relation} and \eqref{eq:cuspidal_fourier} we see that \eqref{eq:global_Fourier} reduces to
\begin{equation*}
    \int_{[N(X)]}\varphi_V(n\gamma g) \theta_W^{V,\vee}(n \gamma g, \Phi) \rd n =\sum_{\gamma' \in N_r(F) \backslash P_r(F)}  \int_{[U_r]} \varphi_V(u \gamma' \gamma g)\theta_{\cH_W}^{\vee}(u,(\omega^\vee_V(\gamma' \gamma g)\Phi)_{Y^\vee}) \rd u.
\end{equation*}
Hence by going back to \eqref{eq:global_period} we obtain
\begin{align}
    \label{eq:triple_integral}
    &\cP_{\U'_V}\left(E_{P(X)}^{\U(V)}(\varphi,s) \otimes \varphi_V , \Phi \right) \nonumber \\
    =&\int_{M(X)(F)N(X)(\bA) \backslash \U(V)(\bA)}  \sum_{\gamma \in N_r(F) \backslash G_r(F)} \varphi_s(\gamma g)  \int_{[U_r]}\varphi_V(u\gamma g) \theta_{\cH_W}^{\vee}(u,(\omega^\vee_V( \gamma g)\Phi)_{Y^\vee}) \rd u \rd g.
\end{align}
Assume for the moment that the double integral $\int_{M(X)(F)N(X)(\bA) \backslash \U(V)(\bA)}  \sum_{\gamma \in N_r(F) \backslash G_r(F)}$ is absolutely convergent. Then we see that 
\begin{align*}
    & \cP_{\U'_V}\left(\varphi_V \otimes E_{P(X)}^{\U(V)}(\varphi,s) , \Phi \right) \\
    =& \int_{\cH_W'(\bA) \backslash \U_V' (\bA)} \int_{[N_r]} \int_{[\U(W)]}   \int_{[U_r]} \varphi_s(n_r g_W g)  \varphi_V(u n_r g_W g)\theta_{\cH_W}^{\vee}(u,(\omega^\vee_V(n_r g_W g)\Phi)_{Y^\vee})\rd u \rd g_W \rd n_r \rd g \\
    =& \int_{\cH_W'(\bA) \backslash \U_V'(\bA)} \int_{[N_r]} \int_{[\U(W)]}   \int_{[U_r]}\varphi_s(n_r g_W g) 
 \varphi_V(ug_W g)  \psi_r(n_r) \theta_{\cH_W}^{\vee}(ug_W,(\omega^\vee_V(g)\Phi)_{Y^\vee}) \rd u \rd g_W \rd n_r \rd g \\
  =& \int_{\cH_W'(\bA) \backslash \U'_V(\bA)} \int_{[\cH_L]}  \varphi_s(hg) 
 \varphi_V(hg)  \theta_{\cH_L}^{\vee}(h,(\omega^\vee_V(g)\Phi)_{Y^\vee}) \rd h \rd g \\
 =&\int_{\cH_W'(\bA) \backslash \U'_V(\bA)} \cP_{\cH_L} \left( \mathrm{R}(g)(\varphi_V \otimes \varphi_s) , (\omega^\vee(g) \Phi)_{Y^\vee} \right) \rd g,
\end{align*}
where we have used a change of variable $u \mapsto u n_r$ and \eqref{eq:theta_relation} to go from the second to the third line.

It remains to show that \eqref{eq:triple_integral} is absolutely convergent. This can be shown by similar techniques as the one used in \cite{BPC22}*{Proposition~8.5.1.1} and Lemma~\ref{lem:convTau}, using the formulae for the Weil representation from \eqref{eq:weil_formula_P1}. This completes the proof of the proposition.
\end{proof}

\begin{coro}
\label{cor:non_vanishing_of_periods}
    The following are equivalent.
    \begin{enumerate}
        \item There exist $\varphi_V \in \sigma_V$, $\varphi_W \in \sigma_W$  and $\phi \in \nu^\vee$ such that $\cP_{\cH_W} \left( \varphi_V \otimes \varphi_W , \phi \right) \neq 0$.
        \item There exist $\varphi_V \in \sigma_V$, $\varphi_\Sigma \in \Sigma$, $\Phi \in \omega^\vee_{V}$ and $s \in \C$ such that $E_{P(X)}^{\U(V)}(\varphi,.)$ has no pole at $s$ and $\cP_{\U'_V}\left( \varphi_V \otimes E_{P(X)}^{\U(V)}(\varphi_\Sigma,s) , \Phi \right) \neq 0$.
    \end{enumerate}
\end{coro}

\begin{proof}
    Let $\varphi_V \in \sigma_V$, $\varphi_\Sigma \in \Sigma$ and $\Phi \in \cS(X^*(\bA_E) \oplus Y^\vee(\bA))$. By \cite{Boi}*{Proposition~7.3.1.1} and Proposition~\ref{prop:global_corank_reduction}, up to enlarging $\tS$ we have 
\begin{align}
\label{eq:L-function_equality}
        \cP_{\U'_V}\left(\varphi_V \otimes E_{P(X)}^{\U(V)}(\varphi_\Sigma,s)  , \Phi \right) =  &\frac{L^{\tS}(\frac{1}{2}+s, \sigma_V \times \tau \otimes \overline{\mu})}{L^{\tS}(1+s,\tau^{\sfc} \times \sigma_W)L^{\tS}(1+2s,\tau,\mathrm{As}^{(-1)^{m}})} \\
        &\times \int_{\cH_W'(F_\tS) \backslash \U'_V(F_\tS)} \cP_{\cH_L} \left( \mathrm{R}(h)(\varphi_V \otimes \varphi_{\Sigma,s} ) , (\omega^\vee_{V}(h) \Phi)_{Y^\vee} \right) \rd h, \nonumber
\end{align}
provided that $\Re(s) \gg 0$. It follows from the hypothesis on $\tau$ that the Whittaker period $\cP_{N_r}$ is non-zero on $\tau$ as it is generic. By the factorization $\cP_{\cH_L}= \cP_{\cH_W} \otimes \cP_{N_r}$, (1) is equivalent to the non vanishing of $\cP_{\cH_L}$ on $\sigma_V  \otimes \tau \otimes \sigma_W \otimes  \nu^\vee$, which is itself equivalent to (2) by \eqref{eq:L-function_equality} and Proposition~\ref{prop:corank_reduction}.
\end{proof}

\section{Proof of the main theorems}

In this chapter, we prove our main results Theorem~\ref{thm:GGP} and \ref{thm:II}.

\label{sec:proofs_positive}
\subsection{Gan--Gross--Prasad conjecture in positive corank}
\label{sec:GGP_positive_corank}

We start with the GGP conjecture. The statement we prove is the following.

\begin{theorem} \label{thm:GGP_arbitrary_corank}
Let $\Pi$ be a discrete Hermitian Arthur parameter of $G_n \times G_m$. The
following assertions are equivalent:
    \begin{enumerate}
        \item $L(\frac{1}{2}, \Pi \otimes \overline{\mu}) \neq 0$;
        \item there exist non-degenerate skew-Hermitian spaces $W\subset V$ and $\sigma$ a cuspidal automorphic representation of $\U(V)(\bA) \times \U(W)(\bA)$ such that $W^\perp$ is split, the weak base change of $\sigma$ to
            $G_n \times G_m$ is $\Pi$ and $\cP_{\cH_W}$ does not vanish identically on $\sigma \otimes
            \nu^\vee$.
    \end{enumerate}
\end{theorem}

\begin{proof}
    Let $\alpha_1, \hdots, \alpha_r$ be automorphic characters of $\bA_E^\times$ such that the $\alpha_1, \hdots, \alpha_r, \alpha_1^*, \hdots, \alpha_r^*$ are two by two distinct. Let $Q_n$ be the standard parabolic of $G_n$ with Levi factor $G_1^{2r} \times G_m$, and define for all $s \in i \R$
    \begin{equation*}
        \widetilde{\Pi}_s:= \Pi_n \boxtimes I_{Q_n}^{G_n}( \alpha_1\valP{.}^s_E \boxtimes \alpha_1^*\valP{.}^{-s}_E \boxtimes \hdots \boxtimes \alpha_r\valP{.}^s_E \boxtimes \alpha_r^*\valP{.}^{-s}_E \boxtimes \Pi_m ).
    \end{equation*}
    Note that the $\widetilde{\Pi}_s$ are $(G,H,\overline{\mu})$-regular Hermitian Arthur parameters. By elementary properties of the Rankin--Selberg $L$-function, assertion (1) of Theorem~\ref{thm:GGP_arbitrary_corank} is equivalent to the following assertion.
    \begin{enumerate}
        \item[(1')] There exists an $s \in i\R$ such that $L(\frac{1}{2},\widetilde{\Pi}_s) \neq 0$. 
    \end{enumerate}
    By the Gan--Gross--Prasad conjecture for $(G,H,\overline{\mu})$-regular Hermitian Arthur parameters from Theorem~\ref{thm:ggp_intro}, this implies (2). 
    
    Conversely, let $\sigma=\sigma_V \boxtimes \sigma_W$ be a cuspidal automorphic representation of $\cU_W$ whose weak base change to $G_n \times G_m$ is $\Pi$. Let $P$ be the parabolic subgroup of $\U(V)$ stabilizing the flag $E x_1 \subset \hdots \subset X$, with Levi factor $T_r \times \U(W)$. 
    \begin{equation}
        \label{eq:kappa_defi}
    \kappa:=\alpha_1 \boxtimes \hdots \boxtimes \alpha_r,
    \end{equation}
    which is an automorphic character of $T_r(\bA)$.
    Then for all $s \in i \R$ the representation $\widetilde{\Pi}_s$ is the weak base change of $(\U(V) \times P,  \sigma_V \boxtimes \kappa_s \boxtimes \sigma_W)$.
    
    Define $\tau=I_{B_r}^{G_r}(\kappa)$ so that $I_P^{\U(V)}(\kappa_s \boxtimes \sigma_W)=I_{P(X)}^{\U(V)} (\tau_s \boxtimes \sigma_W)$. Then by Corollary~\ref{cor:non_vanishing_of_periods}, (2) of Theorem~\ref{thm:GGP_arbitrary_corank} is equivalent to the following assertion.
    \begin{enumerate}
        \item[(2')] There exists an $s \in i \R$ such that the bilinear form $\cP_{\U'_V}$ is non-zero on $\sigma_V \otimes (I_{P(X)}^{\U(V)} \tau_s \boxtimes \sigma_W) \otimes \omega^\vee_{V}$.
    \end{enumerate}
    But (2') implies (1') by Theorem~\ref{thm:ggp_intro}. This concludes the proof of Theorem~\ref{thm:GGP_arbitrary_corank}.

\end{proof}

\subsection{The Ichino--Ikeda conjecture for Fourier--Jacobi periods}
\label{sec:proof_II_corank}

In this section, we prove the Ichino--Ikeda conjecture for Fourier--Jacobi periods in arbitrary corank. Let $\sigma$ be an automorphic cuspidal representation of $\cU_W=\U(V) \times \U(W)$ and assume that for all $v$ the local component $\sigma_v$ is tempered. We use the same normalizations and notations as in \S\ref{subsec:global_II}. We need to prove the following. 

\begin{theorem}
    \label{thm:factorization_intro}
    Let $\sigma$ and $\Pi$ be as above. For every factorizable vectors $\varphi =
\otimes_v^{'} \varphi_v \in \sigma$ and $\phi=\otimes_v^{'} \phi_v \in
\nu^\vee$ we have
    \begin{equation}
    \label{eq:IY_facto}
        |\cP_{\cH_W}(\varphi,\phi)|^2=|S_\Pi|^{-1} \cL(\frac{1}{2},\sigma) \prod_v \cP^\sharp_{\cH_W,v}(\varphi_v,\phi_v).
    \end{equation}
\end{theorem}

\begin{proof}

Write $\sigma=\sigma_V \boxtimes \sigma_W$. We keep $\tau=I_{B_r}^{G_r} \kappa$, where $\kappa$ was the character defined in \eqref{eq:kappa_defi}, which is tempered at each place. If $s \in \C$, we write $\Sigma_s=I_{P(X)}^{\U(V)} \tau_s \boxtimes \sigma_W$. 

Let $\varphi_V \in \sigma_V$, $\varphi_W \in \sigma_W$, $\varphi_\tau \in \tau$ and $\varphi_\Sigma \in \Sigma$. We also pick some Schwartz functions $\Phi \in \omega^\vee_{V}$  and $\phi \in \nu^\vee$. We assume that all these vectors are factorizable. We take a large enough set of places $\tS$ so that the unramified Ichino--Ikeda conjecture of \cite{Boi}*{Theorem~1.0.3.1} holds outside of $\tS$. 

We have the equality of $L$-functions $ L^{\tS}(1,\sigma_V,\mathrm{Ad})L^{\tS}(1,\sigma_W,\mathrm{Ad})=L^{\tS}(1,\Pi,\mathrm{As}^{n,m})$. Therefore, we want to prove that

\begin{equation}
    \label{eq:IY_facto_proof}
        |\cP_{\cH_W}(\varphi,\phi)|^2=|S_\Pi|^{-1} \Delta_{\U(V)}^{\tS}\frac{L^{\tS}(1/2,\sigma_V \times \sigma_W \otimes \overline{\mu})}{ L^{\tS}(1,\sigma_V,\mathrm{Ad})L^{\tS}(1,\sigma_W,\mathrm{Ad})} \prod_{v \in \tS} \cP_{\cH_W,v}(\varphi_v,\phi_v).
    \end{equation}

We use the same notation as in section~\ref{sec:GGP_positive_corank}, so that $\widetilde{\Pi}_s$ is the weak base change of $(\U(V) \times P(X), \sigma_V \boxtimes \tau_s \boxtimes \sigma_W )$. On the one hand, \eqref{eq:L-function_equality} reads
\begin{align}
\label{eq:facto1}
     &\valP{\cP_{\U'_V}\left( \varphi_V \otimes E_{P(X)}^{\U(V)}(\varphi_\Sigma,s) ,\Phi \right)}^2 = \valP{ \frac{L^{\tS}(\frac{1}{2}+s,\sigma_V \times \tau \otimes \overline{\mu})}{L^{\tS}(1+s,\tau^{\sfc} \times \sigma_W)L^{\tS}(1+2s,\tau,\mathrm{As}^{(-1)^{m}})}}^2  \times \\
     & \int_{(\cH_W'(F_\tS) \backslash \U'_V(F_\tS))^2} \cP_{\cH_L} \left( \mathrm{R}(h_1)(\varphi_V \otimes \varphi_{\Sigma,s} ) , (\omega^\vee_{V}(h_1) \Phi)_{Y^\vee} \right) \overline{\cP_{\cH_L} \left( \mathrm{R}(h_2)(\varphi_V \otimes \varphi_{\Sigma,s}) , (\omega^\vee_{V}(h_2) \Phi)_{Y^\vee} \right)} \rd h_i, \nonumber
\end{align}
But on the other hand, by the Ichino--Ikeda conjecture in corank-zero from Theorem~\ref{thm:II_regular_intro} and thanks to the explicit formulae for $\Delta_{\U(V)}$ and $\Delta_{\U(V),v}$ in \eqref{eq:delta_U_formula}, we have
\begin{equation}
    \label{eq:facto_proof_II}
\valP{\cP_{\U'_V}\left( \varphi_V \otimes E_{P(X)}^{\U(V)}(\varphi_\Sigma,s) ,\Phi \right)}^2 =\frac{\Delta_{\U(V)}^\tS}{|S_\Pi|} \frac{L^\tS(\frac{1}{2},\sigma_V \times \Sigma_s \otimes \overline{\mu})}{L^{\tS,*}(1,\Sigma_s,\mathrm{Ad})L^{\tS}(1,\sigma_V,\mathrm{Ad})}  \prod_{v \in \tS} \cP_{\U'_V,v}\left( \varphi_{V,v} \otimes \varphi_{\Sigma,s,v} , \Phi_v \right).
\end{equation}
By the local computation for tempered periods from Proposition~\ref{prop:tempered_computations} applied at the place in $\tS$ and thanks to \eqref{eq:normalization_measures}, we see that \eqref{eq:facto_proof_II} is
\begin{align}
\label{eq:facto2}
\frac{\Delta_{\U(V)}^\tS \Delta_{G_r}^{\tS,*}}{|S_\Pi|}&\frac{L^\tS(\frac{1}{2},\sigma_V \times \Sigma_s \otimes \overline{\mu})}{L^{\tS,*}(1,\Sigma_s,\mathrm{Ad})L^{\tS}(1,\sigma_V,\mathrm{Ad})} \times \\
&\prod_{v \in \tS}  \int_{(\cH_W'(F_v) \backslash \U'_V(F_v))^2} \cP_{\cH_{L,v}}\left((\sigma_{V,v}(h_i) \varphi^i_{V,v} \otimes \varphi^i_{\Sigma,s,v}(h_i) , (\omega^\vee_{V,v}(h_i) \Phi_{i,v})_{Y^\vee})_{i=1,2}\right) \rd h_i. \nonumber
\end{align}
From \eqref{eq:facto1}, \eqref{eq:facto2}, Proposition~\ref{prop:corank_reduction} and the equality of partial $L$-functions: 
\begin{equation*}
    \frac{L^\tS(\frac{1}{2},\sigma_V \times \Sigma_s \otimes \overline{\mu})}{L^{\tS,*}(1,\Sigma_s,\mathrm{Ad})}=\frac{L^\tS(\frac{1}{2},\sigma_V \times \sigma_W \otimes \overline{\mu})}{L^{\tS}(1,\sigma_W,\mathrm{Ad})L^{\tS,*}(1,\tau,\mathrm{Ad})}\valP{\frac{L^\tS(\frac{1}{2}+s,\sigma_V \times \tau  \otimes \overline{\mu})}{L^\tS(1+s, \tau^{c} \times  \sigma_W  )L^{\tS}(1+2s,\tau,\mathrm{As}^{(-1)^{m}})}}^2,
\end{equation*}
we see that 
\begin{align}
\label{eq:factorization_of_periods}
    \valP{\cP_{\cH_L}(\varphi_V \otimes \varphi_\tau \otimes \varphi_W  , \phi)}^2 & =\frac{\Delta_{\U(V)}^\tS \Delta_{G_r}^{\tS,*}}{|S_\Pi|}  \frac{L^\tS(\frac{1}{2},\sigma_V \times \sigma_{W}  \otimes \overline{\mu})}{L^\tS(1, \sigma_{V}, \mathrm{Ad})L^{\tS,*}(1, \tau, \mathrm{Ad})L^\tS(1, \sigma_{W}, \mathrm{Ad})} \\
    & \; \; \; \; \;  \times  \prod_{v\in \tS} \cP_{\cH_L,v}(\varphi_{V,v} \otimes  \varphi_{\tau,v} \otimes  \varphi_{W,v},  \phi_v). \nonumber
\end{align}
Recall that $\cP_{\cH_L}=\cP_{\cH_W} \otimes \cP_{N_r}$. It is know from~\cite{FLO12} and \cite{CS80} that we have
\begin{equation}
\label{eq:Whittaker_facto}
    \valP{\cP_{N_r}(\varphi_\tau)}^2=\frac{\Delta_{G_r}^{\tS,*}}{L^{\tS,*}(1,\tau,\mathrm{Ad})} \prod_{v \in \tS }\cP_{N_r,v}(\varphi_{\tau,v}),
\end{equation}
so that \eqref{eq:IY_facto_proof} follows by dividing \eqref{eq:factorization_of_periods} by \eqref{eq:Whittaker_facto}, which is non-zero for some $\varphi_\tau \in \tau$. This concludes the proof.

\end{proof}

\end{document}